\documentclass[article,3p]{elsarticle} 

\usepackage{amsmath,amsfonts,amssymb,amsbsy,amscd}
\usepackage{mathrsfs,array}
\usepackage{verbatim}
\usepackage{upgreek}
\usepackage{amsthm}
\usepackage{subcaption}
\usepackage{epstopdf}
\usepackage{booktabs} 

\usepackage{hyperref}
\hypersetup{
    colorlinks=true,
    linkcolor=blue,
    urlcolor=red,
    linktoc=all
}

\usepackage[nameinlink,capitalise,noabbrev]{cleveref}

\newtheorem{thm}{Theorem}
\newtheorem{lem}[thm]{Lemma}
\numberwithin{thm}{section} 
\newtheorem{rmk}[thm]{Remark}
\newproof{pf}{Proof}
\newproof{pot}{Proof of Theorem \ref{thm2}}
\newtheorem{prop}[thm]{Proposition}
\newtheorem{corol}[thm]{Corollary}

\newcommand{\afunc}[1]{\operatorname{\mathsf{#1}}}

\def\B#1{\mbox{\boldmath{$#1$}}}

\newcommand{\eps}{\varepsilon}
\newcommand{\pd}{\partial}

\newcommand{\nn}{\nonumber}

\newcommand{\bu}{\B{u}}

\newcommand{\bw}{\B{w}}
\newcommand{\bn}{\mathbf{n}}
\newcommand{\bv}{\B{v}}

\newcommand{\bx}{\B{x}}

\newcommand{\UU}{\mathcal{V}}

\newcommand{\WW}{\mathcal{W}}

\newcommand{\dt}{\dfrac{{\rm d}}{{\rm d}t}}

\def\be{\begin{equation}}
\def\ee{\end{equation}}
\def\ba{\begin{array}}
\def\ea{\end{array}}
\def\bea{\begin{eqnarray}}
\def\eea{\end{eqnarray}}
\def\beas{\begin{eqnarray*}}
\def\eeas{\end{eqnarray*}}
\newcommand{\bseq}{\begin{subequations}}
\newcommand{\eseq}{\end{subequations}}

\biboptions{numbers,sort&compress}

    \makeatletter
    \def\ps@pprintTitle{%
       \let\@oddhead\@empty
       \let\@evenhead\@empty
       \def\@oddfoot{\reset@font\hfil\thepage\hfil}
       \let\@evenfoot\@oddfoot
    }
    \makeatother

\crefname{lem}{Lemma}{Lemmas}
\crefname{thm}{Theorem}{Theorems}
\crefname{prop}{Proposition}{Propositions}
\crefname{corol}{Corollary}{Corollaries}
\crefname{rmk}{Remark}{Remarks}

\makeatletter
\newcommand{\myref}[1]{\cref{#1}\mynameref{#1}{\csname r@#1\endcsname}}
\newcommand{\Myref}[1]{\Cref{#1}\mynameref{#1}{\csname r@#1\endcsname}}

\def\mynameref#1#2{%
  \begingroup
    \edef\@mytxt{#2}%
    \edef\@mytst{\expandafter\@thirdoffive\@mytxt}%
    \ifx\@mytst\empty\else
    \space(\nameref{#1})\fi
  \endgroup
}
\makeatother

\usepackage{changes}
\definechangesauthor[color=blue]{Question}
\definechangesauthor[color=red]{To improve}
\definechangesauthor[color=orange]{Change?}
\definechangesauthor[color=green]{Okay?}
\definechangesauthor[color=purple]{Remove?}
\definechangesauthor[color=teal]{To add}
\definechangesauthor[color=magenta]{Complete?}

\journal{Journal of Computational Physics}

\begin{document}

\begin{frontmatter}

\title{{\large{An energy-dissipative level-set method for the incompressible two-phase Navier-Stokes equations with surface tension using functional entropy variables}}}
\author{M.F.P. ten Eikelder\corref{cor1}}
\ead{marco.ten.eikelder@gmail.com}
\cortext[cor1]{Corresponding author}
\author{I. Akkerman}
\ead{i.akkerman@tudelft.nl}

\address{Delft University of Technology, Department of Mechanical, Maritime and Materials Engineering, P.O. Box 5, 2600 AA Delft, The Netherlands}
\date{}

\begin{abstract}
This paper presents the first energy-dissipative level-set method for the incompressible Navier-Stokes equations with surface tension. The methodology relies on the recently proposed concept of functional entropy variables. Discretization in space is performed with isogeometric analysis. Temporal-integration is performed with a new perturbed midpoint scheme. The fully-discrete scheme is unconditionally energy-dissipative, pointwise divergence-free and satisfies the maximum principle for the density. Numerical examples in two and three dimensions verify the  energetic-stability of the methodology.
\end{abstract}


\end{frontmatter}

{
  \hypersetup{linkcolor=black}
  \tableofcontents
}
\section{Introduction}
\label{sec:Introduction}

This paper proposes a novel energy-dissipative numerical method for the computation of the incompressible Navier-Stokes equations with surface tension. Our method employs the level-set method to capture the fluid interface. The method uses so-called functional entropy variables and is unconditionally energy-dissipative, pointwise divergence-free and satisfies the maximum principle for the density. The energetic stability improves robustness features and as such the proposed approach is suitable choice for the simulation of immiscible fluids.

\subsection{Free-surface flow modeling}
Incompressible free-surface flows with surface tension appear in a large class of applications ranging from marine and offshore engineering, e.g. sloshing of LNG in tanks or wave impacts, to bubble dynamics. 
Applications typically involve violent free-surface flows.
As a result topological changes (e.g. break-up or coalescence) occur.
Numerical methods for two-fluid flow problems typically follow the free-surface motion with either mesh-motion or an extra variable to capture the topological changes.
The first class of methods is known as interface-tracking methods whereas the second are the interface-capturing methods.
When there is a large amount of topological changes interface-tracking methods are an unfortunate choice. 
On the other hand, interface capturing methods \cite{hughes1981lagrangian,tezduyar1992new,unverdi1992front} naturally deal with the interface and seem in this case to be the more suitable choice.

Interface capturing methods can roughly be divided into phase-field methods, volume-of-fluid methods and level-set methods, see \cite{elgeti2016deforming} for a discussion. 
The phase field models \cite{gomez2008isogeometric,liu2013functional,liu2014thermodynamically,gomez2014accurate} are known for their rigorous thermodynamical structure. 
The main issue is that numerical methods for phase field models do not provably satisfy the maximum principle for the density \cite{shokrpour2018diffuse}. 
Volume-of-fluid methods \cite{hirt1981volume,pilliod2004second,seric2018direct} are popular methods, 
also for compressible flows modeling \cite{baer1986two,kapila2001two}, but suffer from the same discrepancy. 
Monotonicity is generally only guaranteed if a CFL-like condition is fulfilled, see e.g. \cite{ten2017acoustic}. 
When simulating air-water flows the monotonicity property is crucial. 
Therefore we employ in this paper the level-set method \cite{SSO94, sethian1999level, Sethian_01,akkerman2017monotone} which by construction satisfies the maximum principle for the density. 
The level set method does not limit the complexity of the free-surface flow nor the flow regime. 
It has proven to be suitable tool for free-surface flows in marine applications, e.g. \cite{Nagrath_etal_05,ABKF11,AkBaBeFaKe12,akkerman2019toward}. 

\subsection{Surface tension}
Apart from the ability to capture the interface location, the extra variable in interface capturing methods may be used to evaluate the surface tension contribution.
In volume-of-fluid and level-set methods the interface normal and curvature may be computed similarly.
It is well-known, see e.g. \cite{abadie2015combined, popinet2018numerical}, that surface tension effects are better represented when using the level-set approach as compared with the volume-of-fluid approach.
We refer to \cite{gross2007finite} for error analysis of the surface tension force in the level-set method.
The standard and most popular approach is to use the continuum model of Brackbill et al. \cite{brackbill1992continuum}.
In the discrete approximation the evaluation of the curvature often employs a projection step for lower-order methods which leads to inaccuracies.
In a recently paper \cite{yan2019isogeometric} the authors show that the accuracy of the curvature improves significantly when using a smooth higher-order NURBS-based isogeometric discretization \cite{HuCoBa04}.

\subsection{Energetic stability}
Level-set methods are, to the best knowledge of the authors, never equipped with a thermodynamically stable algorithm.
However the notion of energetic stability\footnote{Note that thermodynamically stable resembles in the isothermal case 
energetically stable as Clausius-Duhem inequality reduces to an energy-dissipative inequality.} is of practical importance. 
In \cite{AkBaBeFaKe12} is it shown that for a viscous air-water level-set simulation in certain situations artificial energy may be created.
This leads to a nonphysical prediction of the fluid behavior.
The approach of proving an energetic stability result in a Galerkin-type formulation would be to select the appropriate weights.
Unfortunately, the suitable test functions are not available in typical finite element methods.
This applies to the spatial and temporal discretization independently.

\subsection{This work}
In this paper we address one of the main discrepancies of diffuse-interface level-set methods, namely the above mentioned absence of an energetic stability property. We circumvent the limitation caused by the function spaces by introducing the unavailable weighting function as a new variable via so-called functional entropy variables. This concept is the natural alternative to entropy variables when the mathematical entropy associated with the system of equations is a functional (instead of a function) of the conservation variables. We naturally integrate this new variable into the level-set model via the surface tension term. This creates the required extra freedom and as a result the associated weak form is equipped with energetic stability for standard divergence-conforming function spaces. The formulation does not require the evaluation of the curvature and a such is not limited to higher-order discretizations. To inherit energetic stability in a semi-discrete sense we employ a NURBS-based isogeometric analysis Galerkin-type discretization. Furthermore, we introduce an SUPG stabilization mechanisms that does not upset the energy-dissipative property of the method. Additionally, we augment the momentum equation with a residual-based discontinuity capturing term. For the temporal discretization we propose a new time-stepping scheme which can be understood as a perturbation of the midpoint rule. The result is a consistent fully-discrete energy-dissipative scheme that is pointwise divergence-free and satisfies the maximum principle for the density.

\subsection{Structure of the paper}
The remainder of this paper is organized as follows. \cref{sec: Standard continuous formulations} presents and analyzes the energy behavior of the sharp-interface incompressible Navier-Stokes equations with surface tension. In \cref{sec: diffuse} we use the sharp-interface model as a starting point to derive the diffuse level-set model and provide a detailed analysis in terms of energy behavior. Additionally, we extensively discuss the level-set form of the surface tension contribution. In \cref{sec: Energy-dissipative formulation} we employ the functional entropy variables to obtain a modified energy-dissipative formulation. Then, in \cref{sec: Energy-dissipative discretization} we present the semi-discrete energetically-stable formulation.  Next, in \cref{sec: Time-discrete formulation} we present the fully-discrete energy-dissipative method. \cref{sec: Numerical experiments} shows the numerical experiments in two and three dimensions which verify the energy-dissipative property of the scheme.

\section{Sharp-interface formulation}\label{sec: Standard continuous formulations}
\subsection{Governing equations}
Let $\Omega \subset \mathbb{R}^d$, $d=2,3$, denote the spatial domain with boundary $\partial \Omega$.
We consider two immiscible incompressible fluids that occupy subdomains $\Omega_i \subset \Omega$, $i = 1,2$, in the sense $\bar{\Omega} = \bar{\Omega}_1 \cup \bar{\Omega}_2$ and $\Omega_1 \cap \Omega_2 = \emptyset$. A time-dependent smooth interface $\Gamma = \partial \Omega_1 \cap \partial \Omega_2$ separates the fluids.
The problem under consideration consists of solving the incompressible Navier-Stokes equations with surface tension dictating the two-fluid flow:
\begin{subequations}\label{eq: strong form per domain}
\label{eq:weak}
\begin{align}
 \rho_i \left(\partial_t \bu + \bu\cdot \nabla \bu\right) - \mu_i \Delta \bu + \nabla p   =&~ \rho_i \mathbf{g}, & \quad \text{in} &  \quad \Omega_i(t) \label{eq: strong mom}\\
 \nabla \cdot \bu =&~ 0 & \quad \text{in} & \quad  \Omega_i(t), \label{eq: strong cont} \\
[\![\![ \bu ]\!]\!]  =&~ 0 & \quad \text{on} &\quad \Gamma(t), \label{eq: strong u continuous}\\
    [\![\![ \mathbf{S}(\bu,p)\boldsymbol{\nu} ]\!]\!]  =&~ \sigma \kappa \boldsymbol{\nu}  & \quad \text{on} & \quad \Gamma(t), \label{eq: strong stress surface tension}\\
    V =&~ \bu\cdot \boldsymbol{\nu} & \quad \text{on} & \quad \Gamma(t), \label{eq: strong velocity interface}
\end{align}
\end{subequations}
with $\bu(\mathbf{x},0) = \bu_0(\mathbf{x}) $ in $\Omega_i(0)$ and $\Gamma(0)= \Gamma_0$ for the fluid velocity $\bu:\Omega \rightarrow \mathbb{R}^d$ and the pressure $p:\Omega \rightarrow \mathbb{R}$. The stress tensor is given by:
\begin{align}
 \mathbf{S}(\bu,p) = \boldsymbol{\tau}(\bu)-p \mathbf{I} \quad \text{in } \Omega_i(t)
\end{align}
with viscous stress tensor:
\begin{align}
 \boldsymbol{\tau}(\bu) = 2 \mu_i \nabla^s  \bu \quad \text{in } \Omega_i(t).
\end{align}
The jump of a vector $\mathbf{v}$ is denoted as
\begin{align}
 [\![\![ \mathbf{v} ]\!]\!] = (\mathbf{v}_{|\Omega_1}-\mathbf{v}_{|\Omega_2})_{|\Gamma}.
\end{align}
The problem is augmented with appropriate boundary conditions.
We denote with $\mathbf{x}\in \Omega$ the spatial parameter and with $t \in \mathcal{T}=(0,T)$ the time with end time $T>0$. 
Furthermore, we set $\mathbf{g} = - g \boldsymbol{\jmath}$ where $g$ is the gravitational acceleration and $\boldsymbol{\jmath}$ is the vertical unit vector. The initial velocity is $\bu_0: \Omega \rightarrow \mathbb{R}^d$.
We use the standard convention for the various differential operators, i.e. the temporal derivative reads $\pd_t$ and the symmetric gradient denotes $\nabla^s\cdot=\tfrac{1}{2}\left(\nabla\cdot + \nabla^T\cdot\right)$.
The constants $\mu_i>0$ and $\rho_i>0$ denote the dynamic viscosity and density of fluid $i$ respectively.
The normal speed of $\Gamma(t)$ is denoted as $V$, the normal of $\Gamma(t)$, denoted $\boldsymbol{\nu}$, is pointing from $\Omega_2(t)$ into $\Omega_1(t)$ and the tangential vector is $\mathbf{t}$.
The curvature is $\kappa = \nabla \cdot \boldsymbol{\nu}$, i.e. $\kappa(\bx,t)$ is negative when $\Omega_1(t)$ is convex
in a neighborhood of $\bx \in \Gamma(t)$.
Furthermore, the outward-pointing normal of $\partial \Omega$ denotes $\mathbf{n}$.
We defined $u_n = \bu \cdot \mathbf{n}$ and $u_{\nu}=\bu\cdot \boldsymbol{\nu}$ as the normal velocity of $\partial \Omega$ and $\Gamma(t)$, respectively.
The equation (\ref{eq: strong mom}) represents the the balance of momentum while (\ref{eq: strong cont}) is the continuity equation.
Next, (\ref{eq: strong u continuous}) states that the velocities are continuous across the separating interface.
The fourth equation, (\ref{eq: strong stress surface tension}), stipulates that the discontinuity of the stresses at the interface is governed by surface tension.
In absence of surface tension it reduces to an equilibrium of the stresses.
Note that a direct consequence of (\ref{eq: strong stress surface tension}) is the continuity of tangential stress at the interface: 
\begin{align}\label{eq: strong tangential}
  [\![\![ 2 \mu_i (\nabla^s \bu)\boldsymbol{\nu} ]\!]\!]\cdot \mathbf{t}  = 0  \quad \text{on} & \quad \Gamma(t).
\end{align}
We assume that the surface tension coefficient $\sigma\geq 0$ is constant, i.e. Maragoni effects are precluded.
Furthermore, we assume that line force terms vanish as a result of boundary conditions or additional conditions (see also \cite{roudbari2019binary}).
We refer to  \cite{pruss2009two} for some well-posed properties of the problem.

We introduce the notation
\begin{subequations}\label{eq: rho mu}
\label{eq:weak}
\begin{align}
 \rho =&~ \rho_1 \chi_{\Omega_1(t)} + \rho_2 \chi_{\Omega_2(t)}, \label{eq: rho}\\
 \mu  =&~ \mu_1  \chi_{\Omega_1(t)} + \mu_2 \chi_{\Omega_2(t)}, \label{eq: mu}
\end{align}
\end{subequations}
with indicator $\chi_D$ of domain $D$.
System (\ref{eq: strong form per domain}) may now be written as:
\begin{subequations}\label{eq: strong form total}
\begin{align}
 \rho \left(\partial_t \bu + \bu\cdot \nabla \bu\right) - \mu \Delta \bu + \nabla p    =&~ \rho \mathbf{g} & \quad \text{in} &  \quad \Omega, \label{weak form 1 mom}\\
 \nabla \cdot \bu =&~ 0 & \quad \text{in} & \quad  \Omega, \label{weak form 1 cont} \\
[\![\![ \bu ]\!]\!]  =&~ 0 & \quad \text{on} &\quad \Gamma(t), \label{eq: jump u}\\
    [\![\![ \mathbf{S}(\bu,p)\boldsymbol{\nu} ]\!]\!]  =&~ \sigma \kappa \boldsymbol{\nu}  & \quad \text{on} & \quad \Gamma(t), \label{eq: surface tension jump}\\
    V =&~ \bu\cdot \boldsymbol{\nu} & \quad \text{on} & \quad \Gamma(t),\label{eq: interface speed V}
\end{align}
\end{subequations}
where $\boldsymbol{\tau}(\bu) \equiv 2 \mu \nabla^s  \bu$ and $\bu(\mathbf{x},0) = \bu_0(\mathbf{x}) $ in $\Omega_i(0)$ and $\Gamma(0)= \Gamma_0$ .

As we aim to develop an energy-dissipative level-set method, we first study the energy behavior of the sharp-interface model associated with system (\ref{eq: strong form total}). This is the purpose of the remainder of \cref{sec: Standard continuous formulations}. After the energy analysis in \cref{subsec: energy diss} we present a standard weak formulation of (\ref{eq: strong form total}) in \cref{subsec: standard weak}.

\subsection{Energy evolution}\label{subsec: energy diss}
We consider the dissipation of the energy of the problem (\ref{eq: strong form total}).
The total energy consists of three contributions, namely kinetic ($K$), gravitational ($G$) and surface energy ($S$):
\begin{subequations}\label{eq: total energy}
\begin{align}
\mathscr{E}(\bu)           =&~\mathscr{E}^{\text{K}}(\bu) + \mathscr{E}^{\text{G}} + \mathscr{E}^{\text{S}}, \\
  \mathscr{E}^{\text{K}}(\bu) :=&~\displaystyle\int_{\Omega} \tfrac{1}{2}\rho \|\bu\|_2^2 ~{\rm d}\Omega,\\
 \mathscr{E}^{\text{G}} :=&~ \displaystyle\int_{\Omega} \rho g y ~ {\rm d}\Omega,\\
  \mathscr{E}^{\text{S}} :=&~\displaystyle\int_{\Gamma(t)} \sigma~ {\rm d}\Gamma,
\end{align}
\end{subequations}
with $y =  \bx\cdot \boldsymbol{\jmath}$ the vertical coordinate.
\begin{thm}\label{theorem: energy dissipation strong original}
 Let $\bu$ and $p$ be smooth solutions of the incompressible Navier-Stokes equations with surface tension (\ref{eq: strong form total}) The total energy $\mathscr{E}$, given in (\ref{eq: total energy}), satisfies the dissipation inequality:
 \begin{align}\label{eq: thm2.1}
   \frac{{\rm d}}{{\rm d}t} \mathscr{E}(\bu)=- \displaystyle\int_{\Omega} \boldsymbol{\tau}(\bu):\nabla \bu ~{\rm d}\Omega + {\rm bnd} \leq 0  + {\rm bnd},
 \end{align}
 where ${\rm bnd}$ serves as a proxy for the boundary contributions.
\end{thm}
\begin{proof}
To establish the dissipative property (\ref{eq: thm2.1}) we will first consider the evolution of each of the energy contributions (\ref{eq: total energy}) separately and subsequently substitute these in the strong form (\ref{eq: strong form total}).

We start off with the kinetic energy evolution. The following sequence of identities holds:
\begin{align}\label{eq: identify ekin}
  \dt\mathscr{E}^K
    =&~ \displaystyle\int_{\Omega_1(t)} \rho \bu\cdot \partial_t \bu ~{\rm d}\Omega + \displaystyle\int_{\Omega_2(t)} \rho \bu\cdot \partial_t \bu~{\rm d}\Omega\nn\\
  &~+ \displaystyle\int_{\partial \Omega_1(t) \cap \Gamma(t)} \tfrac{1}{2}\rho\|\bu\|^2 \bu\cdot\bn_1  ~{\rm d S}+ \displaystyle\int_{\partial \Omega_2(t) \cap \Gamma(t)} \tfrac{1}{2}\rho\|\bu\|^2 \bu\cdot\bn_2  ~{\rm d S}\nn\\
 =&~ \displaystyle\int_{\Omega} \rho \bu\cdot \left(\partial_t \bu+ \left(\bu\cdot \nabla\right)\bu \right)~{\rm d}\Omega +  \displaystyle\int_\Omega  \tfrac{1}{2}\rho\|\bu\|^2 \nabla \cdot \bu ~{\rm d}\Omega\nn\\
 &~- \displaystyle\int_{\partial\Omega} \tfrac{1}{2}\rho \|\bu\|^2 u_n ~{\rm d S}.
\end{align}
where $\bn_1$ and $\bn_2$ denote the outward unit normal of $\Omega_1(t)$ and $\Omega_2(t)$, respectively. 
The first identity results from the Leibniz-Reynolds transport theorem.
To obtain the second equality one adds a suitable partition of zero, subsequently applies the divergence theorem on both $\Omega_1(t)$ and $\Omega_1(t)$, and lastly uses the chain rule.

In a similar fashion we have the identities for the gravitational energy evolution:
\begin{align}\label{eq: identify egrav}
  \dt\mathscr{E}^G
  =&~ \displaystyle\int_{\partial \Omega_1(t) \cap \Gamma(t)} \rho g y \bu\cdot\bn_1  ~{\rm d S}+ \displaystyle\int_{\partial \Omega_2(t) \cap \Gamma(t)} \rho g y \bu\cdot\bn_2  ~{\rm d S}\nn\\
  =&~ \displaystyle\int_{\Omega_1(t)} \rho g \boldsymbol{\jmath}\cdot \bu  ~{\rm d}\Omega+ \displaystyle\int_{\Omega_1(t)} \rho g y \nabla \cdot \bu  ~{\rm d}\Omega\nn\\
   &~ +\displaystyle\int_{\Omega_2(t)} \rho g \boldsymbol{\jmath}\cdot \bu  ~{\rm d}\Omega+ \displaystyle\int_{\Omega_2(t)} \rho g y \nabla \cdot \bu  ~{\rm d}\Omega -\displaystyle\int_{\partial \Omega} \rho g y \bu \cdot \bn ~{\rm d S} \nn\\
  =&~ \displaystyle\int_{\Omega} \rho g \bu\cdot  \boldsymbol{\jmath} ~{\rm d}\Omega+ \displaystyle\int_{\Omega} \rho g y \nabla \cdot \bu  ~{\rm d}\Omega-\displaystyle\int_{\partial \Omega} \rho g y u_n ~{\rm d S}.
\end{align}
The first identity emanates from the Leibniz-Reynolds transport theorem and the second is a direct consequence of the divergence theorem.

Finally, we consider the energetic contribution due to surface tension. We have from the Reynolds transport theorem in tangential calculus, see e.g. \cite{sokolowski1992introduction}, the identity:
\begin{align}\label{eq: continuous energy surf}
 \dt \mathscr{E}^{\text{S}} =  \displaystyle\int_{\Gamma(t)} \sigma \kappa u_\nu~{\rm d}\Gamma - \displaystyle\int_{\partial \Gamma(t)} \sigma \bu\cdot \boldsymbol{\mathbf{\nu}}_\partial ~{\rm d}(\partial\Gamma),
\end{align}
where we recall that we do not account for Maragoni forces ($\sigma$ is constant). Here $\boldsymbol{\mathbf{\nu}}_\partial$ is the unit-normal vector to $\partial \Gamma(t)$, tangent to $\Gamma(t)$. We refer to \cite{stone1990simple,buscaglia2011variational} for alternative insightful derivations of (\ref{eq: continuous energy surf}). We discard the last member of the right-hand side of (\ref{eq: continuous energy surf}) as it represents a line force.

We multiply the momentum equation by $\bu$ and subsequently integrate over the domain:
\begin{align}\label{eq: integrate domain 0}
  \displaystyle\int_{\Omega} \bu^T \rho \left(\partial_t \bu + \bu\cdot \nabla \bu\right)~{\rm d}\Omega + \displaystyle \int_{\Omega} \bu^T \left(\nabla p - \mu \Delta \bu\right)~{\rm d}\Omega + \displaystyle \int_{\Omega} \rho g \bu \cdot \boldsymbol{\jmath} ~{\rm d}\Omega =&~ 0.
\end{align}
Considering the second expression in (\ref{eq: integrate domain 0}) in isolation we have the two identities:
\begin{align}\label{eq: integrate domain}
  \displaystyle\int_{\Omega} \bu^T \left( \nabla p - \mu \Delta \bu  \right) ~{\rm d}\Omega =&~     \displaystyle\int_{\Omega_1(t)} \bu^T \nabla \left(  p \mathbf{I}- \mu_1 \nabla^s \bu  \right) ~{\rm d}\Omega + \displaystyle\int_{\Omega_2(t)} \bu^T \nabla \left(  p \mathbf{I}- \mu_2 \nabla^s \bu  \right) ~{\rm d}\Omega
    \nn\\
   &~ + \displaystyle\int_{\Omega_1(t)} \mu_1 \bu \cdot \nabla (\nabla \cdot \bu) ~{\rm d}\Omega + \displaystyle\int_{\Omega_2(t)} \mu_2 \bu \cdot \nabla (\nabla \cdot \bu) ~{\rm d}\Omega \nn\\
   =&~     \displaystyle\int_{\Omega} \nabla\bu:  \mathbf{S}(\bu,p) ~{\rm d}\Omega   - \displaystyle\int_{\partial\Omega}    \mathbf{n}^T  \mathbf{S}(\bu,p)  \bu ~{\rm d}\Omega\nn\\
   &~+ \displaystyle\int_{\Omega} \mu \bu \cdot \nabla (\nabla \cdot \bu) ~{\rm d}\Omega+ \displaystyle\int_{\Gamma(t)} \sigma \kappa  u_\nu ~{\rm d}\Gamma.
\end{align}
The first identity follows from adding a suitable partition of zero.
For the second equality we perform integration by parts and make use of the jump (\ref{eq: surface tension jump}) where we note that on $\Gamma(t)$ we have $\bn_1 = -\boldsymbol{\nu}$ and $\bn_2=\boldsymbol{\nu}$.

We deduce from the continuity equation:
\begin{align}\label{eq: cont with weight}
 -\int_{\Omega} (p  + \tfrac{1}{2}\rho\|\bu\|_2^2+\rho g y)\nabla \cdot \bu ~ {\rm d}\Omega + \displaystyle\int_{\Omega} \mu \bu\cdot\nabla (\nabla \cdot \bu)  ~{\rm d}\Omega= 0.
\end{align}
Next, we collect the identities (\ref{eq: identify ekin}), (\ref{eq: identify egrav}), (\ref{eq: continuous energy surf}), (\ref{eq: cont with weight}) and (\ref{eq: integrate domain}), substitute these into (\ref{eq: integrate domain 0}). The first member in (\ref{eq: integrate domain 0}) cancels with the first term in the ultimate expression in (\ref{eq: identify ekin}). By virtue of (\ref{eq: integrate domain}) the second term in (\ref{eq: integrate domain 0}) drops out. The third member of (\ref{eq: integrate domain 0}) disappears due to (\ref{eq: identify egrav}). Some of the other terms in (\ref{eq: identify ekin}), (\ref{eq: identify egrav}) and (\ref{eq: integrate domain}) vanish due to (\ref{eq: continuous energy surf}) and (\ref{eq: cont with weight}). Gathering the expressions we eventually arrive at:
\begin{align}
 \dt\mathscr{E} =&~  - \displaystyle\int_{\Omega}  \boldsymbol{\tau}(\bu): \nabla  \bu ~{\rm d}\Omega + \displaystyle\int_{\partial \Omega} \bn^T \left(\mathbf{S}(\bu,p)-\left(\tfrac{1}{2}\rho\|\bu\|^2+ \rho g y \right)\mathbf{I}\right) \bu  ~{\rm d S}.
\end{align}
This completes the proof with
\begin{align}\label{eq: bnd}
 {\rm bnd} =  \displaystyle\int_{\partial \Omega} \bn^T \left(\mathbf{S}(\bu,p)-\left(\tfrac{1}{2}\rho\|\bu\|^2+ \rho g y \right)\mathbf{I}\right) \bu  ~{\rm d S} .
\end{align}

\end{proof}

\subsection{Standard weak formulation}\label{subsec: standard weak}
Recall that we suppress line force contributions as a result of boundary or auxiliary conditions. At this point we also assume homogeneous boundary conditions to increase readability of the remainder of the paper. Results can be easily extended to non-homogeneous boundary conditions.
We define $(\cdot, \cdot)_\Omega$ as the $L^2(\Omega)$ inner product on the interior and $(\cdot , \cdot)_{\Gamma}$ as the $L^2(\Gamma)$ inner product on the boundary.
We take zero-average pressures for all $t \in \mathcal{T}$. The space-time velocity-pressure function-space satisfying homogeneous boundary condition $\bu=\mathbf{0}$ denotes $\UU_T$ and the corresponding weighting function space denotes $\UU$. 
The standard conservative weak formulation corresponding to strong form (\ref{eq: strong form total}) reads:\\

\textit{Find $\left\{\bu, p \right\} \in \UU $ such that for all $\left\{\bw, q \right\} \in \UU$:}
\begin{subequations}\label{eq: weak form standard 0}
\label{eq:weak}
\begin{align}
(\bw,  \rho \left( \partial_t \bu + \bu\cdot \nabla \bu\right))_\Omega - (\nabla \cdot \bw, p)_\Omega + (\nabla  \bw, \boldsymbol{\tau}(\bu))_\Omega + \left(\bw,\sigma \kappa \boldsymbol{\nu}\right)_{\Gamma(t)}
=&~ (\bw, \rho \mathbf{g})_\Omega, \label{weak form 1 mom}\\
 ( q, \nabla \cdot \bu)_\Omega =&~ 0, \label{weak form 1 cont}
\end{align}
\end{subequations}
with interface speed $V = \bu\cdot \boldsymbol{\nu}$. 
The weak formulation (\ref{eq: weak form standard 0}) is equivalent to the strong form (\ref{eq: strong form total}) for smooth solutions and the associated energy evolution relation coincides with that of the strong form (\ref{eq: strong form total}).
\begin{rmk}
   To show the energy evolution for the case of non-homogeneous boundary conditions one may enforce boundary conditions with a Lagrange multiplier construction \cite{HuEnMaLa00, HuWe05, EiAk17i, EiAk17ii} and subsequently use (\ref{eq: continuous energy surf}) to identify the surface energy contribution. 
\end{rmk}
\begin{rmk}\label{rmk: alternative form}
  In order to avoid evaluating second-derivatives the alternative form $+(\nabla \bw,\sigma \mathbf{P}_T)_\Gamma$ for the surface tension term in (\ref{eq: weak form standard 0}) with tangential projection $\mathbf{P}_T = \mathbf{I} - \boldsymbol{\nu}\otimes\boldsymbol{\nu}$ may be used. In \ref{sec: Surface tension evaluation} we provide the derivation of this alternative form.
\end{rmk}

\section{Diffuse-interface level-set model}\label{sec: diffuse}
In this Section we present the diffuse-interface level-set model and analyze its energy behavior.
To do so, in \cref{subsec: LS} we provide the level-set formulation of (\ref{eq: strong form total}) which we subsequently present in non-dimensional form \cref{subsec: non-dim}. Then in \cref{subsec: reg} we regularize the sharp-interface level-set formulation to obtain the diffuse-interface model. We conclude with a detailed study of the energy behavior of this level-set formulation in \cref{subsec: energy evo}.

\subsection{Sharp-interface level-set formulation}\label{subsec: LS}
We employ the interface capturing level-set method to reformulate model (\ref{eq: weak form standard 0}).
To this purpose we introduce the level-set function $\phi: \Omega(t) \rightarrow \mathbb{R}$ to describe the evolution of the interface $\Gamma(t)$.
The sub-domains and interface are identified as:
\begin{subequations}\label{eq: LS}
\begin{align}
 \Omega_1(t) \equiv &~ \left\{ \mathbf{x} \in \Omega(t) | \phi (\mathbf{x},t) > 0 \right\}, \\
 \Omega_2(t) \equiv &~ \left\{ \mathbf{x} \in \Omega(t) | \phi (\mathbf{x},t) < 0 \right\}, \\
 \Gamma(t)   \equiv &~ \left\{ \mathbf{x} \in \Omega(t) | \phi (\mathbf{x},t) = 0 \right\}. 
\end{align}
\end{subequations}
The motion of the interface $\Gamma(t)$ is governed by pure convection: 
\begin{align}
 \partial_t \phi + \bu \cdot \nabla \phi =  \partial_t \phi + V \|\nabla \phi\|  = 0.
\end{align}
This results from taking the temporal derivative of the zero-level set.
The domain indicator may now be written as: 
\begin{subequations}\label{eq: domain indicator Heaviside}
\begin{align}
  \chi_{\Omega_1} =&~ H(\phi), \\
  \chi_{\Omega_2} =&~ 1-H(\phi), 
\end{align}
\end{subequations}
where $H$ is the Heaviside function with the half-maximum convention:
\begin{align}\label{eq: Heaviside sharp}
 H(\phi) = \left\{ \begin{matrix} 0 & \phi <0 \\
                                  \frac{1}{2} & \phi =0 \\
                                  1 & \phi>0.
                   \end{matrix}
\right.
\end{align}
The resulting density and fluid viscosity are:
\begin{subequations}\label{eq: rho mu LS}
\label{eq:weak}
\begin{align}
 \rho(\phi) =&~ \rho_1 H(\phi) + \rho_2 (1-H(\phi)), \label{eq: rho H}\\
 \mu(\phi)  =&~ \mu_1  H(\phi) + \mu_2 (1-H(\phi)), \label{eq: mu H}
\end{align}
\end{subequations}
and the viscous stress now depends on $\bu$ and $\phi$:
\begin{align}
  \boldsymbol{\tau}(\bu, \phi) = 2 \mu(\phi) \nabla^s \bu. 
\end{align}
In order to write the surface term in (\ref{eq: weak form standard 0}) in the level-set context we need expressions for the surface normal, the curvature and require to convert the surface integral into a domain integral. This is how we proceed. We first define the regularized $2$-norm $\|\cdot \|_{\epsilon,2} : \mathbb{R} \rightarrow \mathbb{R}_+$ for dimensionless $\mathbf{b}\in\mathbb{R}^d$ and $\epsilon\geq 0$ as:
\begin{align}
  \|\mathbf{b} \|_{\epsilon,2}^2 = \mathbf{b}\cdot \mathbf{b} + \epsilon^2.
\end{align}
The surface normal is now continuously extended into the domain via
\begin{align}\label{eq: surface normal domain}
  \hat{\boldsymbol{\nu}}(\phi) := \dfrac{\nabla \phi}{\|\nabla \phi\|_{\epsilon,2}}.
\end{align}
The curvature results from taking the divergence of (\ref{eq: surface normal domain}):
\begin{align}
 \hat{\kappa}(\phi) \equiv \nabla \cdot \hat{\boldsymbol{\nu}}  = \nabla \cdot \left(\frac{\nabla \phi}{\|\nabla \phi\|_{\epsilon,2}}\right).
\end{align}
We may now convert the surface integral into
\begin{align}\label{eq: surf int}
    \displaystyle\int_{\Gamma(t)} \sigma \bw\cdot \boldsymbol{\nu}~  \kappa  ~{\rm d} \Gamma =&~  \displaystyle\int_\Omega \sigma \bw\cdot \hat{\boldsymbol{\nu}}(\phi)~  \hat{\kappa}(\phi)  \delta_{\Gamma}(\phi) ~{\rm d} \Omega.
\end{align}
Here $\delta_{\Gamma}=\delta_{\Gamma}(\phi)$ denotes the Dirac delta concentrated on the interface $\Gamma(t)$:
\begin{align}\label{eq: delta gamma}
 \delta_{\Gamma}(\phi) = \delta(\phi) \|\nabla \phi \|_{\epsilon,2}.
\end{align}
which extends the integral over boundary $\Gamma(t)$ to the domain $\Omega$ \cite{osher2001level}. In (\ref{eq: delta gamma}) $\delta(\phi)$ represents the Dirac delta distribution.
The expression in (\ref{eq: surf int}) is exact for $\epsilon=0$ and an approximation otherwise. We refer to Chang et al. \cite{chang1996level} for an insightful derivation. For more rigorous details the reader may consult \cite{hormander2015analysis}. Note that we have suppressed $\epsilon$ in (\ref{eq: surface normal domain})-(\ref{eq: delta gamma}).
The corresponding strong form writes in terms of the variables $\bu, p$ and $\phi$ as:
\begin{subequations}\label{eq: strong form phi}
\begin{align}
  \rho(\phi)  \left( \partial_t \bu+\bu\cdot\nabla\bu\right)  - \nabla \cdot \boldsymbol{\tau}(\bu, \phi) + \nabla p 
  + \sigma \delta_{\Gamma}(\phi) \hat{\kappa}(\phi) \hat{\boldsymbol{\nu}}(\phi) -\rho(\phi) \mathbf{g}&=~0 , \\
 \nabla \cdot \bu &=~ 0, \\
 \partial_t \phi + \bu \cdot \nabla \phi &=~0,
\end{align}
\end{subequations}
with $\bu(0) = \bu_0$ and $\phi(0)=\phi_0$ in $\Omega$. From this point onward we skip the hat symbols for simplicity.

\subsection{Non-dimensionalization}\label{subsec: non-dim}
  We now perform the non-dimensionalization of the incompressible Navier-Stokes equations with surface tension.
  Here we re-scale the system (\ref{eq: strong form total}) based on physical variables.
  The dimensionless variables are given by:
  \begin{align}
    \mathbf{x}^* =&~ \frac{\mathbf{x}}{L_0}, \quad \bu^* = \frac{\bu}{U_0}, \quad t^* = \frac{t U_0}{L_0}, \quad \rho^* = \frac{\rho}{\rho_1},  \quad \mu^* = \frac{\mu}{\mu_1}, \quad \phi^* =~ \frac{\phi}{L_0}, \quad p^* = \frac{p}{\rho_1 U_0^2},
  \end{align}
where $L_0$ is a characteristic length scale and $U_0$ is a characteristic velocity. A direct consequence is
  \begin{subequations}
\begin{align}
    \kappa^*(\phi^*) :=&~ \nabla^* \cdot\left(\dfrac{\nabla^* \phi^*}{\|\nabla^* \phi^*\|_{\epsilon,2}}\right)= L_0\kappa(\phi) , \\ \delta_\Gamma^*(\phi^*):=&~ \delta(\phi^*)\|\nabla^* \phi^*\|_{\epsilon,2} = L_0\delta_\Gamma(\phi) ,
  \end{align}
\end{subequations}
where we have used the scaling property of the Dirac delta:
\begin{align}
  \delta(\upalpha \phi) = \dfrac{1}{|\upalpha|}\delta(\phi), \quad \upalpha \neq 0.
\end{align}
The dimensionless system reads:
\begin{subequations}\label{eq: non-dimensional}
\begin{align}
     \rho^*(\phi^*) \left( \partial_{t^*} \bu^* +  \bu^* \cdot \nabla^*  \bu^*\right)  - \nabla^* \cdot \boldsymbol{\tau}^*(\bu^*, \phi^*) + \nabla^* p^* {\color{white}dawd}\nn\\
     + \frac{1}{\mathbb{W}{\rm e}} \delta^*_{\Gamma}(\phi^*) \kappa^*(\phi^*) \boldsymbol{\nu^*}(\phi^*) +\frac{1}{\mathbb{F}{\rm r}^2}\rho^*(\phi^*) \boldsymbol{\jmath}=~ 0, \\
 \nabla^* \cdot \bu^* =~ 0, \\
 \partial_{t^*} \phi^* + \bu^* \cdot \nabla^* \phi^* =~0,
\end{align}
\end{subequations}
where dimensionless viscous stress is given by:
\begin{align}
  \boldsymbol{\tau}^*=\boldsymbol{\tau}^*(\bu^*, \phi^*) =\frac{1}{\mathbb{R}{\rm e}}\mu^*(\phi^*) \left( \nabla^*\bu^* + {\nabla^*}^T\bu^*\right).
\end{align}
The used dimensionless coefficients are the Reynolds number (Re) which expresses relative strength of inertial forces and viscous forces, the Weber number (We) measuring the ratio of inertia to surface tension and the Froude number (Fr) which quantifies inertia with respect to gravity. The expressions are given by:
\begin{subequations}\label{eq: dimensionless quantities}
\begin{align}
     \mathbb{R}{\rm e} =&~ \frac{\rho_1 U_0 L_0}{\mu_1},\\
     \mathbb{W}{\rm e} =&~ \frac{\rho_1 U_0^2 L_0}{\sigma},\\
     \mathbb{F}{\rm r} =&~ \frac{U_0}{\sqrt{g L_0}}.     
\end{align}
\end{subequations}
\begin{rmk}
Other related dimensionless numbers are the Ohnesorge number $\mathbb{O}{\rm h} = \mathbb{W}{\rm e}^{1/2}/\mathbb{R}{\rm e}$, the capillarity number $\mathbb{C}{\rm a} = \mathbb{W}{\rm e}/\mathbb{R}{\rm e}$ and the E\"{o}tv\"{o}s number $\mathbb{E}{\rm o}=\mathbb{W}{\rm e}/\mathbb{F}{\rm r}^2$.
\end{rmk}
We supress the star symbols in the remainder of this paper.

\subsection{Regularization}\label{subsec: reg}
In the following we smear the interface over an interface-width of $\eps>0$ via replacing the (sharp) Heaviside function (\ref{eq: Heaviside sharp}) by a regularized differentiable Heaviside $H_\eps(\phi)$. We postpone the specific form of $H_\eps(\phi)$ to \cref{sec: Time-discrete formulation}. The regularized delta function is $\delta_{\Gamma,\eps}(\phi) = \delta_\eps(\phi) \|\nabla \phi\|_{\epsilon,2}$ with one-dimensional continuous regularized delta function $\delta_\eps(\phi) = H_\eps'(\phi)$. We refer to \cite{kublik2016integration} for details concerning the approximation of the Dirac delta.
The density and fluid viscosity are computed as
\begin{subequations}\label{eq: rho mu LS approximate 0}
\label{eq:weak}
\begin{align}
 \rho_\eps \equiv \rho_\eps(\phi) :=&~ \rho_1 H_\eps(\phi) + \rho_2 (1-H_\eps(\phi)), \label{eq: rho H eps}\\
 \mu_\eps  \equiv \mu_\eps(\phi)  :=&~ \mu_1  H_\eps(\phi) + \mu_2 (1-H_\eps(\phi)). \label{eq: mu H eps}
\end{align}
\end{subequations}

Our procedure to arrive at an energy-dissipative formulation, presented in \cref{sec: Energy-dissipative formulation}, requires a conservative model. Using the continuity and level-set equation, the associated conservative model follows straightforwardly:
\begin{subequations}\label{eq: strong form phi conservative}
\begin{align}
     \partial_t (\rho_\eps(\phi)\bu) +  \nabla \cdot \left( \rho_\eps(\phi) \bu\otimes \bu\right)  - \nabla \cdot \boldsymbol{\tau}_\eps(\bu, \phi) + \nabla p + \frac{1}{\mathbb{W}{\rm e}} \delta_{\Gamma,\eps}(\phi) \kappa (\phi)\boldsymbol{\nu}(\phi) +\frac{1}{\mathbb{F}{\rm r}^2}\rho_\eps(\phi) \boldsymbol{\jmath}&=~ 0, \label{eq: strong form phi conservative mom} \\
 \nabla \cdot \bu &=~ 0, \label{eq: strong form phi conservative cont} \\
 \partial_t \phi + \bu \cdot \nabla \phi &=~0,\label{eq: strong form phi conservative LS}
\end{align}
\end{subequations}
where $\boldsymbol{\tau}_\eps(\bu, \phi) = 2 \mu_\eps(\phi) \nabla^s \bu$ and $\bu(0) = \bu_0$ and $\phi(0)=\phi_0$ in $\Omega$. At this point we have assumed a constant interface width $\eps$. In the following we omit the $\eps$ for the sake of notational simplicity.
\begin{rmk}
 In case of a non-constant $\epsilon$ one requires to augment the right-hand side of (\ref{eq: strong form phi conservative mom}) with $\partial \rho_\eps/\partial \eps \left(\partial_t \eps + \bu \cdot \nabla \eps\right)$.
\end{rmk}
\begin{rmk}
  At this point we remark that as an alternative one may also employ a skew-symmetric form for the convective terms.
  Via a partial integration step,
\begin{align}\label{eq: skew}
 (\bw,  \nabla \cdot (\rho \bu \otimes \bu))_\Omega =\tfrac{1}{2}(\bw, \rho \bu \cdot \nabla \bu)_\Omega -\tfrac{1}{2}(\nabla \bw, \rho \bu \otimes \bu)_\Omega+\tfrac{1}{2}(\bw,\bu\bu\cdot \nabla \rho)_\Omega+\tfrac{1}{2}(\bw,\rho \bu \nabla \cdot\bu)_\Omega,
\end{align}
we may replace the convective term in (\ref{eq: strong form phi conservative}) by the first three terms on the right-hand side of (\ref{eq: skew}).
In the current situation the specific form of the convective terms (conservative or skew-symmetric) is not essential. This changes when the formulation is equipped with multiscale stabilization terms. In the single-fluid case (in absence of surface tension) the well-known multiscale discretization that represents an energy-stable system is the skew-symmetric form, see e.g. \cite{EiAk17i, EiAk17ii, evans2020variational}. In contrast to the current two-phase model, this property is for the single-fluid case directly inherited by the fully-discrete case when employing the mid-point rule for time integration. 
\end{rmk}

\subsection{Energy evolution}\label{subsec: energy evo}
In the following we show the energy balance of the level-set formulation (\ref{eq: strong form phi conservative}). 
The kinetic, gravitational and surface energy associated with system (\ref{eq: strong form phi conservative}) are:
\begin{subequations}
\begin{align}
 \mathscr{E}^{\text{K}}(\bu, \phi) :=&~\left( \tfrac{1}{2}\rho(\phi) \bu, \bu \right)_\Omega,\\
 \mathscr{E}^{\text{G}}(\phi) :=&~ \frac{1}{\mathbb{F}{\rm r}^2}\left( \rho(\phi), y\right)_\Omega,\\
 \mathscr{E}^{\text{S}}(\phi) :=&~\frac{1}{\mathbb{W}{\rm e}}\left( 1,  \delta_{\Gamma}(\phi)\right)_\Omega.
\end{align}
\end{subequations}
The total energy is the superposition of the separate energies:
\begin{align}\label{eq: total energy LS}
 \mathscr{E}(\bu, \phi) :=&~\mathscr{E}^{\text{K}}(\bu, \phi)+ \mathscr{E}^{\text{G}}(\phi) + \mathscr{E}^{\text{S}}(\phi).
\end{align}
The local energy is given by:
  \begin{align}\label{eq: loc energy}
 \mathcal{H} = \tfrac{1}{2} \rho(\phi) \|\bu \|_2^2 + \frac{1}{\mathbb{F}{\rm r}^2} \rho(\phi) y + \frac{1}{\mathbb{W}{\rm e}} \delta_\Gamma(\phi).
   \end{align}
We present the local energy balance and subsequently the global balance.
To that purpose we first need to introduce some notation and Lemmas associated with the surface energy.
Let us define the normal projection operator:
\begin{align}
 \mathbf{P}_N(\phi) = \frac{\nabla \phi}{\|\nabla \phi \|_{\epsilon,2}} \otimes \frac{\nabla \phi}{\|\nabla \phi \|_{\epsilon,2}}.
\end{align}
and the tangential projection operator:
\begin{align}
 \mathbf{P}_T(\phi) =&~ \mathbf{I} - \mathbf{P}_N(\phi).
\end{align}
The associated gradient operators are the gradient along the direction normal to the interface:
\begin{align}
 \nabla_N =&~ \mathbf{P}_N(\phi) \nabla.
\end{align}
and the gradient tangent to the interface:
\begin{align}
  \nabla_\Gamma =&~ \mathbf{P}_T(\phi) \nabla = \nabla - \nabla_N.
\end{align}
\begin{lem}\label{lem: norm grad phi}
 The term $\|\nabla \phi \|_{\epsilon,2}$ evolves in time according to:
 \begin{align}\label{eq: VE 1}
   \partial_t \|\nabla \phi \|_{\epsilon,2} + \nabla \cdot \left( \|\nabla \phi \|_{\epsilon,2}\bu\right) - \|\nabla \phi \|_{\epsilon,2} \nabla_\Gamma \bu  = 0.
 \end{align}
\end{lem}
\begin{proof}
This follows when evaluating the normal derivative of the level-set equation. Taking the gradient of the level-set equation and subsequently evaluating the inner product of the result with $\boldsymbol{\nu}(\phi)$ yields:
\begin{align}
 \dfrac{\nabla \phi}{\|\nabla \phi\|_{\epsilon,2}}\cdot \nabla \left( \partial_t \phi + \bu \cdot \nabla \phi \right) = 0.
\end{align}
Applying the gradient operator to each of the members provides
\begin{align}\label{eq: apply grad}
 \dfrac{\nabla \phi}{\|\nabla \phi\|_{\epsilon,2}}\cdot \nabla \left( \partial_t \phi \right)
 +\bu\cdot \left(\nabla \left(  \nabla \phi \right)\dfrac{\nabla \phi}{\|\nabla \phi\|_{\epsilon,2}}\right)    + \nabla  \bu:\left(\dfrac{\nabla \phi}{\|\nabla \phi\|_{\epsilon,2}}\otimes\nabla \phi\right) = 0.
\end{align}
The first term in (\ref{eq: apply grad}) coincides with the first member in expression (\ref{eq: VE 1}). For the second term in (\ref{eq: apply grad}) we note that the term in brackets equals the gradient of $\|\nabla \phi \|_{\epsilon,2}$. Finally, one recognizes the normal projection operator in the latter term of (\ref{eq: apply grad}). This delivers:
\begin{align}
   \partial_t \|\nabla \phi \|_{\epsilon,2} + \bu\cdot \nabla  \|\nabla \phi \|_{\epsilon,2} + \|\nabla \phi \|_{\epsilon,2} \nabla_N \bu  = 0.
\end{align}
Adding a suitable partition of zero completes the proof.
 \end{proof}
\begin{rmk}
  The evolution (\ref{eq: VE 1}) may be linked to the recently proposed variation entropy theory \cite{ten2019variation}.
  Variation entropy is local continuous generalization of the celebrated TVD (total variation diminishing) property derived from entropy principles.
  It serves as a derivation of discontinuity capturing mechanisms \cite{ten2020theoretical}.
  Using the continuity equation (\ref{eq: strong form phi conservative cont}) we obtain an alternative form of (\ref{eq: VE 1}):
    \begin{align}\label{eq: alternative form}
    \partial_t \eta(\nabla \phi) + \nabla \cdot \left( \eta(\nabla \phi) \dfrac{\partial \mathbf{f}}{ \partial \phi}\right) + \eta(\nabla \phi) \nabla_N \dfrac{\partial \mathbf{f}}{ \partial \phi}  =0,
  \end{align}
  with $\eta(\nabla \phi) = \|\nabla \phi \|_{\epsilon,2}$ and $\mathbf{f}(\phi, \bu) = \mathbf{u}\phi$. In the stationary case, i.e. when the term $\nabla_N (\partial \mathbf{f}/ \partial \phi)$ is absent, relation (\ref{eq: alternative form}) represents the evolution of variation entropy $\eta(\nabla \phi)$. This occurs when the velocity normal to the interface is constant.
\end{rmk}
 \begin{lem}\label{lem: surface Dirac}
 The surface Dirac $\delta_\Gamma(\phi)$ evolves in time according to:
 \begin{align}\label{eq: surface Dirac}
   \partial_t \delta_\Gamma(\phi) + \nabla \cdot \left( \delta_\Gamma(\phi) \bu \right)  - \delta_\Gamma(\phi) \nabla_\Gamma \bu  = 0.
 \end{align}
\end{lem}
\begin{proof}
Multiplying the level-set equation by $\delta'(\phi)$ provides:
 \begin{align}\label{eq: dirac 1D}
  \partial_t \delta(\phi) + \mathbf{u}\cdot \nabla \delta(\phi) =0.
 \end{align}
The superposition of (\ref{eq: VE 1}) multiplied by $\delta(\phi)$ and (\ref{eq: dirac 1D}) multiplied by $\|\nabla \phi\|_{\epsilon,2}$ provides the result.
In other words, the operator 
\begin{align}\label{eq: operator}
  \delta(\phi) \dfrac{\nabla \phi}{\|\nabla \phi \|_{\epsilon,2}} \cdot \nabla + \|\nabla \phi\|_{\epsilon,2} \delta'(\phi) \mathcal{I},
\end{align}
in which $\mathcal{I}$ denotes the identity operator, applied to the level-set equation delivers the evolution of the surface Dirac (\ref{eq: surface Dirac}).
\end{proof}
To derive the local energy balance we introduce the following identity.
\begin{prop}\label{prop: id1 main}
It holds:
  \begin{align}
  -\nabla \cdot \left( \mathbf{P}_T(\phi) \delta_\Gamma(\phi) \right) = \delta_\Gamma(\phi)\boldsymbol{\nu}(\phi) \kappa(\phi)- \epsilon^2\delta'(\phi)\boldsymbol{\nu}(\phi).
  \end{align}
\end{prop}
\begin{proof}
 See \ref{appendix: surface tension derivation LS}.
\end{proof}

We now present the local energy balance.
\begin{lem}\label{lem: local energy balance}
  The local energy balance associated with system (\ref{eq: strong form phi conservative}) takes the form:
  \begin{align}\label{eq: local energy balance}
    \partial_t \mathcal{H}  + \nabla \cdot \left(\left(\left(\mathcal{H}+p\right)\mathbf{I}-\boldsymbol{\tau}(\bu,\phi)\right) \bu\right) -\frac{1}{\mathbb{W}{\rm e}}\nabla \cdot \left(  \delta_\Gamma(\phi) \mathbf{P}_T \bu\right)+ \boldsymbol{\tau}(\bu,\phi):\nabla \bu + \epsilon^2\frac{1}{\mathbb{W}{\rm e}}\delta'(\phi)u_\nu = 0.
  \end{align}
\end{lem}
The divergence terms represent the redistribution of energy over the domain and the second to last term accounts for energy dissipation due to diffusion. The last term that emanates from the regularization is unwanted. We return to this issue in \cref{sec: Energy-dissipative formulation}.
\begin{proof}
  First we consider the local kinetic energy of the system (\ref{eq: strong form phi conservative}). By straightforwardly applying the chain-rule we find:
  \begin{align}\label{eq: kin}
    \partial_t \left(\rho \frac{1}{2}\|\bu\|_2^2\right) =&~ \bu \cdot \partial_t (\rho \bu) - \tfrac{1}{2}\|\bu\|_2^2 \dfrac{\partial \rho}{\partial \phi}\partial_t \phi.
  \end{align}
  From the momentum and level-set equations, i.e. (\ref{eq: strong form phi conservative mom}) and (\ref{eq: strong form phi conservative LS}), we deduce:
  \begin{align}\label{eq: kin2}
    \bu \cdot \partial_t (\rho \bu) - \tfrac{1}{2}\|\bu\|_2^2 \dfrac{\partial \rho}{\partial \phi}\partial_t \phi  =&~ -\bu \cdot \nabla \cdot \left( \rho \bu \otimes \bu \right)+\bu^T\nabla \cdot\boldsymbol{\tau}(\bu, \phi) - \bu \cdot \nabla p - \frac{1}{\mathbb{W}{\rm e}} \kappa \delta_\Gamma u_\nu- \frac{1}{\mathbb{F}{\rm r}^2}\rho\bu \cdot \boldsymbol{\jmath} \nn\\
                                                        &~+ \tfrac{1}{2}\|\bu\|_2^2 \dfrac{\partial \rho}{\partial \phi}\bu \cdot \nabla \phi.
  \end{align}
  For the energetic contribution due the gravitational force, the chain-rule and the level-set equation (\ref{eq: strong form phi conservative LS}) convey that:
  \begin{align}\label{eq: grav}
    \partial_t \left(\frac{1}{\mathbb{F}{\rm r}^2} \rho y \right) =&~  \frac{1}{\mathbb{F}{\rm r}^2} y \dfrac{\partial \rho}{\partial \phi}\partial_t\phi =  -\frac{1}{\mathbb{F}{\rm r}^2} y \dfrac{\partial \rho}{\partial \phi}\bu \cdot \nabla \phi.
  \end{align}
  And for the local surface energy evolution we invoke \cref{lem: surface Dirac}:
  \begin{align}\label{eq: surf}
    \partial_t \left(\frac{1}{\mathbb{W}{\rm e}} \delta_\Gamma(\phi) \right) =&~  \frac{1}{\mathbb{W}{\rm e}}\left(\delta(\phi) \dfrac{\nabla \phi}{\|\nabla \phi \|_{\epsilon,2}} \cdot \nabla + \|\nabla \phi\|_{\epsilon,2} \delta'(\phi) \mathcal{I} \right) \partial_t \phi \nn\\
    =&~- \nabla \cdot \left(\frac{1}{\mathbb{W}{\rm e}} \delta_\Gamma(\phi) \bu \right)  + \frac{1}{\mathbb{W}{\rm e}}\delta_\Gamma(\phi) \nabla_\Gamma \bu.
  \end{align}
  Superposition of (\ref{eq: kin2})-(\ref{eq: surf}) yields:
    \begin{align}\label{eq: superpos}
    \partial_t \mathcal{H} =&~ -\bu \cdot \nabla \cdot \left( \rho \bu \otimes \bu \right) + \tfrac{1}{2}\|\bu\|_2^2 \dfrac{\partial \rho}{\partial \phi}\bu \cdot \nabla \phi \nn\\
                                                        &~- \frac{1}{\mathbb{F}{\rm r}^2}\rho\bu \cdot \boldsymbol{\jmath}-\frac{1}{\mathbb{F}{\rm r}^2} y \dfrac{\partial \rho}{\partial \phi}\bu \cdot \nabla \phi \nn\\
                                                        &~- \frac{1}{\mathbb{W}{\rm e}} \kappa \delta_\Gamma u_\nu- \nabla \cdot \left(\frac{1}{\mathbb{W}{\rm e}} \delta_\Gamma(\phi) \bu \right)  + \frac{1}{\mathbb{W}{\rm e}}\delta_\Gamma(\phi) \nabla_\Gamma \bu \nn\\
                                                        &~+\bu^T\nabla \cdot\boldsymbol{\tau}(\bu, \phi) - \bu \cdot \nabla p.
  \end{align}
  With the aim of simplifying (\ref{eq: superpos}) we introduce the identities:  
  \begin{subequations}\label{eq: identity cont}
  \begin{align}
    -\bu^T \nabla \cdot \left( \rho \bu \otimes \bu \right) + \tfrac{1}{2}\|\bu\|_2^2 \dfrac{\partial \rho}{\partial \phi}\bu \cdot \nabla \phi =&~ -  \nabla \cdot \left( \tfrac{1}{2}\rho\|\bu\|^2 \bu\right)- \tfrac{1}{2}\rho\|\bu\|^2 \nabla \cdot \bu, \\
    -\frac{1}{\mathbb{F}{\rm r}^2}\rho \bu \cdot  \boldsymbol{\jmath} -\frac{1}{\mathbb{F}{\rm r}^2} y \dfrac{\partial \rho}{\partial \phi}\bu \cdot \nabla \phi  =&~  - \nabla \cdot \left(\frac{1}{\mathbb{F}{\rm r}^2} \rho y \bu \right)+ \frac{1}{\mathbb{F}{\rm r}^2} \rho y \nabla \cdot\bu, \\
    \frac{1}{\mathbb{W}{\rm e}}\delta_\Gamma(\phi) \nabla_\Gamma \bu =&~ \nabla \cdot \left(\frac{1}{\mathbb{W}{\rm e}}\delta_\Gamma(\phi) \mathbf{P}_T \mathbf{u} \right) + \frac{1}{\mathbb{W}{\rm e}}\delta_\Gamma \kappa u_\nu- \epsilon^2\frac{1}{\mathbb{W}{\rm e}}\delta'(\phi)u_\nu.
  \end{align}
  \end{subequations}  
  The first and the second identity follow from expanding the gradient and divergence operators.
  To obtain the third we note
  \begin{align}
    \delta_\Gamma(\phi) \nabla_\Gamma \bu = \nabla \cdot \left( \delta_\Gamma(\phi) \mathbf{P}_T \mathbf{u} \right) - \bu \cdot \nabla \left( \delta_\Gamma(\phi) \mathbf{P}_T \right)
  \end{align}
  and apply \cref{prop: id1 main} on the second term.
  Invoking (\ref{eq: identity cont}) into (\ref{eq: superpos}) and adding a suitable partition of zero yields:
    \begin{align}\label{eq: superpos2}
    \partial_t \mathcal{H} + \nabla \cdot \left(\left(\left(\mathcal{H}+p\right)\mathbf{I}-\boldsymbol{\tau}(\bu,\phi)\right) \bu\right) -    \nabla \cdot \left(\frac{1}{\mathbb{W}{\rm e}}\delta_\Gamma(\phi) \mathbf{P}_T \mathbf{u} \right)=&~  -\boldsymbol{\tau}(\bu,\phi):\nabla \bu - \epsilon^2\frac{1}{\mathbb{W}{\rm e}}\delta'(\phi)u_\nu\nn\\
    &~+\left(- \tfrac{1}{2}\rho\|\bu\|^2 + p+ \frac{1}{\mathbb{F}{\rm r}^2} \rho y\right) \nabla \cdot \bu.
  \end{align}  
  With the aid of the continuity equation (\ref{eq: strong form phi conservative cont}) the latter member on the right-hand side of (\ref{eq: superpos2}) vanishes. This completes the proof.
\end{proof}
\begin{rmk}
The energy balance of \ref{lem: local energy balance} may also be written as:
 \begin{align}
    \partial_t \mathcal{H}  + \nabla \cdot \left(\left(\mathcal{H}+p\right) \bu\right) &-\frac{1}{\mathbb{R}{\rm e}} \nabla \cdot \left(2 \mu(\phi) \nabla \left( \tfrac{1}{2} \|\bu\|^2_2 \right)\right)~\nn\\
   & -\frac{1}{\mathbb{W}{\rm e}}\nabla \cdot \left(  \delta_\Gamma(\phi) \mathbf{P}_T \bu\right) + \boldsymbol{\tau}(\bu,\phi):\nabla \bu  + \epsilon^2\frac{1}{\mathbb{W}{\rm e}}\delta'(\phi)u_\nu =0.
 \end{align}
 In this form we clearly see that the second divergence term represents the diffusion of kinetic energy density.
\end{rmk}

We can now present the global energy evolution.
\begin{thm}\label{thm: LS energy balance}
 Let $\bu, p$ and $\phi$ be smooth solutions of the strong form (\ref{eq: strong form phi conservative}). The associated total energy $\mathscr{E}$, given in (\ref{eq: total energy LS}), satisfies the dissipation inequality:
 \begin{align}
   \frac{{\rm d}}{{\rm d}t} \mathscr{E}(\bu, \phi)=- (\boldsymbol{\tau}(\bu, \phi),\nabla \bu)_\Omega + {\rm bnd} \leq 0 + {\rm bnd},
 \end{align}
 where we have set $\epsilon = 0$.
\end{thm}
\begin{proof}
This follows from integrating the energy balance of \cref{lem: local energy balance} over $\Omega$ and using the divergence theorem:
\begin{align}
  \displaystyle\int_\Omega \partial_t \mathcal{H} ~{\rm d}\Omega &~+ \displaystyle\int_\Omega \boldsymbol{\tau}(\bu,\phi):\nabla \bu ~{\rm d}\Omega+ \displaystyle\int_{\partial \Omega} u_n \left(\mathcal{H}+p\right) -\bn^T \boldsymbol{\tau}(\bu,\phi)\bu~ {\rm d}\Omega 
  \nn\\
  &~ + \displaystyle\int_\Omega  \epsilon^2\frac{1}{\mathbb{W}{\rm e}}\delta'(\phi)u_\nu ~{\rm d}\Omega - \displaystyle\int_{\partial \Omega} \frac{1}{\mathbb{W}{\rm e}} \delta_\Gamma(\phi)  \bn^T\mathbf{P}_T(\phi)\bu  ~{\rm d S} = 0.
\end{align}
We discard the line force terms on the right-hand side and reorganize to get:
\begin{align}\label{eq: agrg}
  \dt \mathscr{E}(\bu, \phi) =&~ - \displaystyle\int_\Omega \boldsymbol{\tau}(\bu,\phi):\nabla \bu ~{\rm d}\Omega- \displaystyle\int_\Omega  \epsilon^2\frac{1}{\mathbb{W}{\rm e}}\delta'(\phi)u_\nu ~{\rm d}\Omega\nn\\
  &~+\displaystyle\int_{\partial \Omega} \bn^T \boldsymbol{\tau}(\bu,\phi)\bu-u_n \left(\rho \tfrac{1}{2}\|\bu\|_2^2+ \frac{1}{\mathbb{F}{\rm r}^2}\rho y +p\right) ~ {\rm d S} = 0.
\end{align}
Using the homogeneous boundary condition and setting $\epsilon =0$ finalizes the proof.
\end{proof}

The energy balance associated with the original model (\ref{eq: strong form total}) and that of the level-set formulation (\ref{eq: strong form phi conservative}) comply.
\begin{corol}
The energetic balance associated with diffuse model (\ref{eq: strong form phi conservative}) (\cref{thm: LS energy balance}) is consistent that of the original model (\ref{eq: strong form total}) (\cref{theorem: energy dissipation strong original}).
\end{corol}
\begin{proof}
In the limit $\eps \rightarrow 0$ we may transform (\ref{eq: agrg}) back to get:
\begin{align}
  \dt \mathscr{E}(\bu) =&~ \displaystyle\int_{\partial \Omega} \bn^T \boldsymbol{\tau}(\bu)\bu-u_n \left(\rho \tfrac{1}{2}\|\bu\|_2^2+\frac{1}{\mathbb{F}{\rm r}^2} \rho y +p\right) ~ {\rm d S} 
  \nn\\
  &~- \displaystyle\int_\Omega \boldsymbol{\tau}(\bu):\nabla \bu ~{\rm d}\Omega.
\end{align}
\end{proof}

To close this Section we note that one may avoid evaluating second derivatives appearing in the surface tension term.
This holds for the original model (\ref{eq: strong form total}) which we have addressed with briefly in \cref{rmk: alternative form}.
In the following Proposition we note that this alternative form directly converts to the diffuse model (\ref{eq: strong form phi conservative}).
\begin{prop}\label{lem: equivalence}
We have the identity:
  \begin{align}
  \displaystyle\int_\Omega \frac{1}{\mathbb{W}{\rm e}} \delta_\Gamma(\phi) \kappa(\phi) \boldsymbol{\nu}(\phi) \cdot  \bw ~{\rm d}\Omega =&~\displaystyle\int_\Omega \frac{1}{\mathbb{W}{\rm e}} \delta_\Gamma(\phi) \nabla \bw: \mathbf{P}_T(\phi)~{\rm d}\Omega + \frac{1}{\mathbb{W}{\rm e}}\displaystyle\int_\Omega \epsilon^2 \delta'(\phi)\boldsymbol{\nu}(\phi) \bw~{\rm d}\Omega.
  \end{align}
\end{prop}
\begin{proof}
 See \ref{appendix: surface tension derivation LS}.
\end{proof}
With the aid of \cref{lem: equivalence} one can directly evaluate the surface tension term and does not require any additional procedure such as the one from \cite{jansen1999better}.

\section{Energy-dissipative formulation}\label{sec: Energy-dissipative formulation}
We aim to develop an energetically stable Galerkin-type finite element method for the diffuse level-set model (\ref{eq: strong form phi conservative}).
In \cref{sec: Standard continuous formulations,sec: diffuse} we have in great detail depicted the procedure to arrive at the energy dissipative statement.
This procedure involves several steps that are not valid when dealing with standard finite element discretization spaces.
For instance the operator (\ref{eq: operator}) associated with the surface energy is not permittable in a standard discrete setting .
Independently, the temporal discretization also gives rise to issues.
Standard second-order semi time-discrete formulations of (\ref{eq: strong form phi conservative}) are also equipped with an energy-dissipative structure.
We demonstrate this in \ref{appendix: Standard mid-point Galerkin discretizations}.
Lastly, we note that the standard diffuse-interface model contains an unwanted term stemming from the regularization.

The first two issues arise from the fact that the standard model is too restrictive with regard to the function spaces.
Enlarging the standard function spaces introduces many complications and as such we do not further look into this strategy.
The alternative is modify the diffuse model (\ref{eq: strong form phi conservative}).
This is the road we pursue.
We employ the concept of \textit{functional entropy variables} proposed by Liu et al. \cite{liu2013functional}.
Liu and co-workers introduce the concept of functional entropy variables for the isothermal Navier-Stokes-Korteweg equations \cite{liu2013functional} and for the Navier-Stokes-Korteweg equations including the interstitial working flux term \cite{liu2015liquid}.
Here we apply the formalism to the level-set formulation of the incompressible Navier-Stokes equations with surface tension.
This creates the extra space to resolve both discrepancies mentioned above.
Additionally, the unwanted regularization term also vanishes.

\subsection{Functional entropy variables}
Energetic stability for the incompressible Navier-Stokes equations with surface tension coincides with stability with respect to a mathematical entropy function.
Thus to construct an energy-dissipative formulation for the incompressible Navier-Stokes equations the natural approach seems to adopt entropy principles.
For systems of conservation laws classical entropy variables are defined as the partial derivatives of an entropy with respect to the conservation variables.
The Clausius-Duhem inequality plays the role of energetic stability and this results from pre-multiplication of the system of conservation laws by the entropy variables.
The standard approach of constructing an entropy stable discretization as in Hughes et al. \cite{hughes1986new,Sha1091} is not applicable since the mathematical entropy is not an algebraic function of the conservation variables.
In the situation of a general mathematical entropy functional the derivatives should be taken in the functional setting.
The corresponding Clausius-Duhem inequality is then the result from the action of the entropy variables on the system of conservation laws.

In the current study we wish to inherit the notion of energetic stability for the incompressible model with surface tension.
To this purpose we use as mathematical entropy functional the energy density (\ref{eq: loc energy}) which we recall here:
\begin{align}\label{eq: math entropy}
 \mathcal{H} = \tfrac{1}{2}\rho \|\bu\|_2^2 + \frac{1}{\mathbb{F}{\rm r}^2}\rho y +\frac{1}{\mathbb{W}{\rm e}} \delta_\Gamma.
\end{align}
Following the approach described above, energetic stability results from the action of the entropy variables on the system of equations.
In contrast to \cite{liu2013functional} and \cite{liu2015liquid} the notion of conservation variables does not exist.
Instead, the derivatives of $\mathcal{H}$ should here be taken with respect to the model variables $\mathbf{U}=(\phi, \rho \bu)$.
Remark that (\ref{eq: math entropy}) is a functional of the model variables $\mathbf{U}$:
\begin{align}\label{eq: math entropy conservation var}
 \mathcal{H} = \mathcal{H}(\mathbf{U})= \frac{\|\rho \bu\|_2^2}{2\rho(\phi)} +\frac{1}{\mathbb{F}{\rm r}^2}\rho(\phi) y + \frac{1}{\mathbb{W}{\rm e}} \delta(\phi) \| \nabla \phi \|_{\epsilon,2}.
\end{align}
Note that $\mathcal{H}$ contains a gradient term $\|\nabla \phi\|$ which is non-local and thus the appropriate derivative is the functional derivative.
We define the entropy variables as functional derivatives:
\begin{align}
 \mathbf{V} = [V_1; V_2; V_3; V_4]^T= \dfrac{\updelta \mathcal{H}}{\updelta \mathbf{U}}  = \left[\dfrac{\updelta \mathcal{H}}{\updelta \phi}; \dfrac{\updelta \mathcal{H}}{\updelta (\rho u_1)}; \dfrac{\updelta \mathcal{H}}{\updelta (\rho u_2)}; \dfrac{\updelta \mathcal{H}}{\updelta (\rho u_3)}\right]^T.
\end{align}
The resulting functional derivatives are for test functions $\updelta \bv= [\updelta v_1, \updelta v_2, \updelta v_3, \updelta v_4]^T$:
\begin{subequations}\label{eq: functional derivatives}
\begin{align}
 \dfrac{\updelta \mathcal{H}}{\updelta \phi}[\updelta v_1] 
                                                     =&~ -\tfrac{1}{2} \|\bu\|_2^2\rho'(\phi)\updelta v_1 + \frac{1}{\mathbb{F}{\rm r}^2} \rho'(\phi) y \updelta v_1 \nn\\
                                                      &~ +\frac{1}{\mathbb{W}{\rm e}} \delta(\phi) \dfrac{\nabla \phi}{\|\nabla \phi\|_{\epsilon,2}}\cdot \nabla \updelta v_1 + \frac{1}{\mathbb{W}{\rm e}} \|\nabla \phi \|_{\epsilon,2} \delta'(\phi) \updelta v_1,\\
 \dfrac{\updelta \mathcal{H}}{\updelta (\rho u_1)}[\updelta v_2] =&~u_1 \updelta v_2, \\
 \dfrac{\updelta \mathcal{H}}{\updelta (\rho u_2)}[\updelta v_3] =&~u_2 \updelta v_3, \\
 \dfrac{\updelta \mathcal{H}}{\updelta (\rho u_3)}[\updelta v_4] =&~u_3 \updelta v_4.
\end{align}
\end{subequations}
We emphasize that it is essential to use the expression in term of the model variables (\ref{eq: math entropy conservation var}) to evaluate (\ref{eq: functional derivatives}). 
The associated explicit form of (\ref{eq: functional derivatives}) reads:
\begin{subequations}\label{eq: functional derivatives explicit}
\begin{align}
 \dfrac{\updelta \mathcal{H}}{\updelta \phi} =&~ -\tfrac{1}{2} \|\bu\|_2^2\rho'(\phi) + \frac{1}{\mathbb{F}{\rm r}^2} \rho'(\phi) y  -\frac{1}{\mathbb{W}{\rm e}} \delta(\phi) \nabla \cdot \left(\dfrac{\nabla \phi}{\|\nabla \phi\|_{\epsilon,2}}\right) + \frac{1}{\mathbb{W}{\rm e}} \delta'(\phi) \dfrac{\epsilon^2}{\|\nabla \phi\|_{\epsilon,2}},\\
 \dfrac{\updelta \mathcal{H}}{\updelta (\rho \bu)} =&~\bu^T.
\end{align}
\end{subequations}
We may use the functional entropy variables to systematically recover the energy balance (\ref{eq: local energy balance}).
\begin{thm}\label{thm: apply entropy var}
 Applying the functional entropy variables to the incompressible two-phase Navier-Stokes equations with surface tension recovers the energy balance (\ref{eq: local energy balance}):
  \begin{align}
    \partial_t \mathcal{H}  + \nabla \cdot \left(\left(\left(\mathcal{H}+p\right)\mathbf{I}-\boldsymbol{\tau}(\bu,\phi)\right) \bu\right) + \boldsymbol{\tau}(\bu,\phi):\nabla \bu -\frac{1}{\mathbb{W}{\rm e}} \nabla \cdot \left(  \delta_\Gamma(\phi) \mathbf{P}_T \bu\right) + \epsilon^2\frac{1}{\mathbb{W}{\rm e}}\delta'(\phi)u_\nu = 0.
  \end{align}
\end{thm}
\begin{proof}
  Application of the functional entropy variables on the time-derivatives provides:
  \begin{align}\label{eq: entropy var 0}
    \mathbf{V}\left[\dfrac{\partial \mathbf{U} }{\partial t}\right] = \dfrac{\updelta \mathcal{H}}{\updelta \mathbf{U}}\left[\dfrac{\partial \mathbf{U} }{\partial t}\right] = \dfrac{\partial \mathcal{H} }{\partial t}.
  \end{align}
  Next we apply the entropy variables on the fluxes to get:
  \begin{align}\label{eq: entropy var 1}
    \mathbf{V}\begin{bmatrix}
                \bu \cdot \nabla \phi \\
                \nabla \cdot \left(\rho \bu \otimes \bu \right) + \nabla p 
              \end{bmatrix} =&~ -\left(\tfrac{1}{2}\|\bu\|^2_2 - \frac{1}{\mathbb{F}{\rm r}^2}y\right) \dfrac{\partial \rho}{\partial \phi} \bu \cdot \nabla \phi + \bu^T \nabla \cdot \left(\rho \bu \otimes \bu \right) \nn\\
              &~+ \nabla \cdot \left( p \bu\right) - p \nabla \cdot \bu \nn\\
                             &~ +\frac{1}{\mathbb{W}{\rm e}} \delta(\phi) \dfrac{\nabla \phi}{\|\nabla \phi\|_{\epsilon,2}}\cdot \nabla (\bu \cdot \nabla \phi) + \frac{1}{\mathbb{W}{\rm e}} \|\nabla \phi \|_{\epsilon,2} \delta'(\phi) (\bu \cdot \nabla \phi)
  \end{align}
  Testing the entropy variables with the surface tension term gives:
  \begin{align}\label{eq: entropy var 2a}
    \mathbf{V}\begin{bmatrix}
                0 \\
               \dfrac{1}{\mathbb{W}{\rm e}} \delta_\Gamma (\phi)\boldsymbol{\nu}(\phi) \kappa(\phi) 
              \end{bmatrix} =&~ \frac{1}{\mathbb{W}{\rm e}} \delta_\Gamma (\phi) \kappa(\phi) u_\nu(\phi).                            
  \end{align}
  Testing the entropy variables with the viscous stress yields:
  \begin{align}\label{eq: entropy var 2}
    \mathbf{V}\begin{bmatrix}
                0 \\
               -\nabla \cdot \boldsymbol{\tau}(\bu, \phi)
              \end{bmatrix} 
                            =&~ -\nabla\cdot \left( \boldsymbol{\tau}(\bu, \phi)\bu\right) + \boldsymbol{\tau}(\bu, \phi): \nabla \bu
  \end{align}  
  And finally testing with the body force yields:
  \begin{align}\label{eq: entropy var 3}
    \mathbf{V}\begin{bmatrix}
                0 \\
               \dfrac{1}{\mathbb{F}{\rm r}^2}\rho \boldsymbol{\jmath}
              \end{bmatrix} =  \frac{1}{\mathbb{F}{\rm r}^2}\rho \bu\cdot\boldsymbol{\jmath}.
  \end{align}  
  Addition of (\ref{eq: entropy var 1}), (\ref{eq: entropy var 2a}), (\ref{eq: entropy var 2}) and (\ref{eq: entropy var 3}) gives:
  \begin{align}\label{eq: entropy var end}
   & \mathbf{V}\begin{bmatrix}
                \bu \cdot \nabla \phi \\[8pt]
                \nabla \cdot \left(\rho \bu \otimes \bu \right) + \nabla p-\nabla \cdot \boldsymbol{\tau}+\dfrac{1}{\mathbb{F}{\rm r}^2}\rho \boldsymbol{\jmath}
                + \dfrac{1}{\mathbb{W}{\rm e}} \delta_\Gamma (\phi)\boldsymbol{\nu}(\phi) \kappa(\phi)
              \end{bmatrix} {\color{white}afafdasdfafsefsegssegsgsgseg}\nn\\
              & \quad \quad\quad \quad\quad \quad \quad \quad = -\tfrac{1}{2}\|\bu\|^2_2  \dfrac{\partial \rho}{\partial \phi} \bu \cdot \nabla \phi + \bu^T \nabla \cdot \left(\rho \bu \otimes \bu \right) \nn\\              
              &\quad \quad \quad\quad \quad\quad \quad\quad \quad+ \nabla \cdot \left( p \bu\right) - p \nabla \cdot \bu \nn\\
                              &\quad \quad \quad\quad \quad\quad \quad\quad \quad+ \frac{1}{\mathbb{F}{\rm r}^2}\rho \bu\cdot\boldsymbol{\jmath} +\frac{1}{\mathbb{F}{\rm r}^2}y \bu \cdot \nabla \rho\nn\\
                              &\quad \quad \quad\quad \quad\quad \quad\quad \quad+\frac{1}{\mathbb{W}{\rm e}} \delta(\phi) \dfrac{\nabla \phi}{\|\nabla \phi\|_{\epsilon,2}}\cdot \nabla (\bu \cdot \nabla \phi) + \frac{1}{\mathbb{W}{\rm e}} \|\nabla \phi \|_{\epsilon,2} \delta'(\phi) (\bu \cdot \nabla \phi) +\frac{1}{\mathbb{W}{\rm e}} \boldsymbol{\nu}(\phi) \kappa(\phi) u_\nu(\phi)\nn\\
                              &\quad \quad \quad\quad \quad\quad \quad\quad \quad-\nabla\cdot \left( \boldsymbol{\tau}(\bu, \phi)\bu\right) + \boldsymbol{\tau}(\bu, \phi): \nabla \bu.
  \end{align}
Recognize the operator (\ref{eq: operator}) on the fourth line of the right-hand side of (\ref{eq: entropy var end}). We may thus use \cref{lem: surface Dirac} and write:
\begin{align}\label{eq: id2}
\frac{1}{\mathbb{W}{\rm e}} \delta(\phi) \dfrac{\nabla \phi}{\|\nabla \phi\|_{\epsilon,2}}\cdot \nabla (\bu \cdot \nabla \phi) + \frac{1}{\mathbb{W}{\rm e}} \|\nabla \phi \|_{\epsilon,2} \delta'(\phi) (\bu \cdot \nabla \phi) =  \nabla \cdot \left( \frac{1}{\mathbb{W}{\rm e}} \delta_\Gamma(\phi) \bu \right)  - \frac{1}{\mathbb{W}{\rm e}}\delta_\Gamma(\phi) \nabla_\Gamma \bu.
\end{align}
  Invoking the identities (\ref{eq: identity cont}) and (\ref{eq: id2}) the expression (\ref{eq: entropy var end}) collapses to
  \begin{align}\label{eq: entropy var end2}
   & \mathbf{V}\begin{bmatrix}
                \bu \cdot \nabla \phi \\[8pt]
                \nabla \cdot \left(\rho \bu \otimes \bu \right) + \nabla p-\nabla \cdot \boldsymbol{\tau}(\bu, \phi)+\dfrac{1}{\mathbb{F}{\rm r}^2}\rho \boldsymbol{\jmath}
                + \dfrac{1}{\mathbb{W}{\rm e}} \delta_\Gamma (\phi)\boldsymbol{\nu}(\phi) \kappa(\phi)
              \end{bmatrix} {\color{white}afafdasdfafsefsegssegsgsgseg}\nn\\
              & \quad \quad\quad \quad\quad \quad \quad \quad =   \nabla \cdot \left( \tfrac{1}{2}\rho\|\bu\|^2 \bu\right)+ \tfrac{1}{2}\rho\|\bu\|^2 \nabla \cdot \bu \nn\\    
                              &\quad \quad \quad\quad \quad\quad \quad\quad \quad+ \nabla \cdot \left( p \bu\right) - p \nabla \cdot \bu \nn\\
                              &\quad \quad \quad\quad \quad\quad \quad\quad \quad+ \nabla \cdot \left(\frac{1}{\mathbb{F}{\rm r}^2} \rho y \bu \right)- \frac{1}{\mathbb{F}{\rm r}^2} \rho y \nabla \cdot\bu\nn\\
                              &\quad \quad \quad\quad \quad\quad \quad\quad \quad +\nabla \cdot \left( \frac{1}{\mathbb{W}{\rm e}} \delta_\Gamma(\phi) \bu \right)  - \nabla \cdot \left(\frac{1}{\mathbb{W}{\rm e}}\delta_\Gamma(\phi) \mathbf{P}_T \mathbf{u} \right)+ \epsilon^2\frac{1}{\mathbb{W}{\rm e}}\delta'(\phi)u_\nu = 0\nn\\
                              &\quad \quad \quad\quad \quad\quad \quad\quad \quad-\nabla\cdot \left( \boldsymbol{\tau}(\bu, \phi)\bu\right) + \boldsymbol{\tau}(\bu, \phi): \nabla \bu.
  \end{align}
  We merge the terms in (\ref{eq: entropy var end2}) and use the continuity equation (\ref{eq: strong form phi conservative cont}) to cancel the terms containing the divergence of velocity. Taking the superposition of (\ref{eq: entropy var 0}) and (\ref{eq: entropy var end2}) while recognizing $\mathcal{H}$ on the right-hand side of (\ref{eq: entropy var end2}) completes the proof.
\end{proof}

\subsection{Modified formulation}
\cref{thm: apply entropy var} implies that an energy-dissipative relation may be recovered when the functional entropy variables are available as test functions.
For standard test function spaces we can not select the weight $V_1$.
We circumvent this issue, similar as in \cite{liu2013functional}, by explicitly adding $V_1$ as a new unknown $v$ to the system of equations.
Thus we introduce the extra variable:
\begin{align}\label{eq: extra var}
 v = -\frac{\varrho}{2} \|\bu\|_2^2 + \frac{1}{\mathbb{F}{\rm r}^2} \varrho y -\frac{1}{\mathbb{W}{\rm e}} \delta(\phi) \nabla \cdot \left(\dfrac{\nabla \phi}{\|\nabla \phi\|_{\epsilon,2}}\right)+ \frac{1}{\mathbb{W}{\rm e}} \delta'(\phi) \dfrac{\epsilon^2}{\|\nabla \phi\|_{\epsilon,2}}.
\end{align}
where we use the notation $\varrho = \varrho(\phi): = \rho'(\phi)$.
The question arises how to couple the extra variable (\ref{eq: extra var}) to the diffuse-interface model (\ref{eq: strong form phi conservative}).
Note that a direct consequence of (\ref{eq: extra var}) is:
\begin{align}\label{eq: new term}
  -\left(v +\frac{ \varrho}{2} \|\bu\|_2^2-\frac{1}{\mathbb{F}{\rm r}^2}\varrho  y \right)\nabla \phi=&~ \frac{1}{\mathbb{W}{\rm e}} \delta(\phi) \nabla \cdot \left(\dfrac{\nabla \phi}{\|\nabla \phi\|_{\epsilon,2}}\right)\nabla \phi - \frac{1}{\mathbb{W}{\rm e}} \delta'(\phi) \dfrac{\epsilon^2}{\|\nabla \phi\|_{\epsilon,2}}\nabla \phi\nn\\
                =&~ \frac{1}{\mathbb{W}{\rm e}} \nabla \cdot \left(\dfrac{\nabla \phi}{\|\nabla \phi\|_{\epsilon,2}}\right)\dfrac{\nabla \phi}{\|\nabla \phi\|_{\epsilon,2}} \delta_\Gamma(\phi)- \epsilon^2\frac{1}{\mathbb{W}{\rm e}} \delta'(\phi) \boldsymbol{\nu}(\phi).
\end{align}
Recall that the diffuse-interface model (\ref{eq: strong form phi conservative}) is only associated with an energy-dissipative structure for $\epsilon = 0$, see \cref{thm: LS energy balance}.
This dissipative structure does not change when performing a consistent modification.
Thus adding a suitable partition of zero based on (\ref{eq: new term}) to the momentum equation (\ref{eq: strong form phi conservative mom}) keeps the same energy behavior.
Instead, we suggest to replace the surface tension term in (\ref{eq: strong form phi conservative}), i.e.
\begin{align}
 \frac{1}{\mathbb{W}{\rm e}} \nabla \cdot \left(\dfrac{\nabla \phi}{\|\nabla \phi\|_{\epsilon,2}}\right)\dfrac{\nabla \phi}{\|\nabla \phi\|_{\epsilon,2}} \delta_\Gamma(\phi),
\end{align}
by the left-hand side of (\ref{eq: new term}), i.e.
\begin{align}
 -\left(v +\frac{ \varrho}{2} \|\bu\|_2^2-\frac{1}{\mathbb{F}{\rm r}^2}\varrho  y \right)\nabla \phi.
\end{align}
In this way we eliminate the unwanted regularization term. The new strong form writes in terms of the variables $\bu, p, \phi$ and $v$ as:
\begin{subequations}\label{eq: strong form phi v}
\begin{align}
    \partial_t (\rho(\phi)\bu) + \nabla \cdot \left(\rho(\phi) \bu \otimes \bu\right) - \nabla \cdot \boldsymbol{\tau}(\bu, \phi) + \nabla p 
  -\left(v+\frac{\varrho}{2} \|\bu\|_2^2 -\frac{1}{\mathbb{F}{\rm r}^2} \varrho y \right)\nabla \phi +\frac{1}{\mathbb{F}{\rm r}^2}\rho(\phi) \boldsymbol{\jmath}&=~ 0, \\
 \nabla \cdot \bu &=~ 0, \\
 \partial_t \phi + \bu \cdot \nabla \phi &=~0,\\
 v + \varrho\frac{\|\bu\|_2^2}{2}- \frac{1}{\mathbb{F}{\rm r}^2} \varrho y +\frac{1}{\mathbb{W}{\rm e}} \delta(\phi) \nabla \cdot \left(\dfrac{\nabla \phi}{\|\nabla \phi\|_{\epsilon,2}}\right)- \frac{1}{\mathbb{W}{\rm e}} \delta'(\phi) \dfrac{\epsilon^2}{\|\nabla \phi\|_{\epsilon,2}}&=~0,
\end{align}
\end{subequations}
with $\bu(0) = \bu_0$ and $\phi(0)=\phi_0$ in $\Omega$.
\begin{rmk}
  Even in absence of surface tension effects the substitution (\ref{eq: new term}) is essential to arrive at an energy-dissipative system.
\end{rmk}

The corresponding weak formulation reads:\\

\textit{Find $(\bu,p, \phi, v) \in \mathcal{W}_T$ such that for all $(\bw, q, \psi, \zeta) \in \mathcal{W}$:}
\begin{subequations}\label{eq: weak form LS cont v}
\begin{align}
(\bw,  \partial_t (\rho \bu))_\Omega  -(\nabla \bw, \rho \bu \otimes \bu)_\Omega - (\nabla \cdot \bw, p)_\Omega+ (\nabla \bw, \boldsymbol{\tau}(\bu,\phi))_\Omega +\frac{1}{\mathbb{F}{\rm r}^2}(\bw, \rho \boldsymbol{\jmath})_\Omega &\nn \\
-\left(\bw, v \nabla \phi\right)_\Omega - \left(\bw, \left(\frac{\varrho}{2} \|\bu\|_2^2 -\frac{1}{\mathbb{F}{\rm r}^2}\varrho y \right) \nabla \phi \right)_\Omega &=~ 0, \label{weak form LS mom cont}\\
 ( q, \nabla \cdot \bu)_\Omega &=~ 0, \label{weak form LS continuity}\\
  ( \psi, \partial_t \phi + \bu \cdot \nabla \phi )_\Omega &=~ 0, \label{weak form LS phi cont}\\
  \left(\zeta, v+\varrho\frac{\|\bu\|_2^2}{2}-\frac{1}{\mathbb{F}{\rm r}^2}\varrho y \right)_\Omega   
  -\left(\frac{1}{\mathbb{W}{\rm e}} \delta(\phi) \dfrac{\nabla \phi}{\|\nabla \phi\|_{\epsilon,2}}, \nabla \zeta\right)_\Omega - \left(\frac{1}{\mathbb{W}{\rm e}} \|\nabla \phi \|_{\epsilon,2} \delta'(\phi), \zeta\right)_\Omega  
  &=~0 , \label{weak form v cont}
\end{align}
\end{subequations}
where we recall $\varrho = \partial \rho/\partial \phi$ and have $\bu(0) = \bu_0$ and $\phi(0)=\phi_0$ in $\Omega$. Here $\mathcal{W}_T$ denotes a divergence-compatible space-time space and $\mathcal{W}$ is the test-function space. For details about the divergence-compatible space we refer to \cite{Evans13unsteadyNS}.
\begin{thm}
\label{thm: energy dissipation modified model cont}
 Let $(\bu, p, \phi)$ be a smooth solution of the weak form (\ref{eq: weak form LS cont v}). 
 The formulation (\ref{eq: weak form LS cont v}) has the properties:
 \begin{enumerate}
  \item The formulation satisfies the maximum principle for the density, i.e. without loss of generality we assume that $\rho_2 \leq \rho_1$ and then have:
  \begin{align}
     \rho_2 \leq \rho(\phi) \leq \rho_1,.
  \end{align}
  \item The formulation is divergence-free as a distribution:
  \begin{align}\label{prop 2 cont}
    \nabla \cdot \bu \equiv 0.
  \end{align}
  \item The formulation satisfies the dissipation inequality:
    \begin{align}\label{thm: cont diss}
   \frac{{\rm d}}{{\rm d}t} \mathscr{E}(\bu, \phi)=- (\nabla \bu, \boldsymbol{\tau}(\bu, \phi))_\Omega \leq 0.
 \end{align}
 \end{enumerate}
\end{thm}
Dissipation inequality (\ref{thm: cont diss}) is not equipped with terms supported on the outer boundary $\partial \Omega$ since these vanish due to assumed boundary conditions.
\begin{proof}

1. This is a direct consequence of the definition of $\rho=\rho(\phi)$.\\

2. The divergence-conforming space allows to take $q = \nabla \cdot \bu$ in (\ref{weak form LS continuity}) and hence we find:
\begin{align}
 0= \left( \nabla \cdot \bu, \nabla \cdot \bu\right)_\Omega \quad \Rightarrow \quad \nabla \cdot \bu \equiv 0.
\end{align}

3. Selection of the weights
$\psi=v$ in (\ref{weak form LS phi cont}) and $\zeta = -\partial_t \phi$ in (\ref{weak form v cont}) yields:
\begin{subequations}\label{eq: cont proof weights 1 cont}
\begin{align}
   (v, \partial_t \phi+ \bu \cdot \nabla  \phi)_\Omega =&~ 0,\\
  -\left( \partial_t \phi, v+\varrho\frac{\|\bu\|_2^2}{2}- \frac{1}{\mathbb{F}{\rm r}^2}\varrho y\right)_\Omega +\left(\frac{1}{\mathbb{W}{\rm e}} \delta(\phi) \dfrac{\nabla \phi}{\|\nabla \phi\|_{\epsilon,2}}, \nabla \partial_t \phi\right)_\Omega + \left(\frac{1}{\mathbb{W}{\rm e}} \|\nabla \phi \|_{\epsilon,2} \delta'(\phi), \partial_t \phi\right)_\Omega=&~0.
\end{align}
\end{subequations}
We add the equations (\ref{eq: cont proof weights 1 cont}) and find:
\begin{align}\label{eq: cont combining}
 (v, \bu \cdot \nabla  \phi)_\Omega- \left( \frac{\varrho}{2}\|\bu\|_2^2, \partial_t \phi \right)_\Omega 
 +\frac{1}{\mathbb{F}{\rm r}^2}\left( \partial_t \phi,  \varrho y\right)_\Omega  +\left(\frac{1}{\mathbb{W}{\rm e}} \delta(\phi) \dfrac{\nabla \phi}{\|\nabla \phi\|_{\epsilon,2}}, \nabla \partial_t \phi\right)_\Omega &\nn\\ + \left(\frac{1}{\mathbb{W}{\rm e}} \|\nabla \phi \|_{\epsilon,2} \delta'(\phi), \partial_t \phi\right)_\Omega &~= 0.
\end{align}
Performing integration by parts yields:
\begin{align}\label{eq: cont combining IP}
  \left(\partial_t \phi, -\frac{\varrho}{2}\|\bu\|_2^2 +\varrho \frac{1}{\mathbb{F}{\rm r}^2} y -\frac{1}{\mathbb{W}{\rm e}} \delta(\phi) \nabla \cdot \left(\dfrac{\nabla \phi}{\|\nabla \phi\|_{\epsilon,2}}\right)+ \frac{1}{\mathbb{W}{\rm e}} \delta'(\phi) \dfrac{\epsilon^2}{\|\nabla \phi\|_{\epsilon,2}} \right)_\Omega =&~ -(v, \bu \cdot \nabla  \phi)_\Omega.
\end{align}
Recall that the line integral terms vanish due to auxiliary boundary conditions.
Noting that $ \dfrac{\updelta \mathcal{H}}{\updelta \phi} = -\dfrac{\varrho}{2} \|\bu\|_2^2 + \varrho \dfrac{1}{\mathbb{F}{\rm r}^2}y -\dfrac{1}{\mathbb{W}{\rm e}} \delta(\phi) \nabla \cdot \left(\dfrac{\nabla \phi}{\|\nabla \phi\|_{\epsilon,2}}\right)+ \dfrac{1}{\mathbb{W}{\rm e}} \delta'(\phi) \dfrac{\epsilon^2}{\|\nabla \phi\|_{\epsilon,2}} $ we arrive at:
\begin{align}\label{eq: cont combining IP 3}
 \dfrac{\updelta \mathscr{E}}{\updelta \phi}\left[\dfrac{\partial \phi}{\partial t}\right]:=\left(\dfrac{\partial \phi}{\partial t},\dfrac{\updelta \mathcal{H}}{\updelta \phi} \right)_\Omega  =&~ -(v, \bu \cdot \nabla  \phi)_\Omega .
\end{align}
Next we take $\bw = \bu$ in (\ref{weak form LS mom cont}) to get:
\begin{align}\label{eq: cont combining 2}
 (\bu,  \partial_t (\rho  \bu))_\Omega   -(\nabla \bu, \rho \bu \otimes \bu)_\Omega - (\bw, \tfrac{1}{2}\|\bu\|_2^2\varrho(\phi)\nabla \phi)_\Omega  - (\nabla \cdot \bu, p)_\Omega & \nn\\
 + (\nabla \bu, \boldsymbol{\tau}(\bu,\phi))_\Omega  -\left(\bu, v \nabla \phi\right)_\Omega+ \frac{1}{\mathbb{F}{\rm r}^2}\left(\bu, \varrho y \nabla \phi \right)_\Omega + \frac{1}{\mathbb{F}{\rm r}^2}(\bu, \rho \boldsymbol{\jmath})_\Omega &=0. 
\end{align}
From the identities (\ref{eq: identity cont}), the continuity equation (\ref{prop 2 cont}), homogeneous boundary conditions and integration by parts we extract the identities:
\begin{subequations}\label{eq: props cont}
\begin{align}
 -(\nabla \bu, \rho \bu \otimes \bu)_\Omega - (\bu, \tfrac{1}{2}\|\bu\|_2^2\varrho(\phi)\nabla \phi)_\Omega &~=0 \\[6 pt]
 - (\nabla \cdot \bu, p)_\Omega &~=0,\\
   \frac{1}{\mathbb{F}{\rm r}^2}\left(\bu, \varrho y \nabla \phi \right)_\Omega+\frac{1}{\mathbb{F}{\rm r}^2}(\bu, \rho \boldsymbol{\jmath})_\Omega &~= 0.
\end{align}
\end{subequations}
Noting that $ \dfrac{\updelta \mathcal{H}}{\updelta (\rho \bu)} = \bu^T$ and employing (\ref{eq: props cont}) we arrive at:
\begin{align}\label{eq: cont combining 2 a cont}
\dfrac{\updelta \mathscr{E}}{\updelta (\rho\bu)}\left[\dfrac{\partial (\rho \bu)}{\partial t}\right]:=\left(\dfrac{\partial (\rho\bu)}{\partial t},\dfrac{\updelta \mathcal{H}}{\updelta (\rho\bu)} \right)_\Omega =&~ - (\nabla \bu, \boldsymbol{\tau}(\bu,\phi))_\Omega  +\left(\bu, v \nabla \phi\right)_\Omega.
\end{align}
Addition of (\ref{eq: cont combining IP 3}) and (\ref{eq: cont combining 2 a cont}) yields:
\begin{align}\label{eq: cont combining 3}
\dt \mathscr{E} = \dfrac{\updelta \mathscr{E}}{\updelta \phi}\left[\dfrac{\partial \phi}{\partial t}\right] + \dfrac{\updelta \mathscr{E}}{\updelta (\rho\bu)}\left[\dfrac{\partial (\rho \bu)}{\partial t}\right]
= - (\nabla \bu, \boldsymbol{\tau}(\bu,\phi))_\Omega.
\end{align}
\end{proof}

\section{Energy-dissipative spatial discretization}\label{sec: Energy-dissipative discretization}
In this Section we present the spatial discretization of the modified model (\ref{eq: weak form LS cont v}). First we introduce some notation, then discuss the stabilization mechanisms and subsequently provide the semi-discrete formulation.

\subsection{Notation}
We employ an isogeometric analysis discretization.
To provide the appropriate setting, we introduce the parametric domain denoted as $\hat{\Omega} := (-1,1)^d \subset \mathbb{R}^d$ with corresponding mesh $\mathcal{M}$. 
The element size $h_Q = \text{diag}(Q)$ of an element $Q$ in $\mathcal{M}$ is its diagonal length.
The physical domain $\Omega \subset \mathbb{R}^d$ follows as usual via the continuously differentiable geometrical map (with continuously differentiable inverse) $\mathbf{F}:\hat{\Omega} \rightarrow \Omega$ and the corresponding physical mesh reads:
\begin{align}
  \mathcal{K} = \mathbf{F}(\mathcal{M}):= \left\{\Omega_K: \Omega_K= \mathbf{F}(Q), Q \in \mathcal{M} \right\}.
\end{align}
The Jacobian mapping is $\mathbf{J} = \partial \mathbf{x}/\partial \boldsymbol{\xi}$. The physical mesh size $h_K$ is given by
\begin{align}
  h_{K}^2 = \frac{h_Q^2}{d} \| \mathbf{J} \|_F^2,
\end{align}
with the subscript $F$ referring to the Frobenius norm. Note that on a Cartesian mesh it reduces to the diagonal-length of an element.
The element metric tensor reads
\begin{align}
  \mathbf{G} = \dfrac{\partial \boldsymbol{\xi}}{\partial \mathbf{x}}^T \dfrac{\partial \boldsymbol{\xi}}{\partial \mathbf{x}} = \mathbf{J}^{-T} \mathbf{J}^{-1},
\end{align}
with inverse
\begin{align}
  \mathbf{G}^{-1} = \dfrac{\partial \mathbf{x}}{\partial \boldsymbol{\xi}} \dfrac{\partial \mathbf{x}}{\partial \boldsymbol{\xi}}^T = \mathbf{J} \mathbf{J}^{T}.
\end{align}
Using the metric tensor we see that the Frobenius norm is objective:
\begin{align}
  \| \mathbf{J} \|_F^2 = {\rm Tr}\left(\mathbf{G}^{-1}\right),
\end{align}
where ${\rm Tr}$ denotes the trace operator.

We define approximation spaces $\WW^h_T \subset \WW_T, \WW^h \subset \WW$ spanned by finite element or NURBS basis functions. Recall that we utilize the div-conforming function spaces proposed by Evans et al. \cite{Evans13steadyNS, Evans13unsteadyNS}. Furthermore, we use the conventional notation superscript $h$ to indicate the discretized (vector) field of the corresponding quantity. 

\subsection{Stabilization}
It is well-known that a plain Galerkin discretization is prone to the development of numerical instabilities. This motivates the use of stabilization mechanisms. We employ the standard SUPG stabilization \cite{BroHug82} for the level-set convection, i.e. we augment the discrete level-set equation with
\begin{align}
 +  \displaystyle\sum_K \left(\tau_K \bu^h \cdot \nabla \psi^h, \mathscr{R}_{I} \phi^h\right)_{\Omega_K},
\end{align}
with residual
\begin{align}
 \mathscr{R}_{I} \phi^h :=\partial_t \phi^h + \bu^h \cdot \nabla \phi^h.
\end{align}
We use the standard definition for stabilization parameter $\tau$ as also given in \cite{EiAk17i}.
To ensure that the stabilization term does not upset the energetic stability property we balance it with the term:
\begin{align}
-  \displaystyle\sum_K \left(\tau_K \bw^h \cdot \nabla v^h, \mathscr{R}_{I} \phi^h\right)_{\Omega_K}
\end{align}
in the momentum equation.

\begin{rmk}
In the current paper we focus on an energy-dissipative method without multiscale stabilization contributions in the momentum equation such as \cite{BaCaCoHu07}.
Standard stabilized methods are not directly associated with an energy dissipative property and thus specific techniques are required to establish such a property, see e.g. \cite{principe2010dissipative, EiAk17ii, evans2020variational}.
We note that these methods are developed for the single-fluid case.
An extension to the current two-fluid case may be the topic of another paper.
\end{rmk}

A popular method to stabilize the momentum equation is to use discontinuity capturing devices.
We follow this road and augment the momentum equation with the discontinuity capturing term:
\begin{align}
 +  \displaystyle\sum_K \left(\nabla \bw^h, \theta_K \nabla \bu^h \right)_{\Omega_K}.
\end{align}
The discontinuity capturing viscosity is given by:
\begin{align}
 \theta_K = \mathcal{C} h_K \dfrac{\|\pmb{\mathscr{R}}_M(\rho^h \bu^h)\|_{\epsilon,2}}{\|\nabla \bu^h\|_{\epsilon,2}},
\end{align}
with conservative momentum residual
\begin{align}
   \pmb{\mathscr{R}}_M (\rho^h \bu^h):= \partial_t (\rho^h \bu^h) + \nabla \cdot (\rho^h \bu^h\otimes \bu^h) + \nabla \cdot \boldsymbol{\tau}(\bu^h, \phi^h) + \nabla p^h + \dfrac{1}{\mathbb{W}{\rm e}} \delta (\phi^h) \kappa \nabla \phi^h + \dfrac{1}{\mathbb{F}{\rm r}^2} \rho^h \boldsymbol{\jmath},
\end{align}
and $\mathcal{C}$ a user-defined constant.
The term clearly dissipates energy.
\begin{rmk}
  In order to avoid evaluating second derivatives in the surface tension contribution, one may project the residual onto the mesh and subsequently use \cref{prop: id1 main}.
\end{rmk}
\begin{rmk}
 Even though we present the stabilization and discontinuity capturing terms in an \textit{ad hoc} fashion, we wish to emphasize that these may be derived with the aid of the multiscale framework. The natural derivation for discontinuity capturing terms can be found in \cite{ten2020theoretical}.
\end{rmk}

\subsection{Semi-discrete formulation}\label{sec: Semi-discrete formulation}
The semi-discrete approximation of (\ref{eq: weak form LS cont v}) is stated as follows:\\

\textit{Find $(\bu^h, p^h, \phi^h, v^h) \in \WW^h_T$ such that for all $(\bw^h, q^h, \psi^h, \zeta^h) \in \WW^h$:}
\begin{subequations}\label{eq: weak form semi-discrete0}
\begin{align}
(\bw^h, \partial_t (\rho^h \bu^h))_\Omega -(\nabla \bw^h, \rho^h \bu^h \otimes \bu^h)_\Omega
- (\nabla \cdot \bw^h, p^h)_\Omega+ \left(\nabla \bw^h, \boldsymbol{\tau}(\bu^h,\phi^h)\right)_\Omega &\nn\\   
 +\frac{1}{\mathbb{F}{\rm r}^2}(\bw^h, \rho^h \boldsymbol{\jmath})_\Omega-\left(\bw^h, v^h \nabla \phi^h\right)_\Omega - \left(\bw^h, \varrho(\phi^h)\left(\frac{\|\bu^h\|_2^2}{2}-\frac{1}{\mathbb{F}{\rm r}^2} y\right) \nabla \phi^h \right)_\Omega&\nn\\
 +  \displaystyle\sum_K \left(\nabla \bw^h, \theta_K \nabla \bu^h \right)_{\Omega_K}-  \displaystyle\sum_K \left(\tau_K \bw^h \cdot \nabla v^h, \mathscr{R}_{I} \phi^h \right)_{\Omega_K} &=~ 0, \label{weak form LS mom cont 2 semi}\\
( q^h, \nabla \cdot \bu^h)_\Omega &=~ 0, \label{weak form LS continuity semi}\\
  \left(\psi^h, \partial_t \phi^h + \bu^h \cdot \nabla \phi^h\right)_\Omega +  \displaystyle\sum_K \left(\tau_K \bu^h \cdot \nabla \psi^h, \mathscr{R}_{I} \phi^h\right)_{\Omega_K}  &=~0, \label{weak form LS phi cont semi}\\
  \left(\zeta^h, v^h+\varrho(\phi^h)\left(\frac{\|\bu^h\|_2^2}{2}-\frac{1}{\mathbb{F}{\rm r}^2} y\right)\right)_\Omega -\left(\frac{1}{\mathbb{W}{\rm e}} \delta(\phi^h) \dfrac{\nabla \phi^h}{\|\nabla \phi^h\|_{\epsilon,2}}, \nabla \zeta^h\right)_\Omega &\nn\\
  - \left(\frac{1}{\mathbb{W}{\rm e}} \|\nabla \phi^h \|_{\epsilon,2} \delta'(\phi^h), \zeta^h\right)_\Omega  &=~0 . \label{weak form v cont semi}
\end{align}
\end{subequations}
where we recall $\varrho^h=\varrho(\phi^h)$ and $\bu^h(0) = \bu_0$ and $\phi^h(0)=\phi_0$ in $\Omega$. The initial fields $\bu_0$ and $\phi_0$ are obtained via standard $L^2$-projections onto the mesh.
The density and fluid viscosity are computed as
\begin{subequations}\label{eq: rho mu LS approximate}
\label{eq:weak}
\begin{align}
 \rho^h\equiv \rho(\phi^h), \label{eq: rho H eps}\\
 \mu^h\equiv \mu(\phi^h). \label{eq: mu H eps}
\end{align}
\end{subequations}
The discrete counterparts of the kinetic, gravitational and surface energy are:
\begin{subequations}\label{eq: discrete energies}
\begin{align}
 \mathscr{E}^{\text{K},h}\equiv&~ \mathscr{E}^{\text{K}}(\bu^h; \phi^h) ,\\
 \mathscr{E}^{\text{G},h}\equiv&~ \mathscr{E}^{\text{G}}(\phi^h) ,\\
 \mathscr{E}^{\text{S},h}\equiv&~ \mathscr{E}^{\text{S}}(\phi^h) .
\end{align}
\end{subequations}
The total energy is the superposition of the separate energies:
\begin{align}
 \mathscr{E}^h :=&~\mathscr{E}^{\text{K},h} + \mathscr{E}^{\text{G},h} + \mathscr{E}^{\text{S},h}.
\end{align}
Similarly, the semi-discrete local energy reads
\begin{align}
  \mathcal{H}^h \equiv \mathcal{H}(\mathbf{U}^h).
\end{align}
The semi-discrete formulation (\ref{eq: weak form semi-discrete0}) inherits to a large extend \cref{thm: energy dissipation modified model cont}. The notable difference lies in the usage of stabilization terms.
\begin{thm}
\label{thm: energy dissipation modified model semi}
 Let $(\bu^h, p^h, \phi^h, v^h)$ be a smooth solution of the weak form of incompressible Navier-Stokes equations with surface tension (\ref{eq: weak form semi-discrete0}).
 The formulation (\ref{eq: weak form semi-discrete0}) has the properties:
 \begin{enumerate}
  \item The formulation satisfies the maximum principle for the density, i.e. without loss of generality we assume that $\rho_2 \leq \rho_1$ and then have:
  \begin{align}
     \rho_2 \leq \rho(\phi^h) \leq \rho_1.
  \end{align}
  \item The formulation is divergence-free as a distribution:
  \begin{align}\label{prop 2 semi-discrete}
    \nabla \cdot \bu^h \equiv 0.
  \end{align}
  \item The formulation satisfies the dissipation inequality:
 \begin{align}
   \frac{{\rm d}}{{\rm d}t} \mathscr{E}^h =- \left(\nabla \bu^h,\boldsymbol{\tau}\left(\bu^h, \phi^h\right)\right)_\Omega - \displaystyle\sum_K \left(\nabla \bu^h, \theta_K \nabla \bu^h\right)_{\Omega_K}  \leq 0.
 \end{align}
 \end{enumerate}
\end{thm}
The proof of \cref{thm: energy dissipation modified model semi} goes along the same lines as that of \cref{thm: energy dissipation modified model cont}.
\begin{proof}
1 $\&$ 2. The first two properties are directly inherited from the continuous case. Note that the weighting function choice for the second property is in general not permitted. The specific NURBS function spaces proposed by Evans et al. \cite{Evans13steadyNS, Evans13unsteadyNS} do allow this selection.\\

3. Selection of the weights 
$\psi^h=v^h$ in (\ref{weak form LS phi cont semi}) and $\zeta^h = -\partial_t \phi^h$ in (\ref{weak form v cont semi}) gives:
\begin{subequations}\label{eq: cont proof weights 1}
\begin{align}
   \left(v^h, \partial_t \phi^h + \bu^h \cdot \nabla \phi^h\right)_\Omega +  \displaystyle\sum_K \left(\tau_K \bu^h \cdot \nabla v^h, \mathscr{R}_I\phi^h\right)_{\Omega_K}  &~= 0,\\
  -\left( \partial_t \phi^h, v^h+\varrho^h\frac{\|\bu\|_2^2}{2}- \frac{1}{\mathbb{F}{\rm r}^2}\varrho^h y\right)_\Omega +\left(\frac{1}{\mathbb{W}{\rm e}} \delta(\phi^h) \dfrac{\nabla \phi^h}{\|\nabla \phi^h\|_{\epsilon,2}}, \nabla \partial_t \phi^h\right)_\Omega &\nn\\
  + \left(\frac{1}{\mathbb{W}{\rm e}} \|\nabla \phi^h \|_{\epsilon,2} \delta'(\phi^h), \partial_t \phi^h\right)_\Omega &~=0.
\end{align}
\end{subequations}
Addition of the equations (\ref{eq: cont proof weights 1}) results in:
\begin{align}\label{eq: cont combining}
 (v^h, \bu^h \cdot \nabla  \phi^h)_\Omega- \left( \frac{\varrho^h}{2}\|\bu^h\|_2^2, \partial_t \phi^h \right)_\Omega 
 +\frac{1}{\mathbb{F}{\rm r}^2}\left( \partial_t \phi^h,  \varrho^h y\right)_\Omega  +\left(\frac{1}{\mathbb{W}{\rm e}} \delta(\phi^h) \dfrac{\nabla \phi^h}{\|\nabla \phi^h\|_{\epsilon,2}}, \nabla \partial_t \phi^h\right)_\Omega & \nn\\
 + \left(\frac{1}{\mathbb{W}{\rm e}} \|\nabla \phi^h \|_{\epsilon,2} \delta'(\phi^h), \partial_t \phi^h\right)_\Omega-  \displaystyle\sum_K \left(\tau_K \bu^h \cdot \nabla v^h, \mathscr{R}_{I} \phi^h\right)_{\Omega_K} &~= 0.
\end{align}
By performing integration by parts we obtain:
\begin{align}\label{eq: cont combining IP}
  \left(\partial_t \phi^h, -\frac{\varrho^h}{2}\|\bu^h\|_2^2 +\varrho^h \frac{1}{\mathbb{F}{\rm r}^2} y -\frac{1}{\mathbb{W}{\rm e}} \delta(\phi^h) \nabla \cdot \right. &\left. \left(\dfrac{\nabla \phi^h}{\|\nabla \phi^h\|_{\epsilon,2}}\right) + \frac{1}{\mathbb{W}{\rm e}} \delta'(\phi) \dfrac{\epsilon^2}{\|\nabla \phi\|_{\epsilon,2}}\right)_\Omega \nn\\[8pt]
  =& -(v^h, \bu^h \cdot \nabla  \phi^h)_\Omega -  \displaystyle\sum_K \left(\tau_K \bu^h \cdot \nabla v^h, \mathscr{R}_{I} \phi^h \right)_{\Omega_K}.
\end{align}
Recognize $ \dfrac{\updelta \mathcal{H}^h}{\updelta \phi^h}$ on the left-hand side to arrive at:
\begin{align}\label{eq: cont combining IP}
 \dfrac{\updelta \mathscr{E}^h}{\updelta \phi^h}\left[\dfrac{\partial \phi^h}{\partial t}\right]:=\left(\dfrac{\partial \phi^h}{\partial t},\dfrac{\updelta \mathcal{H}^h}{\updelta \phi^h} \right)_\Omega  =&~ -(v^h, \bu^h \cdot \nabla  \phi^h)_\Omega\nn\\
 &~-  \displaystyle\sum_K \left(\tau_K \bu^h \cdot \nabla v^h, \mathscr{R}_{I} \phi^h\right)_{\Omega_K}.
\end{align}
Next we take $\bw^h = \bu^h$ in (\ref{weak form LS mom cont}) to get:
\begin{align}\label{eq: cont combining 2}
 (\bu^h,  \partial_t (\rho^h  \bu^h))_\Omega   -(\nabla \bu^h, \rho^h \bu^h \otimes \bu^h)_\Omega - (\bu^h, \tfrac{1}{2}\|\bu^h\|_2^2\varrho^h\nabla \phi^h)_\Omega  - (\nabla \cdot \bu^h, p^h)_\Omega & \nn\\
 + (\nabla \bu^h, \boldsymbol{\tau}(\bu^h,\phi^h))_\Omega  -\left(\bu^h, v^h \nabla \phi^h\right)_\Omega+ \frac{1}{\mathbb{F}{\rm r}^2}\left(\bu^h, \varrho^h y \nabla \phi^h \right)_\Omega + \frac{1}{\mathbb{F}{\rm r}^2}(\bu^h, \rho^h \boldsymbol{\jmath})_\Omega & \nn\\
 +  \displaystyle\sum_K \left(\nabla \bu^h, \theta_K \nabla \bu^h \right)_{\Omega_K} -  \displaystyle\sum_K \left(\tau_K \bu^h \cdot \nabla v^h, \mathscr{R}_{I} \phi^h\right)_{\Omega_K} &=0. 
\end{align}
Similar as in the continuous case, we have the identities:
\begin{subequations}\label{eq: props}
\begin{align}
 -(\nabla \bu^h, \rho^h \bu^h \otimes \bu^h)_\Omega - (\bu^h, \tfrac{1}{2}\|\bu^h\|_2^2\varrho^h\nabla \phi^h)_\Omega &~=0 \\[6pt]
 - (\nabla \cdot \bu^h, p^h)_\Omega &~=0, \\
   \frac{1}{\mathbb{F}{\rm r}^2}\left(\bu^h, \varrho^h y \nabla \phi^h \right)_\Omega+\frac{1}{\mathbb{F}{\rm r}^2}(\bu^h, \rho^h \boldsymbol{\jmath})_\Omega &~= 0.
\end{align}
\end{subequations}
Noting that $ \dfrac{\updelta \mathcal{H}^h}{\updelta (\rho^h \bu^h)} = (\bu^h)^T$ and employing (\ref{eq: props}) we arrive at:
\begin{align}\label{eq: cont combining 2 a}
\dfrac{\updelta \mathscr{E}^h}{\updelta (\rho^h\bu^h)}\left[\dfrac{\partial (\rho^h \bu^h)}{\partial t}\right]:=\left(\dfrac{\partial (\rho^h\bu^h)}{\partial t},\dfrac{\updelta \mathcal{H}^h}{\updelta (\rho^h\bu^h)} \right)_\Omega =&~ - (\nabla \bu^h, \boldsymbol{\tau}(\bu^h,\phi^h))_\Omega  +\left(\bu^h, v^h \nabla \phi^h\right)_\Omega \nn\\
&~- \displaystyle\sum_K \left(\nabla \bw^h, \theta_K \nabla \bu^h \right)_{\Omega_K}\nn\\
&~+ \displaystyle\sum_K \left(\tau_K \bu^h \cdot \nabla v^h, \mathscr{R}_{I} \phi^h\right)_{\Omega_K}.
\end{align}
The superposition of (\ref{eq: cont combining IP}) and (\ref{eq: cont combining 2 a}) yields:
\begin{align}\label{eq: cont combining 3}
\dt \mathscr{E}^h = \dfrac{\updelta \mathscr{E}^h}{\updelta \phi^h}\left[\dfrac{\partial \phi^h}{\partial t}\right] + \dfrac{\updelta \mathscr{E}^h}{\updelta (\rho^h\bu^h)}\left[\dfrac{\partial (\rho^h \bu^h)}{\partial t}\right]
= - (\nabla \bu^h, \boldsymbol{\tau}(\bu^h,\phi^h))_\Omega-  \displaystyle\sum_K \left(\nabla \bu^h, \theta_K \nabla \bu^h \right)_{\Omega_K}.
\end{align}
\end{proof}

\section{Energy-dissipative temporal discretization}\label{sec: Time-discrete formulation}
In this Section we present the energy-stable time-integration methodology. We present a modified version of the mid-point time-discretization method. First we introduce some required notation in \cref{subsec: time not} and then explain the time-discretization of the terms that differ from the standard midpoint rule in \cref{subsec: id temp1,subsec: id temp2}. The eventual method is presented in \cref{subsec: fully disc}.

The simplest fully-discrete algorithm would be to start from the semi-discrete version of (\ref{eq: weak form semi-discrete0}) and then discretize in time using the second-order mid-point time-discretization.
An important observation is that this approach does not lead to a provable energy-dissipative formulation, see \ref{appendix: Standard mid-point Galerkin discretizations}. We note that this is in contrast to the single-fluid case (in absence of surface tension effects).

In the following we present our strategy to arrive at a provable energy-dissipative formulation.
Our approach is to mirror the semi-discrete case as closely as possible. We first focus on the terms that are directly associated with temporal derivatives of the energies and then treat the remaining terms.

\subsection{Notation}\label{subsec: time not}
Let us divide the time-interval $\mathcal{T}$ into sub-intervals $\mathcal{T}_n=(t_n,t_{n+1})$ (with $n=0,1,...,N$) and denote the size of interval $\mathcal{T}_n$ as time-step $\Delta t_n=t_{n+1}-t_n$. We use subscripts to indicate the time-level of the unknown quantities, i.e. the unknowns at time-level $n$ are $\bu^h_n, p^h_n, \phi^h_n$ and $v^h_n$. 
Lastly, we denote the intermediate time-levels and associated time derivatives as:
\begin{subequations}\label{eq:  intermediate time-levels}
\label{eq:weak}
\begin{align}
  \bu^h_{n+1/2} =&~ \tfrac{1}{2}(\bu^h_n+\bu^h_{n+1}), & \frac{1}{\Delta t_n}[\![\bu^h]\!]_{n+1/2} =&~ \frac{1}{\Delta t_n}(\bu^h_{n+1}-\bu^h_{n}),\\
  \phi^h_{n+1/2} =&~ \tfrac{1}{2}(\phi^h_n+\phi^h_{n+1}), \\
 \rho^h_{n+1/2} =&~ \rho(\phi^h_{n+1/2}), & \frac{1}{\Delta t_n}[\![\rho^h]\!]_{n+1/2} =&~ \frac{1}{\Delta t_n}(\rho^h_{n+1}-\rho^h_{n}),\\
 &&\frac{1}{\Delta t_n}[\![\rho^h\bu^h]\!]_{n+1/2} =&~ \frac{1}{\Delta t_n}\left(\rho^h_{n+1}  \bu^h_{n+1}-\rho^h_{n}  \bu^h_{n}\right),\\
 \mu^h_{n+1/2}  =&~ \mu(\phi^h_{n+1/2}),&
\end{align}
\end{subequations}
where $\rho^h_{n} = \rho(\phi^h_{n})$.

\subsection{Identification energy evolution terms}\label{subsec: id temp1}
In order to identify the energy evolution terms we wish to have the fully discrete version of
\begin{subequations}\label{eq: identify}
\begin{align}
  \dt \mathscr{E}^{K,h} =&~ (\bw^h, \partial_t (\rho^h \bu^h))_\Omega + (\zeta^h, \varrho(\phi^h)\tfrac{1}{2}\|\bu^h\|_2^2)_\Omega, \quad &\text{with }\bw^h = \bu^h, \text{ and }\zeta^h = -\partial_t \phi^h,\\
  \dt \mathscr{E}^{G,h} =&~ -\frac{1}{\mathbb{F}{\rm r}^2}(\zeta^h, \varrho(\phi^h) y)_\Omega, \quad &\text{with }\zeta^h = -\partial_t \phi^h,\\
  \dt \mathscr{E}^{S,h} =&~ -\left(\frac{1}{\mathbb{W}{\rm e}} \delta(\phi^h) \dfrac{\nabla \phi^h}{\|\nabla \phi^h\|_{\epsilon,2}}, \nabla \zeta^h\right)_\Omega \nn\\
  &~- \left(\frac{1}{\mathbb{W}{\rm e}} \|\nabla \phi^h \|_{\epsilon,2} \delta'(\phi^h), \zeta^h\right)_\Omega , \quad &\text{with }\zeta^h = -\partial_t \phi^h,\\
\end{align}
\end{subequations}

Three issues arise: (i) the approximation of the internal energy density $\tfrac{1}{2}\|\bu^h\|^2_2$ in the additional equation (\ref{weak form v cont semi}), (ii) the approximation of the interface density jump term $\varrho^h$ and (iii) the approximation of the surface tension contribution.\\
In the following we discuss the considerations for their time-discretization.

(i) The first matter is resolved when taking a shift in the time-levels in the energy density, analogously as in Liu et al. \cite{liu2013functional}, i.e. we take $\tfrac{1}{2}\bu^h_n\cdot \bu^h_{n+1}$ in the additional equation.\\

(ii) Concerning the second problem, we require a stable time-discretization of $\varrho^h$ such that the approximation of $\varrho^h \partial_t \phi^h$ equals that of $\partial_t \rho^h$. This suggests to approximate $\varrho^h$ at the intermediate time level $t_{n+1/2}$ as
\begin{align}\label{eq: approx varsigma}
 \varrho^h(t_{n+1/2})  \approx \varrho_{F,n+1/2}^h :=\dfrac{\rho(\phi_{n+1}^h)-\rho(\phi_{n}^h)}{\phi_{n+1}^h-\phi_{n}^h},
\end{align}
such that
\begin{align}\label{eq: relation varsigma}
  \dfrac{[\![\rho^h]\!]_{n+1}}{\Delta t_n} = \varrho_{F,n+1/2}^h \dfrac{[\![\phi^h]\!]_{n+1/2}}{\Delta t_n}.
\end{align}
Unfortunately, the approximation (\ref{eq: approx varsigma}) is not defined when $\phi_{n+1}^h = \phi_{n}^h$.
If $\varrho$ is a polynomial function of $\phi$ we may use truncated Taylor expansions around $\phi^h_{n+1/2}$ to find:
\begin{align}\label{eq: truncated series}
 \varrho_{F,n+1/2}^h = \sum_{j=0}^M \dfrac{1}{2^{2j}(2j+1)!}\varrho^{(2j)}(\phi^h_{n+1/2}) [\![\phi^h]\!]_{n+1/2}^{2j},
\end{align}
where $M$ chosen such that latter terms in the sum vanish and where we use the notation $h^{(m)}(x) = {\rm d}^mh/{\rm d} x^m$ for the $m$-th derivative of function $h=h(x)$.
This motivates to use a (piece-wise) higher-order polynomial for $\varrho^h$.
We define the regularized Heaviside as
\begin{align}\label{def reg heavi}
 H_\eps(\phi^h_n) = H^p(\phi^h_n/\eps)
\end{align}
where $H^p=H^p(\upphi)$ is the piece-wise polynomial regularization:
\begin{align}
 H^p=H^p(\upphi) = \left\{ \begin{matrix} 0 & ~~~~~~~~~~~~~\upphi < -1, \\[6pt]
                                                 -\frac{3}{4} \upphi^5 - \frac{5}{2}\upphi^4 - \frac{5}{2} \upphi^3 +\frac{5}{4}\upphi +\frac{1}{2}   &~~ -1 \leq \upphi < 0, \\[6pt]
                                                 -\frac{3}{4} \upphi^5 + \frac{5}{2}\upphi^4 - \frac{5}{2} \upphi^3 +\frac{5}{4}\upphi +\frac{1}{2}  & ~~~~~0 \leq \upphi < 1,\\[6pt]
                                                    1 & 1 \leq \upphi.
                   \end{matrix} 
\right.
\end{align}
This function is $\mathcal{C}^3$-continuous at $\upphi = 0$ and $\mathcal{C}^3$-continuous at $\upphi = -1, \upphi = 1$. Furthermore, we base the regularization of Dirac on the Heaviside, i.e. we have $\delta_\eps(\phi^h) = H_\eps^{(1)}(\phi^h)$.
\begin{rmk}
  The regularized Dirac delta $\delta_\eps(\phi^h)$ has area $1$.
\end{rmk}
\begin{rmk}
  If $\varrho^h$ is non-polynomial one may use perturbed trapezoidal rules. In case of positive higher-order derivatives this leads to a stable approximation for $\varrho^h$.
\end{rmk}
\begin{rmk}
 This regularization closely resembles the popular goniometric regularization:
\begin{align}
 H^g=H^g(\upphi) = \left\{ \begin{matrix} 0 & ~~~~~~~~~~~~~\upphi < -1, \\[6pt]
                                                 \frac{1}{2} \left( 1 + \upphi + \frac{1}{\pi} \sin(\pi \upphi) \right)    &~~ -1 \leq \upphi < 1, \\[6pt]
                                                    1 & 1 \leq \upphi.
                   \end{matrix} 
\right.
\end{align}
\cref{fig:Heavisides} shows the polynomial regularization $H^p=H^p(\upphi)$, the goniometric regularization $H^g=H^g(\upphi)$ and their first two derivatives. At $\upphi = -1$ and $\upphi = 1$ the goniometric regularization is $\mathcal{C}^2$-continuous where $H^p=H^p(\upphi)$ is $\mathcal{C}^3$-continuous.

\begin{figure}[h!]
\begin{subfigure}{0.33\textwidth}
\centering
\includegraphics[width=0.95\textwidth]{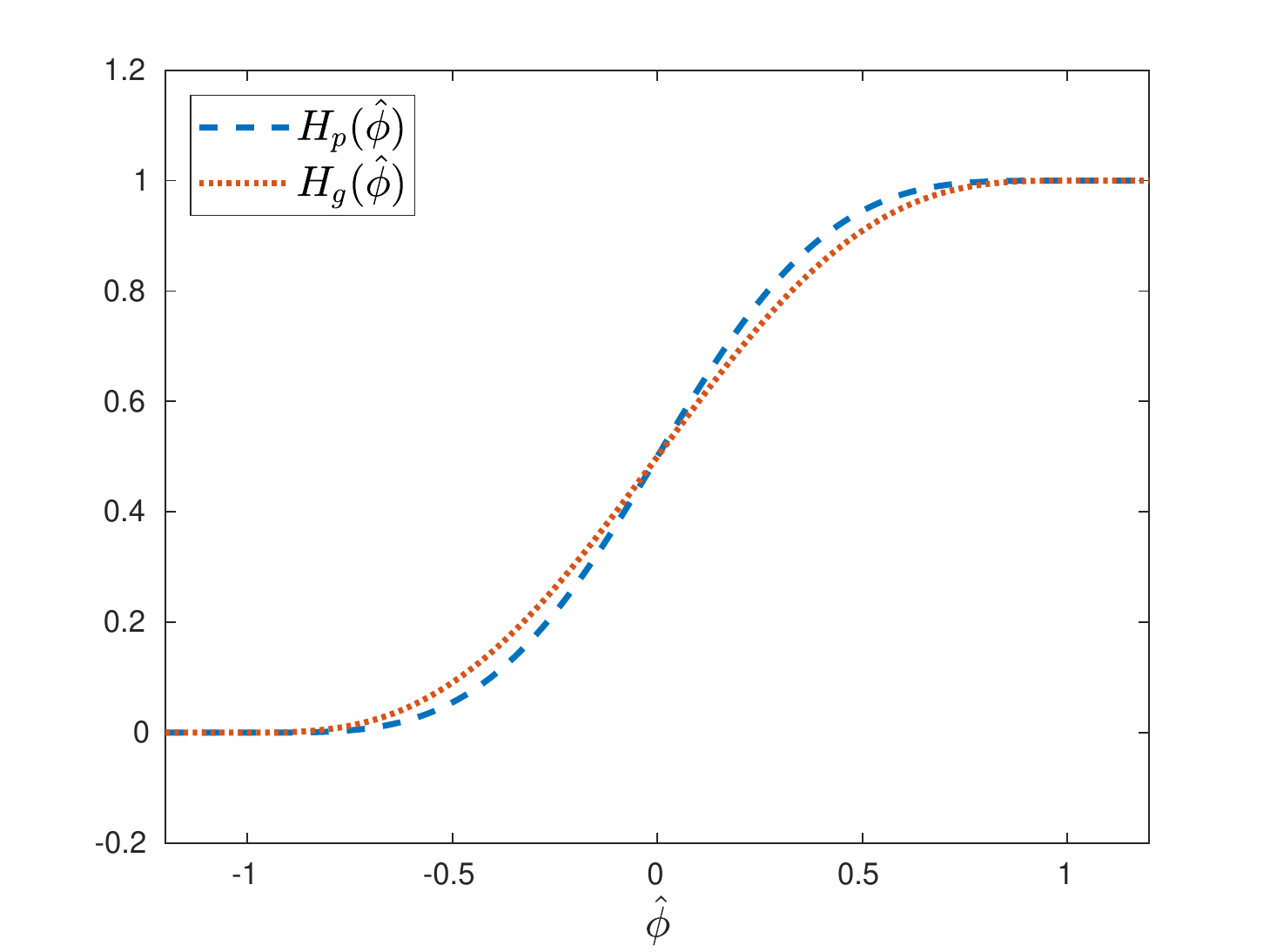}
\caption{Regularized heaviside}
\end{subfigure}
\begin{subfigure}{0.33\textwidth}
\centering
\includegraphics[width=0.95\textwidth]{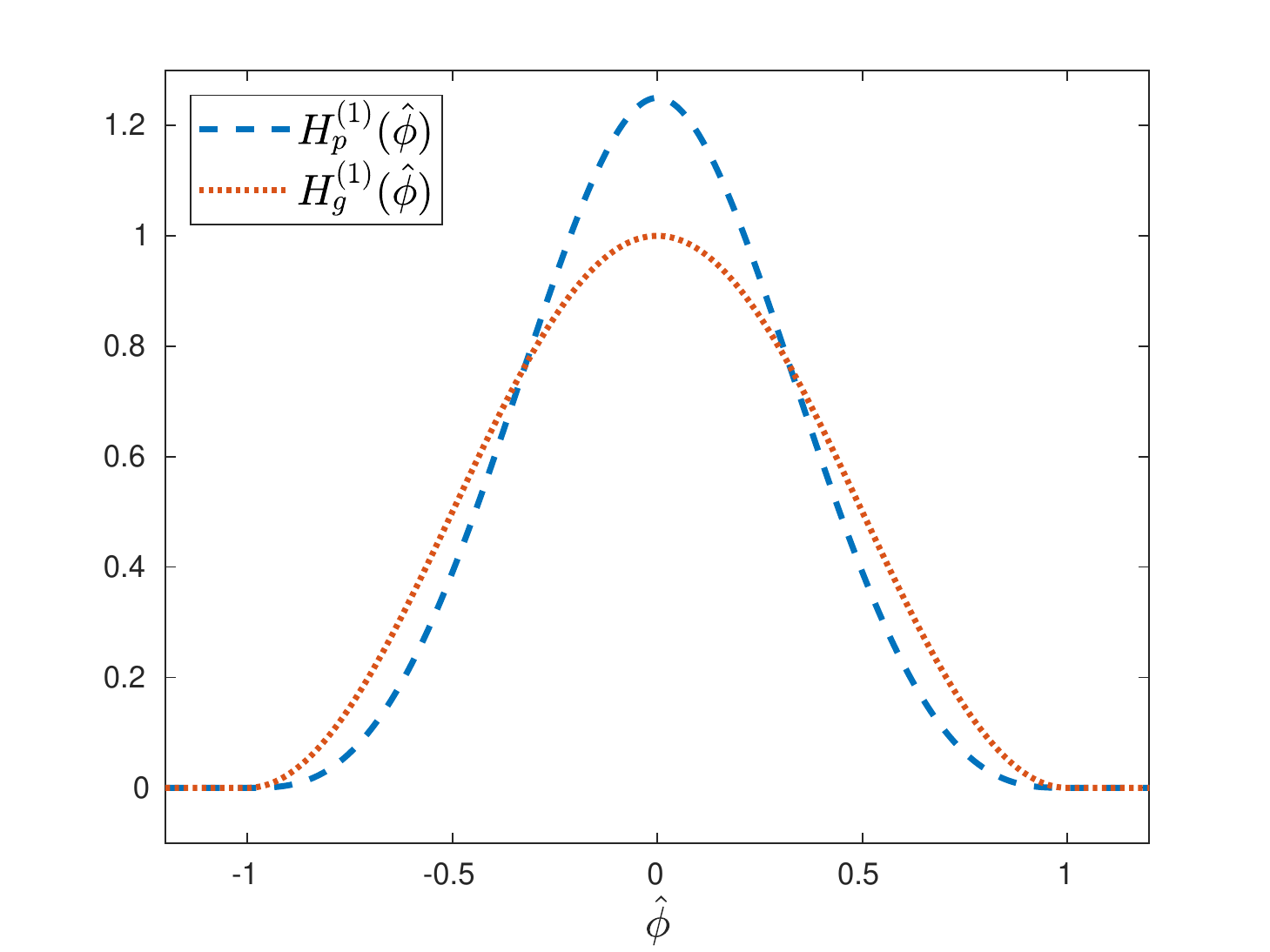}
\caption{Regularized Dirac delta}
\end{subfigure}
\begin{subfigure}{0.33\textwidth}
\centering
\includegraphics[width=0.95\textwidth]{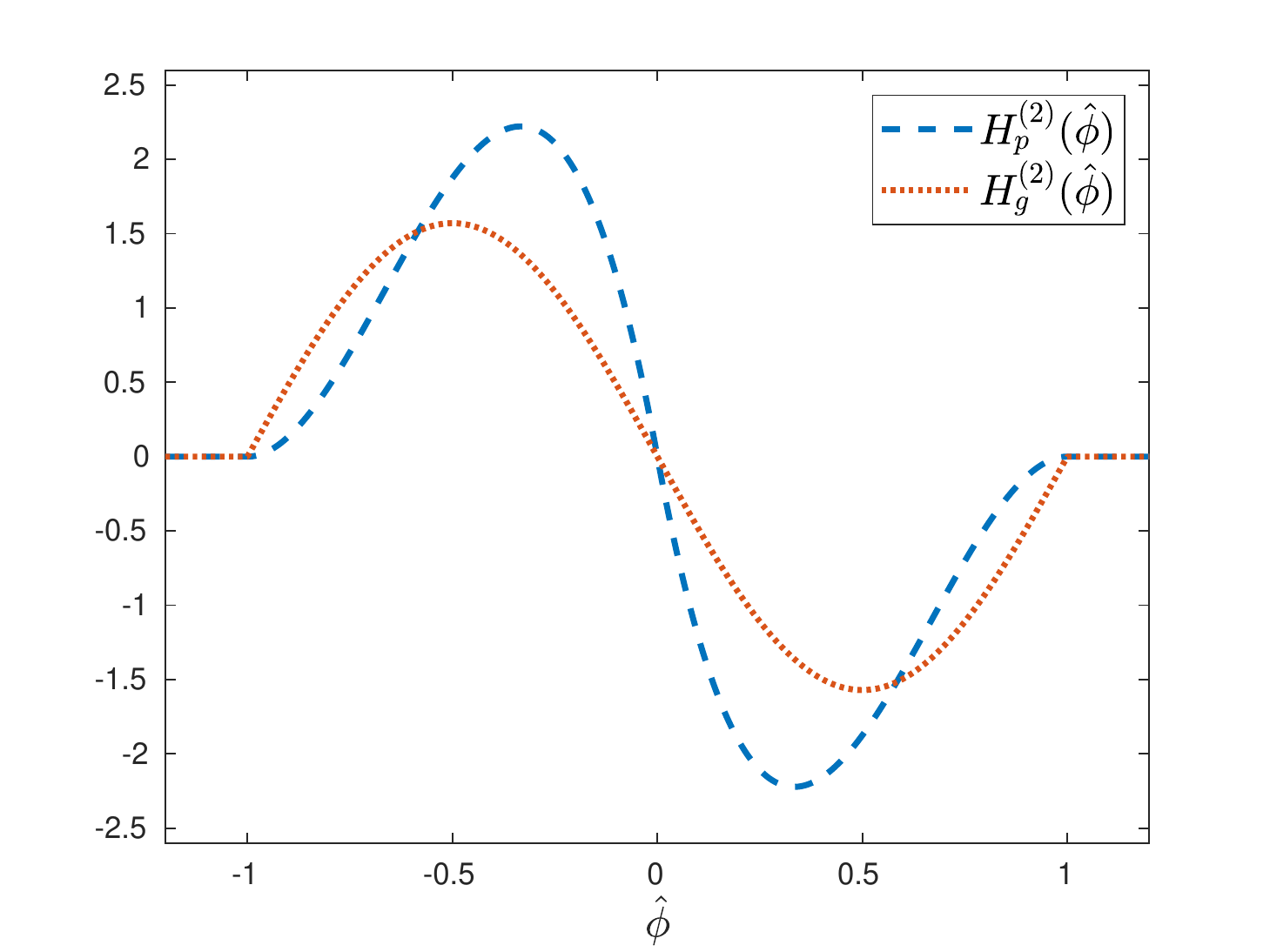}
\caption{Derivative regularized heaviside}
\end{subfigure}
\caption{Comparison of polynomial and goniometric regularization of the Heaviside.}
\label{fig:Heavisides}
\end{figure}
\end{rmk}

Since $\varrho^h(t_{n+1/2})$ is a piece-wise polynomial (\ref{eq: truncated series}) only holds if $\phi^h_n$ and $\phi^h_{n+1}$ are in the same `piece'. In the other case we have $\phi^h_n \neq \phi^h_{n+1}$ and thus we may use $\varrho_{F,n+1/2}^h$.  
Thus, to define $\varrho^h(t_{n+1/2})$ in the auxiliary equation we distinguish the cases
\begin{enumerate}
 \item [1.] $\phi^h_{n}$ and $\phi^h_{n+1}$ are in the same `piece' of the polynomial $H_\eps$
 \item [2.] $\phi^h_{n}$ and $\phi^h_{n+1}$ are in another piece of the polynomial $H_\eps$.
\end{enumerate}
In the first case employ the truncated series (\ref{eq: truncated series}) whereas in the second case we directly employ the left-hand side of (\ref{eq: truncated series}):
\begin{align}\label{eq: def varsigma}
  \varrho^h(t_{n+1/2})\approx \varrho^h_{a,n+1/2} := \left\{\begin{matrix}
                          \varrho_{T, n+1/2}^h \quad & \text{in case 1}\\[6pt]
                          \varrho_{F, n+1/2}^h \quad & ~\text{in case 2,}
                          \end{matrix}
                        \right.
\end{align}
with Taylor series representation:
\begin{align}
 \varrho_{T, n+1/2}^h := [\![\![\rho]\!]\!]\left( H_\eps^{(1)}(\phi^h_{n+1/2}) + \frac{1}{24}H_\eps^{(3)}(\phi^h_{n+1/2}) [\![\phi^h]\!]_{n+1/2}^{2} + \frac{1}{1920}H_\eps^{(5)}(\phi^h_{n+1/2}) [\![\phi^h]\!]_{n+1/2}^{4}\right).
\end{align}
Definition (\ref{eq: def varsigma}) satisfies condition (\ref{eq: relation varsigma}):
\begin{align}\label{eq: condition varsigma}
 \dfrac{[\![\rho^h]\!]_{n+1}}{\Delta t_n} = \varrho_{a,n+1/2}^h \dfrac{[\![\phi^h]\!]_{n+1/2}}{\Delta t_n}.
\end{align}

(iii) We now turn our focus to the surface tension contribution, which writes in semi-discrete form:
\begin{align}\label{eq: chain rule surface term}
-\left(\frac{1}{\mathbb{W}{\rm e}} \delta(\phi^h) \dfrac{\nabla \phi^h}{\|\nabla \phi^h\|_{\epsilon,2}}, \nabla \zeta^h\right)_\Omega - \left(\frac{1}{\mathbb{W}{\rm e}} \|\nabla \phi^h \|_{\epsilon,2} \delta'(\phi^h), \zeta^h\right)_\Omega.
\end{align}
Recall that in the semi-discrete form the surface energy evolution follows when substituting $\zeta^h = - \partial_t \phi^h$:
\begin{align}\label{eq: surface energy semi-discrete}
  \left(\frac{1}{\mathbb{W}{\rm e}} \|\nabla \phi^h \|_{\epsilon,2} \delta'(\phi^h), \partial_t \phi^h\right)_\Omega + \left(\frac{1}{\mathbb{W}{\rm e}} \delta(\phi^h) \dfrac{\nabla \phi^h}{\|\nabla \phi^h\|_{\epsilon,2}}, \nabla \partial_t \phi^h\right)_\Omega =&~ \nn\\
 \frac{1}{\mathbb{W}{\rm e}}  \left( \partial_t \delta(\phi^h), \|\nabla\phi^h\|_{\epsilon,2}\right)_\Omega + \frac{1}{\mathbb{W}{\rm e}} \left(\delta(\phi^h),  \partial_t \|\nabla \phi^h\|_{\epsilon,2} \right)_\Omega =&~  \nn\\
 \dt \left(\delta(\phi^h), \frac{1}{\mathbb{W}{\rm e}} \|\nabla \phi^h\|_{\epsilon,2} \right)_\Omega =&~
 \dt \mathscr{E}_S^h.
\end{align}
Here we have utilized following identities:\\
\begin{subequations}
\textbullet~ for the first term:
\begin{align}
 &{\rm (I)} \quad \partial_t \phi^h \delta'(\phi^h) = \partial_t \delta(\phi^h), \quad \quad \quad \quad &&&&&&&&&& \label{eq: identities: 1} 
\end{align}
\textbullet~ for the second term:
\begin{align}
  &{\rm (II)} \quad \nabla \partial_t \phi^h \cdot \dfrac{\nabla\phi^h}{\|\nabla \phi^h\|_{\epsilon,2}} = \partial_t \|\nabla \phi^h\|_{\epsilon,2},\quad \quad \quad \quad &&&&&&&&&& \label{eq: identities: 3}
\end{align}
\textbullet~ and for combining the terms:
\begin{align}
  &{\rm (III)} \quad \|\nabla \phi^h\|_{\epsilon,2} \partial_t \delta(\phi^h) + \delta(\phi) \partial_t \|\nabla \phi^h\|_{\epsilon,2} = \partial_t \left(\delta(\phi^h) \|\nabla \phi^h\|_{\epsilon,2}\right).\quad  &&&&&& \label{eq: identities: 4}
\end{align}
\end{subequations}
We wish to follow the same steps in the fully-discrete sense.
However, these identities are not directly guaranteed in a fully discrete sense.
In the following we describe the fully-discrete approximation of each of the three terms in (\ref{eq: chain rule surface term}), i.e. $\delta'(\phi^h), \delta(\phi^h)$ and $\nabla\phi^h/\|\nabla \phi^h\|_{\epsilon,2}$, that complies with these identities.
To that purpose we introduce the mid-point approximation of the time-derivative.
\begin{prop}\label{prop: product-rule midpoint}
  The mid-point approximation of the time-derivative satisfies the product-rule in the following sense:
  \begin{align}
    \dfrac{[\![\mathbf{a}^h\cdot\mathbf{b}^h]\!]_{n+1/2}}{\Delta t_n} = \mathbf{a}^h_{n+1/2} \cdot \dfrac{[\![\mathbf{b}^h]\!]_{n+1/2}}{\Delta t_n} + \dfrac{[\![\mathbf{a}^h]\!]_{n+1/2}}{\Delta t_n} \cdot \mathbf{b}^h_{n+1/2},
  \end{align}
  where $\mathbf{a}^h$ and $\mathbf{b}^h$ are scalar or vector fields.
\end{prop}

(III) We start off with the last identity (\ref{eq: identities: 4}).
The fully-discrete version of the product rule in (\ref{eq: identities: 4}) follows from \cref{prop: product-rule midpoint}:
  \begin{align}\label{eq: mid point surface energy}
    \dfrac{[\![\delta(\phi^h)\|\nabla \phi^h \|_{\epsilon,2}]\!]_{n+1/2}}{\Delta t_n} =&~ \dfrac{[\![\delta(\phi^h)]\!]_{n+1/2}}{\Delta t_n} \left(\|\nabla \phi^h \|_{\epsilon,2}\right)_{n+1/2} +  \left(\delta(\phi^h)\right)_{n+1/2}  \dfrac{[\![\|\nabla \phi^h \|_{\epsilon,2}]\!]_{n+1/2}}{\Delta t_n}.
  \end{align}
This implies that we require the approximation:
\begin{subequations}\label{eq: identify2}
\begin{align}
 \delta(\phi^h)(t_{n+1/2}) \approx&~ (\delta(\phi^h))_{n+1/2}, \nn\\
  \|\nabla \phi^h\|_{\epsilon,2}(t_{n+1/2}) \approx&~  \left(\|\nabla \phi^h\|_{\epsilon,2}\right)_{n+1/2}.
 \end{align}
 \end{subequations}
We now aim to identify the first and the second term on the right-hand side of (\ref{eq: mid point surface energy}) with first and second term on the right-hand side of (\ref{eq: surface energy semi-discrete}) respectively.

(I) To identify the first term we require, in a similar fashion as for $\varrho$, the approximation $\varsigma_{n+1/2}^h\approx \delta'(\phi^h)(t_{n+1/2})$ to satisfy:
\begin{align}\label{eq: approx 1}
 \dfrac{[\![\delta(\phi^h)]\!]_{n+1/2}}{\Delta t_n} = \varsigma_{n+1/2}^h \dfrac{[\![\phi^h]\!]_{n+1/2}}{\Delta t_n}. 
\end{align}
To this purpose we define
\begin{align}\label{eq: def deltap}
  \varsigma_{n+1/2}^h := \left\{\begin{matrix}
                           \varsigma_{T,n+1/2}^h \quad & \text{in case 1}\\[6pt]
                          \varsigma_{F,n+1/2}^h \quad & ~\text{in case 2},
                          \end{matrix} 
                        \right.
\end{align}
with truncated series:
\begin{align}\label{eq: def deltap}
  \varsigma_{T,n+1/2}^h := \delta^{(1)}_\eps(\phi^h_{n+1/2}) + \dfrac{[\![\phi]\!]_{n+1/2}^2}{24}\delta_\eps^{(3)}(\phi^h_{n+1/2}),
\end{align}
and the fraction:
\begin{align}
   \varsigma_{F,n+1/2}^h := \dfrac{\delta_\eps(\phi_{n+1}^h)-\delta_\eps(\phi_{n}^h)}{\phi_{n+1}^h-\phi_{n}^h}.
\end{align}

(II) We take in (\ref{eq: identities: 3}) the approximation:
\begin{align}
  \left(\dfrac{\nabla \phi^h}{\|\nabla \phi^h\|_{\epsilon,2}}\right)(t_{n+1/2})  \approx \dfrac{\left(\nabla \phi^h\right)_{n+1/2}}{\left(\|\nabla \phi^h\|_{\epsilon,2}\right)_{n+1/2}}=\dfrac{\nabla \phi^h_{n+1} + \nabla \phi^h_n}{\|\nabla \phi^h_{n+1}\|_{\epsilon,2} + \|\nabla \phi^h_{n}\|_{\epsilon,2}},
\end{align}
such that (II) is satisfied in a fully-discrete sense:
\begin{align}\label{eq: approx 3}
  \nabla \dfrac{[\![\phi^h]\!]_{n+1/2}}{\Delta t_n}\cdot\dfrac{\left(\nabla \phi^h\right)_{n+1/2}}{\left(\|\nabla \phi^h\|_{\epsilon,2}\right)_{n+1/2}} = \dfrac{\|\nabla \phi^h_{n+1}\|_{\epsilon,2} - \|\nabla \phi^h_{n}\|_{\epsilon,2}}{\Delta t}.
\end{align}

\subsection{Discretization other terms}\label{subsec: id temp2}
We discretize the continuity equation using the mid-point rule, i.e.
\begin{align}
 \left( q^h, \nabla\cdot \bu^h_{n+1/2} \right)_\Omega = 0,
\end{align}
which implies pointwise divergence-free solutions on a fully-discrete level.

Next, we require the fully-discrete version of the identities:
\begin{subequations}\label{eq: identify3}
\begin{align}
  -(\nabla \bu^h, \rho^h \bu^h \otimes \bu^h)_\Omega - (\bu^h, \tfrac{1}{2}\|\bu^h\|_2^2\varrho(\phi^h)\nabla \phi^h)_\Omega = 0,\\
  +\frac{1}{\mathbb{F}{\rm r}^2}(\bu^h, \rho^h \boldsymbol{\jmath})_\Omega +\frac{1}{\mathbb{F}{\rm r}^2}  (\bu^h, y\varrho(\phi^h)\nabla \phi^h)_\Omega = 0,
\end{align}
\end{subequations}
which make use of the pointwise divergence-free property. These identities are fulfilled when we have
\begin{align}
 \nabla \rho(\phi^h) = \varrho(\phi^h) \nabla \phi^h.
\end{align}
Applying the chain-rule implies that we can take as approximation in the momentum equation:
\begin{align}\label{eq: def varsigma2}
  \varrho^h(t_{n+1/2}) \approx \varrho_{m,n+1/2}^h:= [\![\![\rho]\!]\!] H_\eps'(\phi^h_{n+1/2}),
\end{align}
where the subscript $m$ refers to the momentum equation.
\begin{rmk}
Note that we employ two different approximations for $\varrho^h(t_{n+1/2})$, namely (\ref{eq: def varsigma}) in the additional equation (\ref{weak form v cont semi}) and (\ref{eq: def varsigma2}) in the momentum equation (\ref{weak form LS mom cont 2 semi}).
\end{rmk}
The remaining terms utilize the standard midpoint discretization.
\subsection{Fully-discrete energy-dissipative method}\label{subsec: fully disc}
We are now ready to present the fully-discrete energy-dissipative method:\\

\textit{Given $\bu_n^h, p_n^h, \phi^h_n$ and $v_n^h$, find $\bu_{n+1}^h, p_{n+1}^h, \phi_{n+1}^h$ and $v_{n+1}^h$ such that for all $(\bw^h, q^h, \psi^h, \zeta^h) \in \WW^h$:}
\begin{subequations}\label{eq: step 2}
\begin{align}
(\bw^h,  \dfrac{[\![\rho\bu]\!]_{n+1/2}}{\Delta t_n} )_\Omega -(\nabla \bw^h, \rho^h_{n+1/2} \bu^h_{n+1/2} \otimes \bu^h_{n+1/2})_\Omega&\nn\\
- (\nabla \cdot \bw^h, p^h_{n+1})_\Omega+ (\nabla \bw^h, \boldsymbol{\tau}(\bu^h_{n+1/2},\phi^h_{n+1/2}))_\Omega +\frac{1}{\mathbb{F}{\rm r}^2}(\bw^h, \rho^h_{n+1/2} \boldsymbol{\jmath})_\Omega& \nn\\
-\left(\bw^h, v^h_{n+1} \nabla \phi^h_{n+1/2}\right)_\Omega - \left(\bw^h, \varrho^h_{m,n+1/2}\left(\frac{\|\bu^h_{n+1/2}\|_2^2}{2} - \frac{1}{\mathbb{F}{\rm r}^2} y\right)\nabla \phi^h_{n+1/2} \right)_\Omega & \nn\\
 -  \displaystyle\sum_K \left(\tau_K \bw^h \cdot \nabla v_{n+1}^h, \mathscr{R}_I\phi^h_{n+1/2}\right)_{\Omega_K}&=~ 0, \label{weak form LS mom time}\\
 \left( q^h, \nabla \cdot \bu^h_{n+1/2} \right)_\Omega &=~ 0, \label{weak form LS continuity time}\\ 
 ( \psi^h, \dfrac{[\![\phi^h]\!]_{n+1/2}}{\Delta t_n} + \bu^h_{n+1/2} \cdot \nabla \phi^h_{n+1/2} )_\Omega +  \displaystyle\sum_K \left(\tau_K \bu_{n+1/2}^h \cdot \nabla \psi^h, \mathscr{R}_I\phi^h_{n+1/2} \right)_{\Omega_K} &=~ 0, \label{weak form LS phi cont time}\\
  \left(\zeta^h, v^h_{n+1}+\varrho^h_{a,n+1/2}\left(\tfrac{1}{2}\bu^h_{n+1}\cdot \bu^h_n-\frac{1}{\mathbb{F}{\rm r}^2} y\right)\right)_\Omega & \nn\\
  - \frac{1}{\mathbb{W}{\rm e}} \left(\zeta^h \varsigma^h_{n+1/2}, \left(\|\nabla \phi^h\|_{\epsilon,2}\right)_{n+1/2}\right)_\Omega- \frac{1}{\mathbb{W}{\rm e}}\left(\delta(\phi^h)_{n+1/2}\nabla \zeta^h,    \dfrac{\left(\nabla \phi^h\right)_{n+1/2}}{\left(\|\nabla \phi^h\|_{\epsilon,2}\right)_{n+1/2}}\right)_\Omega &=~ 0 \label{weak form v cont time}.
\end{align}
\end{subequations}

\begin{rmk}
  Due to \cref{prop: product-rule midpoint} the time-derivative in the momentum equation may be implemented as:
  \begin{align}\label{eq: mid point for momentum}
    \dfrac{[\![\rho^h\bu^h]\!]_{n+1/2}}{\Delta t_n} = \rho^h_{n+1/2}  \dfrac{[\![\bu^h]\!]_{n+1/2}}{\Delta t_n} + \dfrac{[\![\rho^h]\!]_{n+1/2}}{\Delta t_n}  \bu^h_{n+1/2}.
  \end{align}
\end{rmk}
\begin{thm}\label{thm: fully-discrete}
 The algorithm (\ref{eq: step 2}) has the properties:
 \begin{enumerate}
  \item The scheme satisfies the maximum principle for the density, i.e. without loss of generality we assume that $\rho_2 \leq \rho_1$ and then have:
  \begin{align}
     \rho_2 \leq \rho^h_{n} \leq \rho_1, \quad \text{for all} ~~n=0,1,...,N.
  \end{align}
  \item The scheme is divergence-free as a distribution:
  \begin{align}\label{prop 2}
    \nabla \cdot \bu^h_{n+1/2} \equiv 0.
  \end{align}
  \item The scheme satisfies the dissipation inequality:
   \begin{align}
   \frac{[\![\mathscr{E}^h]\!]_{n+1/2}}{\Delta t_n} =&~- \left(\nabla \bu^h_{n+1/2},\boldsymbol{\tau}(\bu^h_{n+1/2},\phi^h_{n+1/2})\right)_\Omega -\displaystyle\sum_K \left(\nabla \bu_{n+1/2}^h, \theta_K \nabla \bu_{n+1/2}^h\right)_{\Omega_K} \nn\\
   \leq&~ 0 , \quad \text{for all} ~~n=0,1,...,N.
 \end{align}
 \end{enumerate}
\end{thm}

\begin{proof}
1 $\&$ 2. Analogously to the semi-discrete case.\\

3. 
Selection of the weights $\psi^h = v_{n+1}^h$ in (\ref{weak form LS phi cont time}) and $\zeta^h = -[\![\phi^h]\!]_{n+1/2}/\Delta t_n$ in (\ref{weak form v cont time}) yields:
\begin{subequations}\label{eq: cont proof weights 1 time 0}
\begin{align}
    (v^h_{n+1}, \frac{[\![\phi^h]\!]_{n+1/2}}{\Delta t_n} + \bu^h_{n+1/2}\cdot \nabla \phi^h_{n+1/2} )_\Omega  
    +  \displaystyle\sum_K \left(\tau_K \bu_{n+1/2}^h \cdot \nabla v_{n+1}^h, \mathscr{R}_I\phi^h_{n+1/2}\right)_{\Omega_K}  &~= 0,\\    
  -\left( \frac{[\![\phi^h]\!]_{n+1/2}}{\Delta t_n}, v^h+\varrho^h_{a,n+1/2}\left(\tfrac{1}{2}\bu^h_{n+1}\cdot \bu^h_n-\frac{1}{\mathbb{F}{\rm r}^2} y\right)\right)_\Omega &\nn\\
  + \frac{1}{\mathbb{W}{\rm e}}\left(\frac{[\![\phi^h]\!]_{n+1/2}}{\Delta t_n} \varsigma_{n+1/2}^h, \left(\|\nabla \phi^h\|_{\epsilon,2}\right)_{n+1/2}\right)_\Omega & \nn\\
  + \frac{1}{\mathbb{W}{\rm e}}\left(\delta(\phi^h_{n+1/2})\nabla \frac{[\![\phi^h]\!]_{n+1/2}}{\Delta t_n},   \dfrac{ \left(\nabla \phi^h\right)_{n+1/2}}{ \left(\|\nabla \phi^h\|_{\epsilon,2}\right)_{n+1/2}}\right)_\Omega&~=0.
\end{align}
\end{subequations}
We add the equations (\ref{eq: cont proof weights 1 time 0}) and find:
\begin{align}\label{eq: cont combining time 1}
  (v^h_{n+1},  \bu^h_{n+1/2}\cdot \nabla \phi^h_{n+1/2} )_\Omega -\left( \frac{[\![\phi^h]\!]_{n+1/2}}{\Delta t_n},\tfrac{1}{2}\varrho^h_{a,n+1/2}\bu^h_{n+1}\cdot \bu^h_n\right)_\Omega   & \nn \\
  +\left( \frac{[\![\phi^h]\!]_{n+1/2}}{\Delta t_n},\varrho^h_{a,n+1/2} \frac{1}{\mathbb{F}{\rm r}^2}y \right)_\Omega +  \displaystyle\sum_K \left(\tau_K \bu_{n+1/2}^h \cdot \nabla v_{n+1}^h, \mathscr{R}_I\phi^h_{n+1/2}\right)_{\Omega_K} &\nn\\
  + \left(\frac{[\![\phi^h]\!]_{n+1/2}}{\Delta t_n} \varsigma_{n+1/2}^h, \frac{1}{\mathbb{W}{\rm e}} \left(\|\nabla \phi^h\|_{\epsilon,2}\right)_{n+1/2}\right)_\Omega & \nn\\
  + \left(\delta(\phi^h_{n+1/2})\nabla \frac{[\![\phi^h]\!]_{n+1/2}}{\Delta t_n},  \frac{1}{\mathbb{W}{\rm e}}  \dfrac{ \left(\nabla \phi^h\right)_{n+1/2}}{ \left(\|\nabla \phi^h\|_{\epsilon,2}\right)_{n+1/2}}\right)_\Omega&~=0.
\end{align}
Using (\ref{eq: condition varsigma}), (\ref{eq: mid point surface energy}), (\ref{eq: approx 1}) and (\ref{eq: approx 3}) we get
\begin{align}\label{eq: cont combining time 1}
   \left( \dfrac{[\![\rho^h]\!]_{n+1/2}}{\Delta t_n},-\tfrac{1}{2}\bu^h_{n+1}\cdot \bu^h_n+ \frac{1}{\mathbb{F}{\rm r}^2}y \right)_\Omega ~~~~& \nn \\
  + \left( \frac{1}{\mathbb{W}{\rm e}}, \frac{[\![\delta(\phi^h) \|\nabla \phi^h\|_{\epsilon,2}]\!]_{n+1/2}}{\Delta t_n} \right)_\Omega~ =& ~-(v^h_{n+1},  \bu^h_{n+1/2}\cdot \nabla \phi^h_{n+1/2} )_\Omega \nn\\
  &~-  \displaystyle\sum_K \left(\tau_K \bu_{n+1/2}^h \cdot \nabla v_{n+1}^h, \mathscr{R}_I\phi^h_{n+1/2}\right)_{\Omega_K}.
\end{align}
Next we take $\bw^h = \bu^h_{n+1/2}$ in (\ref{weak form LS mom time}) to get:
\begin{align}\label{eq: cont combining time 2 1}
 (\bu^h_{n+1/2},  \dfrac{[\![\rho^h \bu^h]\!]_{n+1/2}}{\Delta t_n})_\Omega ~~=&~~(\nabla \bu^h_{n+1/2}, \rho^h_{n+1/2} \bu^h_{n+1/2} \otimes \bu^h_{n+1/2})_\Omega \nn\\
 &~+ (\bu^h_{n+1/2}, \tfrac{1}{2}\|\bu^h_{n+1/2}\|_2^2\varrho^h_{m,n+1/2}\nabla \phi^h_{n+1/2})_\Omega\nn\\
 &~-\frac{1}{\mathbb{F}{\rm r}^2}(\bu^h_{n+1/2}, \rho^h_{n+1/2} \boldsymbol{\jmath})_\Omega - \frac{1}{\mathbb{F}{\rm r}^2}\left( \bu^h_{n+1/2}, \varrho^h_{m,n+1/2}  y \nabla \phi^h_{n+1/2} \right)_\Omega\nn\\
 &~- (\nabla \cdot \bu^h_{n+1/2}, p_{n+1}^h)_\Omega - (\nabla \bu^h_{n+1/2},  \boldsymbol{\tau}(\bu^h_{n+1/2},\phi_{n+1/2}^h))_\Omega \nn\\
  &~ +\left(\bu^h_{n+1/2}, v^h_{n+1} \nabla \phi_{n+1/2}^h\right)_\Omega \nn\\
  &~-  \displaystyle\sum_K \left(\nabla \bu_{n+1/2}^h, \theta_K \nabla \bu_{n+1/2}^h \right)_{\Omega_K} \nn\\
  &~ +  \displaystyle\sum_K \left(\tau_K \bu_{n+1/2}^h \cdot \nabla v_{n+1}^h, \mathscr{R}_I\phi^h_{n+1/2}\right)_{\Omega_K} .
\end{align}
By virtue of (\ref{eq: identify3}) and (\ref{prop 2}) we have the identities:
\begin{subequations}
\begin{align}
   (\nabla \bu^h_{n+1/2}, \rho^h_{n+1/2} \bu^h_{n+1/2} \otimes \bu^h_{n+1/2})_\Omega + (\bu^h_{n+1/2}, \tfrac{1}{2}\|\bu^h_{n+1/2}\|_2^2\varrho^h_{m,n+1/2}\nabla \phi_{n+1/2}^h)_\Omega &~=0,\\[6 pt]
   - (\nabla \cdot \bu^h_{n+1/2}, p_{n+1}^h)_\Omega &~=0,\\
   \frac{1}{\mathbb{F}{\rm r}^2}(\bu^h_{n+1/2}, \rho^h_{n+1/2} \boldsymbol{\jmath})_\Omega + \frac{1}{\mathbb{F}{\rm r}^2}\left( \bu^h_{n+1/2},  \varrho^h_{m,n+1/2}  y \nabla \phi^h_{n+1/2} \right)_\Omega  &~=0.
\end{align}
\end{subequations}
These reduce (\ref{eq: cont combining time 2 1}) to
\begin{align}\label{eq: reduced}
 (\bu^h_{n+1/2},  \dfrac{[\![\rho^h \bu^h]\!]_{n+1/2}}{\Delta t_n})_\Omega ~~=&~~- (\nabla \bu^h_{n+1/2},  \boldsymbol{\tau}(\bu^h_{n+1/2},\phi_{n+1/2}^h))_\Omega -  \displaystyle\sum_K \left(\nabla \bu_{n+1/2}^h, \theta_K \nabla \bu_{n+1/2}^h \right)_{\Omega_K}\nn\\
  &~ +\left(\bu^h_{n+1/2}, v^h_{n+1} \nabla \phi_{n+1/2}^h\right)_\Omega
   +  \displaystyle\sum_K \left(\tau_K \bu_{n+1/2}^h \cdot \nabla v_{n+1}^h, \mathscr{R}_I\phi^h_{n+1/2}\right)_{\Omega_K}.
\end{align}
Addition of (\ref{eq: cont combining time 1}) and (\ref{eq: reduced}) by using (\ref{eq: mid point for momentum}) gives:
\begin{align}\label{eq: cont combining time 3}
  \left(\bu^h_{n+1/2},  \rho^h_{n+1/2}  \dfrac{[\![\bu^h]\!]_{n+1/2}}{\Delta t_n}  \right)_\Omega &\nn\\
 +\left(\dfrac{[\![\rho^h]\!]_{n+1/2}}{\Delta t_n} ,\bu^h_{n+1/2}\cdot   \bu^h_{n+1/2}-\tfrac{1}{2}\bu^h_{n+1}\cdot \bu^h_n  \right)_\Omega \nn &\\
 +\frac{1}{\mathbb{F}{\rm r}^2}\left(\dfrac{[\![\rho^h]\!]_{n+1/2}}{\Delta t_n} ,  y \right)_\Omega\nn &\\
 + \frac{1}{\mathbb{W}{\rm e}}\left( 1, \frac{[\![\delta(\phi) \|\nabla \phi^h\|_{\epsilon,2}]\!]_{n+1/2}}{\Delta t_n} \right)_\Omega &=~- ~(\nabla \bu^h_{n+1/2},  \boldsymbol{\tau}(\bu^h_{n+1/2},\phi_{n+1/2}^h))_\Omega \nn\\
 &~~~~~-  \displaystyle\sum_K \left(\nabla \bu_{n+1/2}^h, \theta_K \nabla \bu_{n+1/2}^h \right)_{\Omega_K}.
\end{align}
Using the identity
\begin{align}
   \|\bu^h_{n+1/2}\|^2 -\tfrac{1}{2}\bu^h_{n+1}\cdot \bu^h_n= \tfrac{1}{2}(\|\bu^h\|^2)_{n+1/2}\equiv\tfrac{1}{2}\|\bu^h_{n+1}\|^2+\tfrac{1}{2}\|\bu^h_n\|^2,
\end{align}
we identify the sum of the first two terms on the left-hand side of (\ref{eq: cont combining time 3}) as the change of kinetic energy.
Next, the third term on the left-hand side of (\ref{eq: cont combining time 3}) represents change in gravitational energy.
The latter term on the left-hand side of (\ref{eq: cont combining time 3}) resembles the surface energy evolution.
We are left with:
\begin{align}\label{eq: cont combining time 4}
 \dfrac{[\![\mathscr{E}^h]\!]_{n+1/2}}{\Delta t_n}=&~- \left(\nabla \bu^h_{n+1/2}, \boldsymbol{\tau}(\bu^h_{n+1/2},\phi_{n+1/2}^h)\right)_\Omega \nn\\
 &~-  \displaystyle\sum_K \left(\nabla \bu_{n+1/2}^h, \theta_K \nabla \bu_{n+1/2}^h \right)_{\Omega_K}.
\end{align}
\end{proof}
\begin{rmk}
 Following Brackbill \cite{brackbill1992continuum} we employ the time-step restriction $\Delta t_n \leq \Delta t_{{\rm max}}$ with
 \begin{align}\label{eq: time-step restriction}
  \Delta t_{{\rm max}} =  \left( \dfrac{\bar{\rho} \left(\min_Q h_Q\right)^3\mathbb{W}{\rm e}}{2 \pi}  \right)^{1/2},
 \end{align}
where $\bar{\rho} = (\rho_1+\rho_2)/2$.
\end{rmk}

\section{Numerical experiments}\label{sec: Numerical experiments}

In this Section we evaluate the proposed numerical methodology on several numerical examples in two and
three dimensions. To test the formulation we use both a static and dynamic equilibrium problem and check the energy dissipative property of the method. We do not test the method on a `violent' problem in order to avoid the usage of redistancing procedures. All problems are evaluated with NURBS basis functions that are mostly $C^1$-quadratic but every velocity space is enriched to cubic $C^2$ in the associated direction \cite{Evans13steadyNS, Evans13unsteadyNS}.
\subsection{Static spherical droplet}
Here we test the surface tension component of the formulation by considering a spherical droplet in equilibrium \cite{francois2006balanced}.
Viscous and gravitational forces are absent and hence the surface tension forces are in balance with the pressure difference between the two fluids. The interface balance (\ref{eq: strong stress surface tension}) thus reduces to:
\begin{align}\label{eq: Young-LaPlace}
 [\![\![p]\!]\!] = -\sigma \kappa,
\end{align}
which is also referred to as the Young-Laplace equation. The exact curvature is given by:
\begin{align}
 \kappa = -\dfrac{d-1}{r},
\end{align}
where we recall $d =2,3$ as the number of spatial dimensions.
The spherical droplet of radius $r =2$ of fluid 1 with density $\rho_1 = 1.0$ is immersed in fluid 2 with density $\rho_2 = 0.1$ .  The surface tension coefficient is $\sigma = 73$. The computational domain is a cubic with a side length of 8 units and the spherical droplet is positioned in the center of it. On all surfaces a non-penetration boundary condition ($u_n =0$) is imposed. 

We employ three meshes with uniform elements: $20\times20$, $40\times40$ and $80\times80$. We take $\eps=2h_K$ for all simulations in this Section. The time-step is taken as $\Delta t_n = 10^{-3}$ which satisfies (\ref{eq: time-step restriction}) for each of the meshes. We exclude the discontinuity capturing mechanisms for this problem, i.e. we set $\mathcal{C} =0$. In \cref{fig:pressure} we display the pressure for the finest mesh.
\begin{figure}[h!]
\begin{center}
 \includegraphics[width=0.5\textwidth]{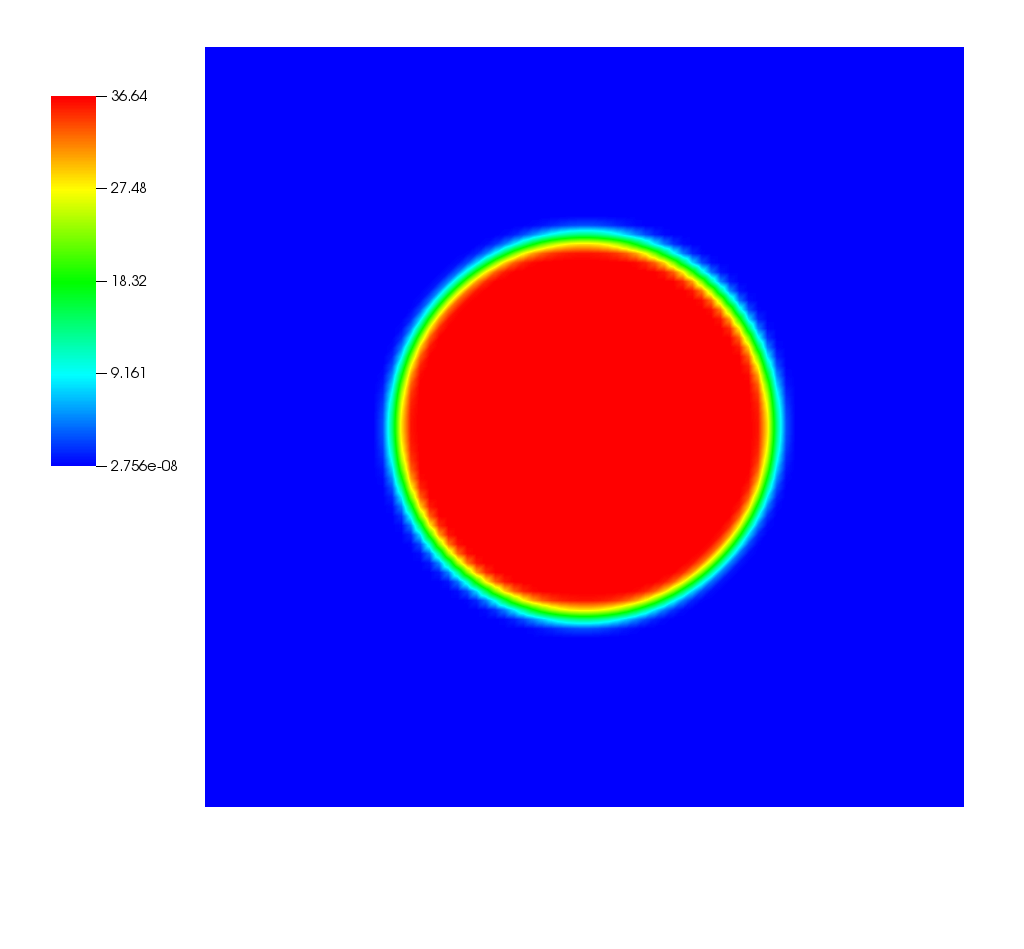}
\end{center}
\caption{Pressure}
\label{fig:pressure}
\end{figure}\\
In \cref{fig:pressure slice} we display the pressure contours for each of the meshes.  The corresponding pressure jump is $37.97, 36.80$ and $36.56$ for the meshes $20\times20$, $40\times40$ and $80\times80$ respectively. This implies second-order convergence.
\begin{figure}[h!]
\begin{center}
 \includegraphics[width=0.5\textwidth]{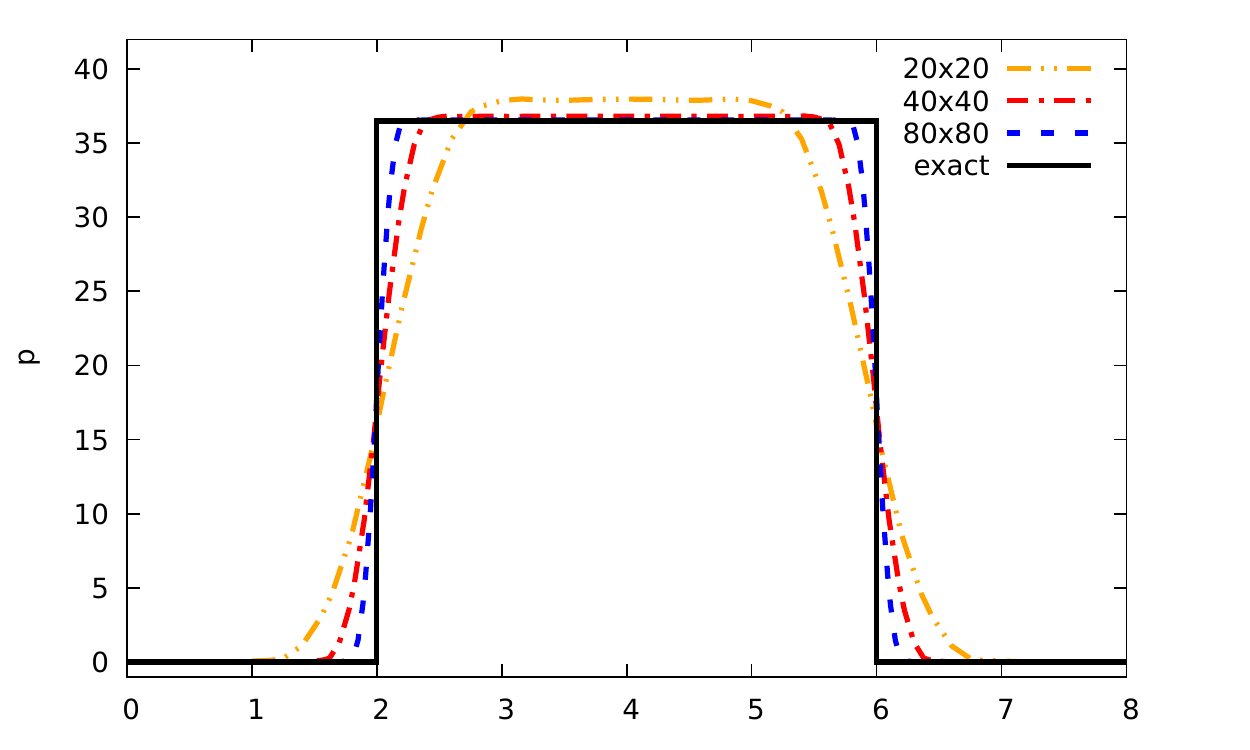}
\end{center}
\caption{Pressure slice at $y = 4.0$.}
\label{fig:pressure slice}
\end{figure}\\

\newpage
In the \cref{fig:Static droplet Energy evo,fig:Static droplet Energy diss} we depict the energy evolution and dissipation for each of the meshes. The theoretical value of the surface energy is $2 \pi r \sigma \approx 917.34$ which is well represented on the finest mesh. We see that the total and surface energies are (virtually) constant and the kinetic energy grows but has an insignificant contribution to the total energy.

\begin{figure}[h!]
\begin{subfigure}{0.49\textwidth}
\centering
\includegraphics[width=0.95\textwidth]{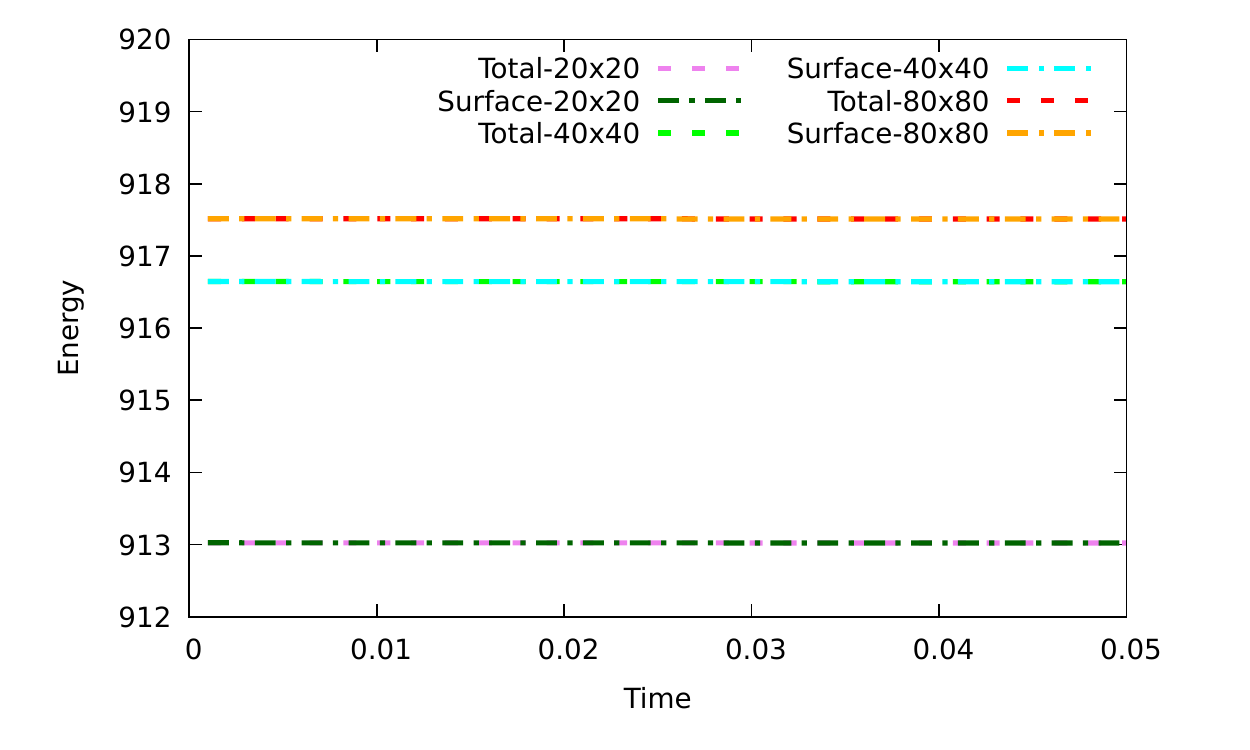}
\caption{Total and surface energy evolution.}
\end{subfigure}
\begin{subfigure}{0.49\textwidth}
\centering
\includegraphics[width=0.95\textwidth]{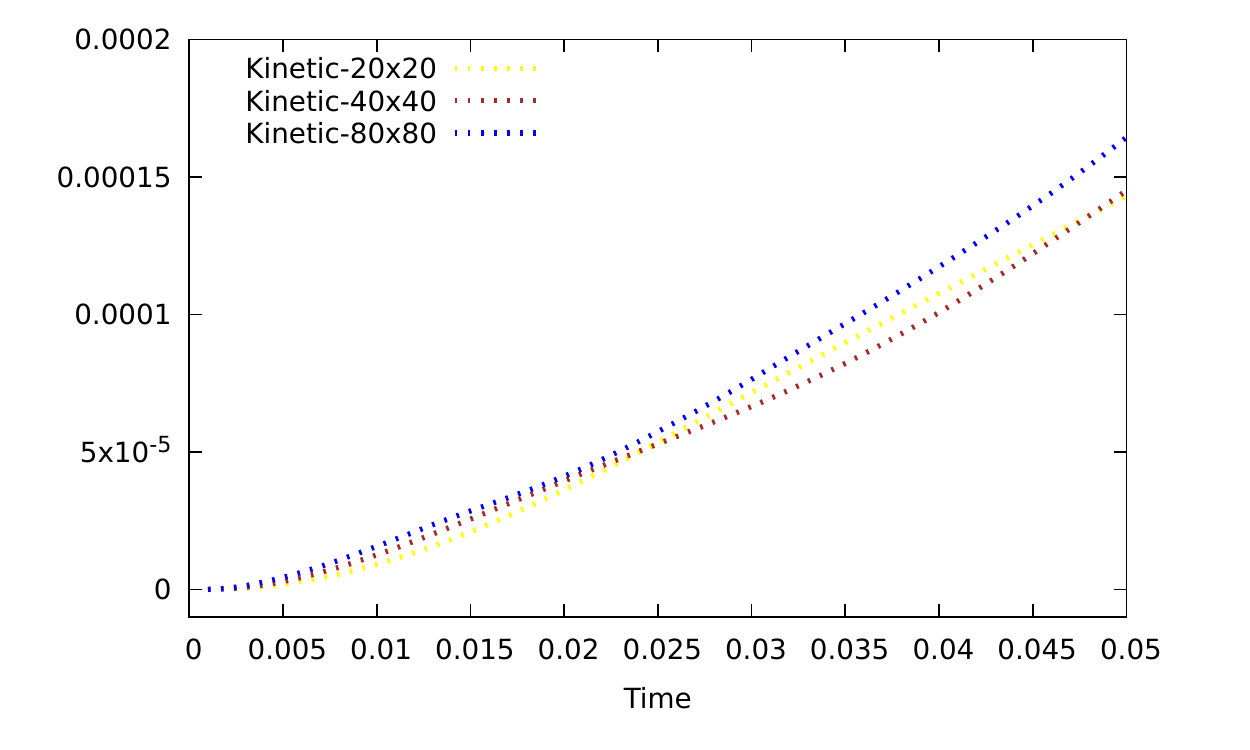}
\caption{Kinetic energy evolution.}
\end{subfigure}
\caption{Static droplet. Energy evolution for the various meshes.}
\label{fig:Static droplet Energy evo}
\end{figure}

\begin{figure}[h!]
\begin{subfigure}{0.49\textwidth}
\centering
\includegraphics[width=0.95\textwidth]{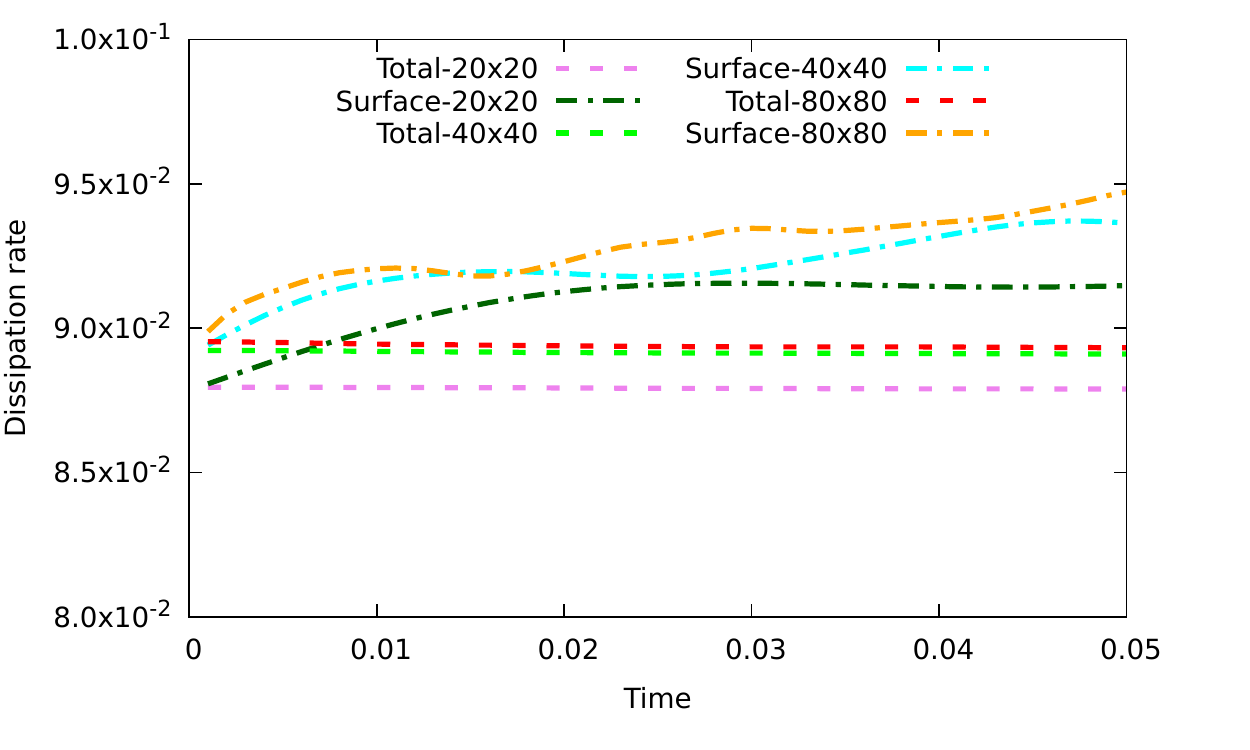}
\caption{Total and surface energy dissipation.}
\end{subfigure}
\begin{subfigure}{0.49\textwidth}
\centering
\includegraphics[width=0.95\textwidth]{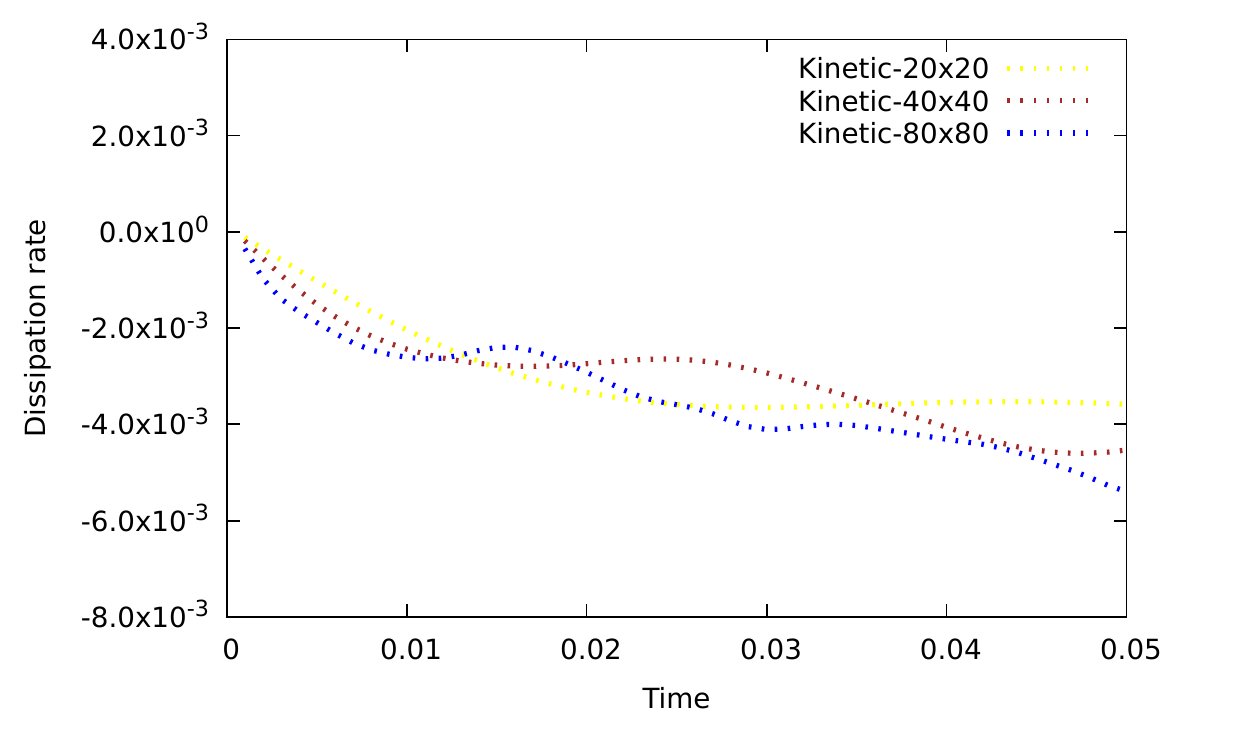}
\caption{Kinetic energy dissipation.}
\end{subfigure}
\caption{Static droplet. Energy dissipation for the various meshes.}
\label{fig:Static droplet Energy diss}
\end{figure}

Note that this test-case represents a stable situation and as such velocities and thus the kinetic energy should vanish.
Since the system is not in a total energy-stable state we note the occurrence of parasitic currents.
We report the magnitude of these currents in \cref{fig:parasitic}.
Even though the parasitic currents are very small, they are unfortunately present. This is a well-known problem. One can use several `tricks' to reduce parasitic currents. A possibility is to use a so-called balanced-force algorithm \cite{abadie2015combined} which assumes that the curvature is determined analytically.
\begin{figure}[h!]
\begin{subfigure}{0.47\textwidth}
\centering
\includegraphics[width=0.95\textwidth]{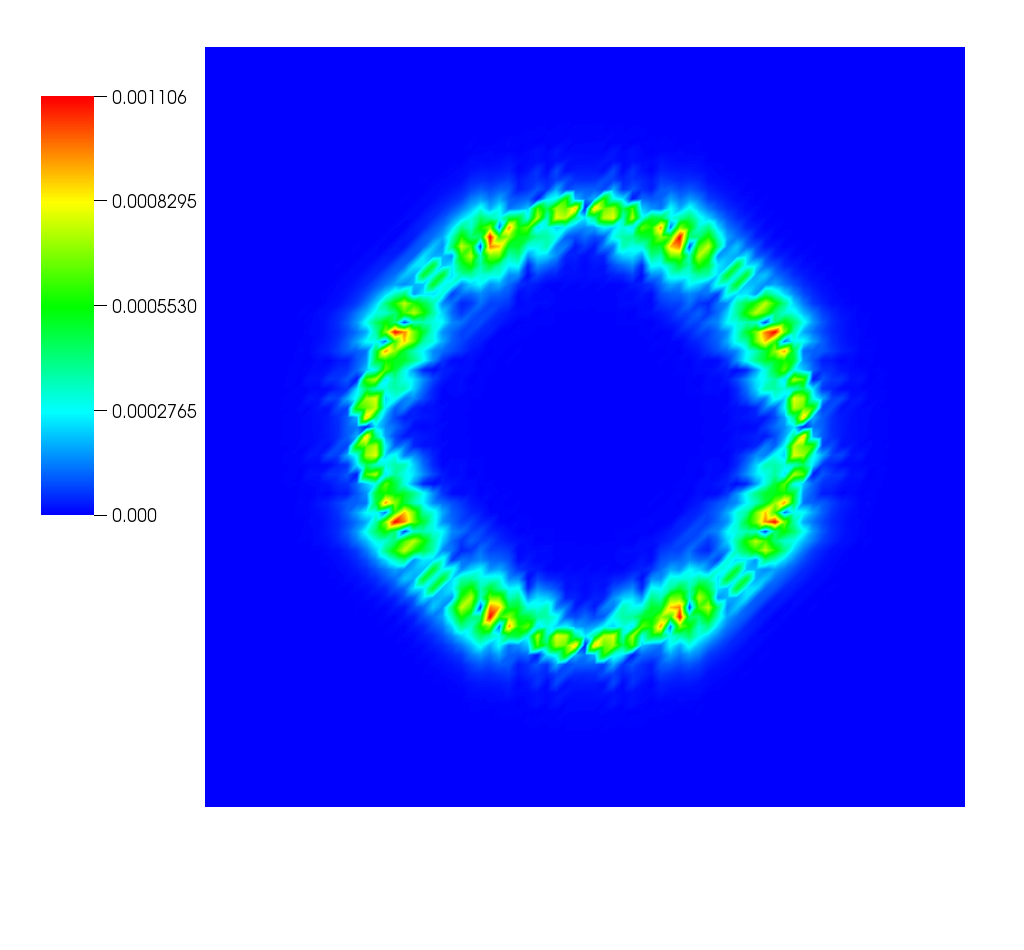}
\caption{At time-step $1$.}
\end{subfigure}
\begin{subfigure}{0.47\textwidth}
\centering
\includegraphics[width=0.95\textwidth]{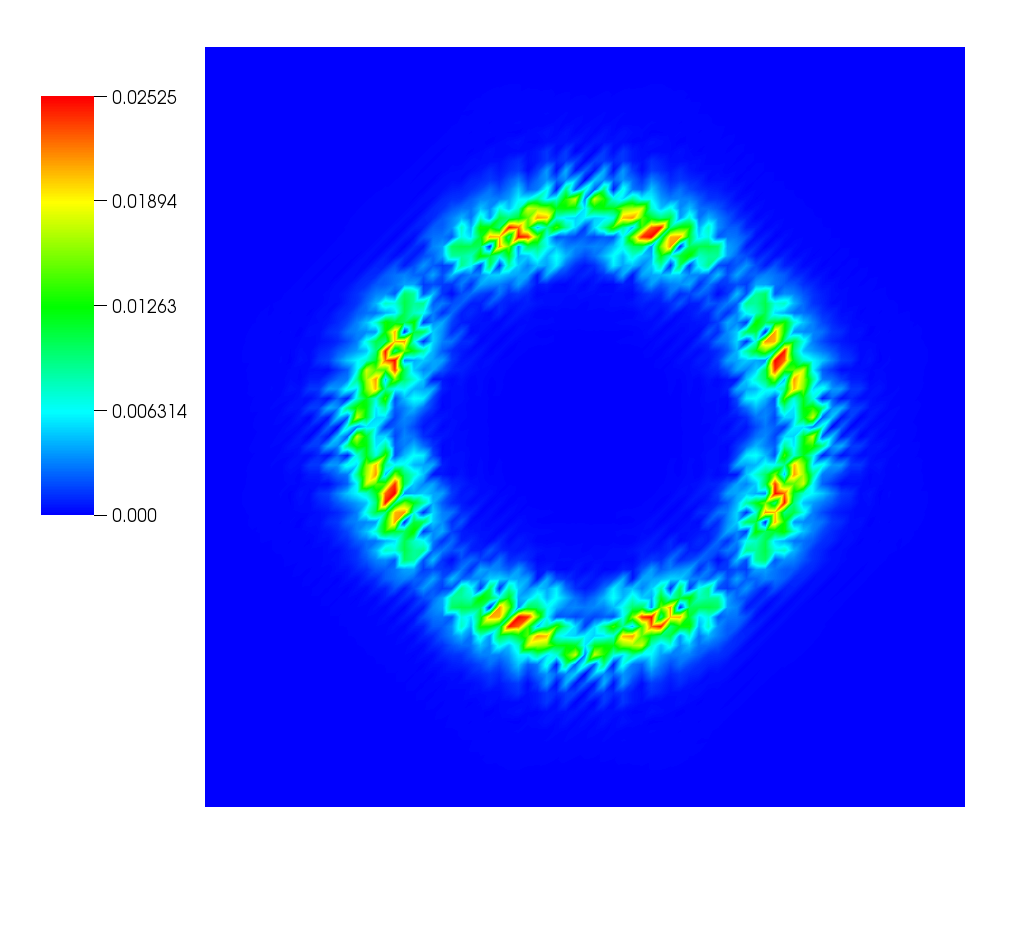}
\caption{At time-step $50$.}
\end{subfigure}
\caption{Norm of the velocity.}
\label{fig:parasitic}
\end{figure}

\begin{rmk}
We note that additional dissipation mechanisms for the surface evolution can upset energy-stability of the system. Well-balanced dissipation, introduced for the Navier-Stokes-Korteweg equations \cite{giesselmann2014energy}, is a possible strategy to resolve this.
\end{rmk}
\newpage
In \cref{fig:auxiliary} we plot the variable $v_{n+1}^h$. Note that the maximum theoretical value is
\begin{align}
 \max_{\bx \in \Omega} v =&~ -\sigma \min_{\bx \in \Omega} \left(\delta_\eps(\phi) \nabla \cdot \left(\dfrac{\nabla \phi}{\|\nabla \phi\|_{\epsilon,2}}\right)\right) \nn\\
 \approx&~ \frac{\sigma}{2} \max_{\bx \in \Omega}\delta_\eps(\phi)  \nn\\
 \approx&~ 161.3,
\end{align}
where the $\max_{\bx \in \Omega} \delta_\eps(\phi) = \max_{\bx \in \Omega} \frac{1}{\eps} (H^p)^{(1)}(\frac{\phi}{\eps})= \frac{1}{2 h_K}\max_{\bx \in \Omega} (H^p)^{(1)}(\frac{\phi}{\eps}) = \frac{1}{2*\frac{8}{80}*\sqrt{2}} \frac{5}{4}$. We see that the finest mesh is able to accurately represent $v_{n+1}^h$ whereas on the coarser meshes $v_{n+1}^h$ is smeared out significantly.

\begin{figure}[h!]
\begin{subfigure}{0.33\textwidth}
\centering
\includegraphics[width=0.95\textwidth]{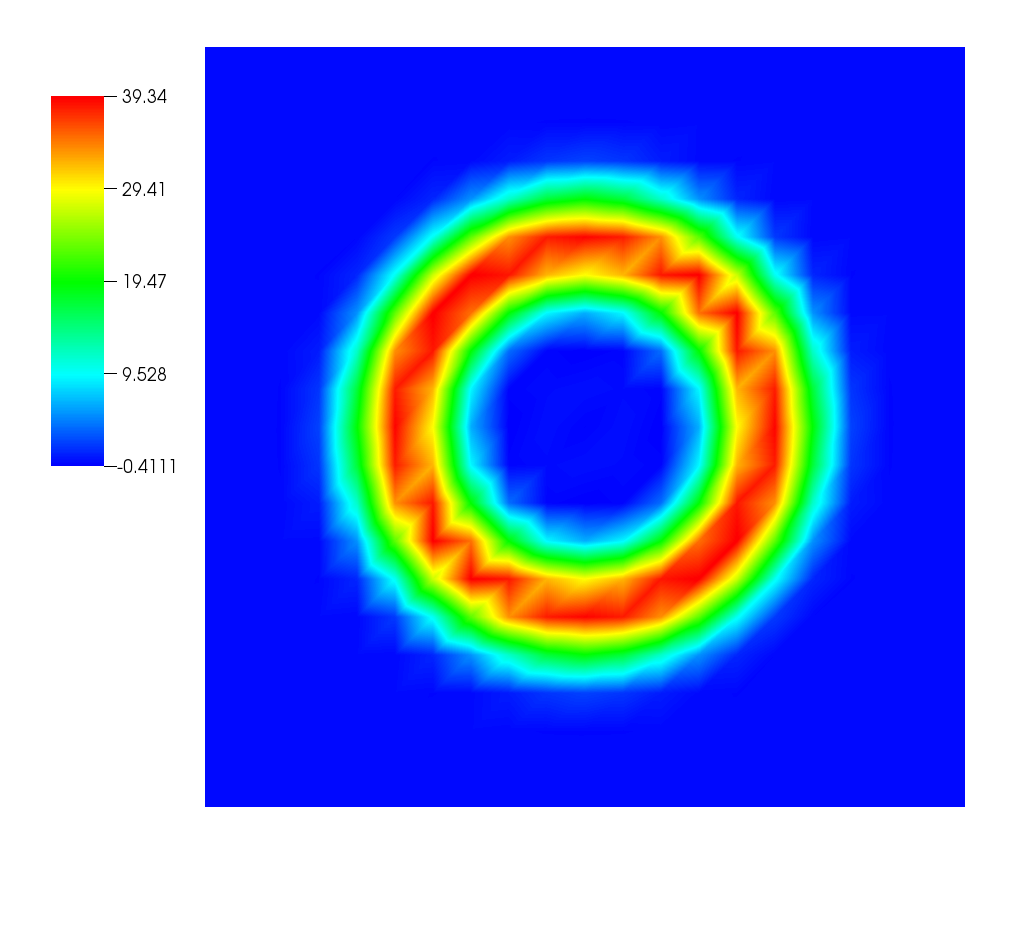}
\caption{20$\times$20 mesh.}
\end{subfigure}
\begin{subfigure}{0.33\textwidth}
\centering
\includegraphics[width=0.95\textwidth]{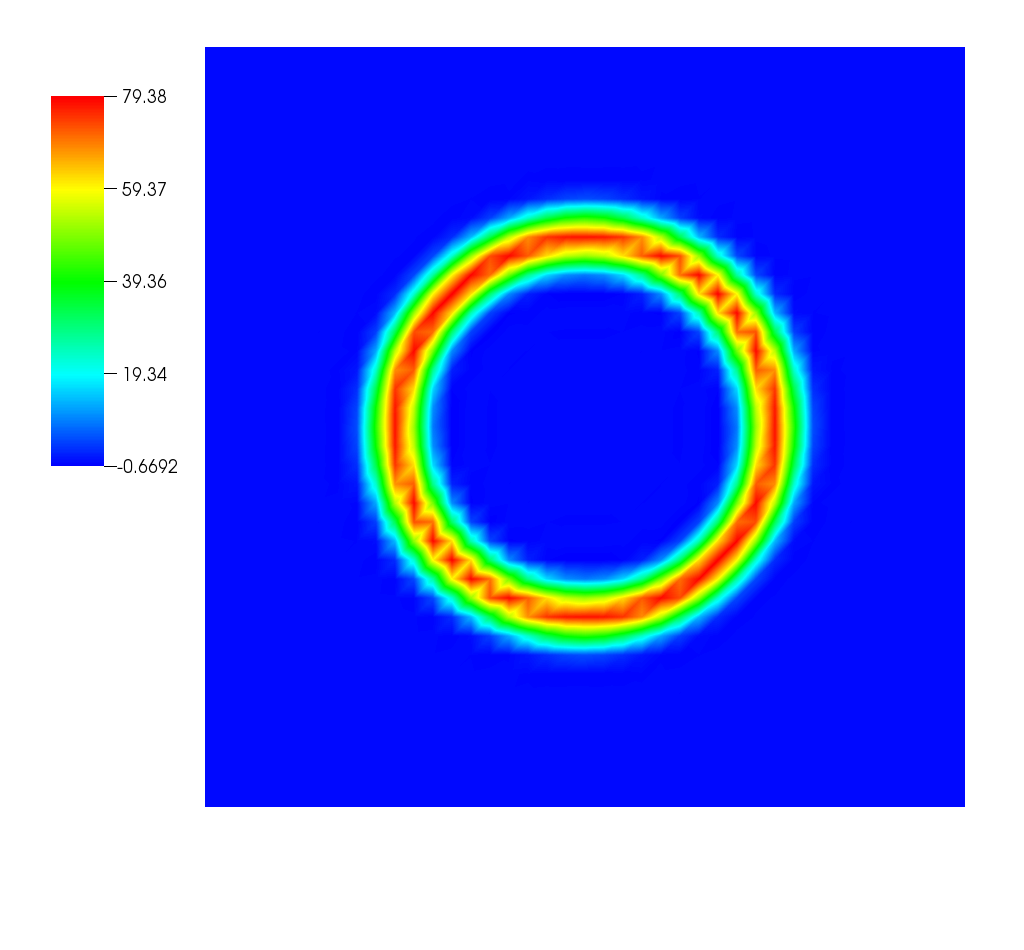}
\caption{40$\times$40 mesh.}
\end{subfigure}
\begin{subfigure}{0.33\textwidth}
\centering
\includegraphics[width=0.95\textwidth]{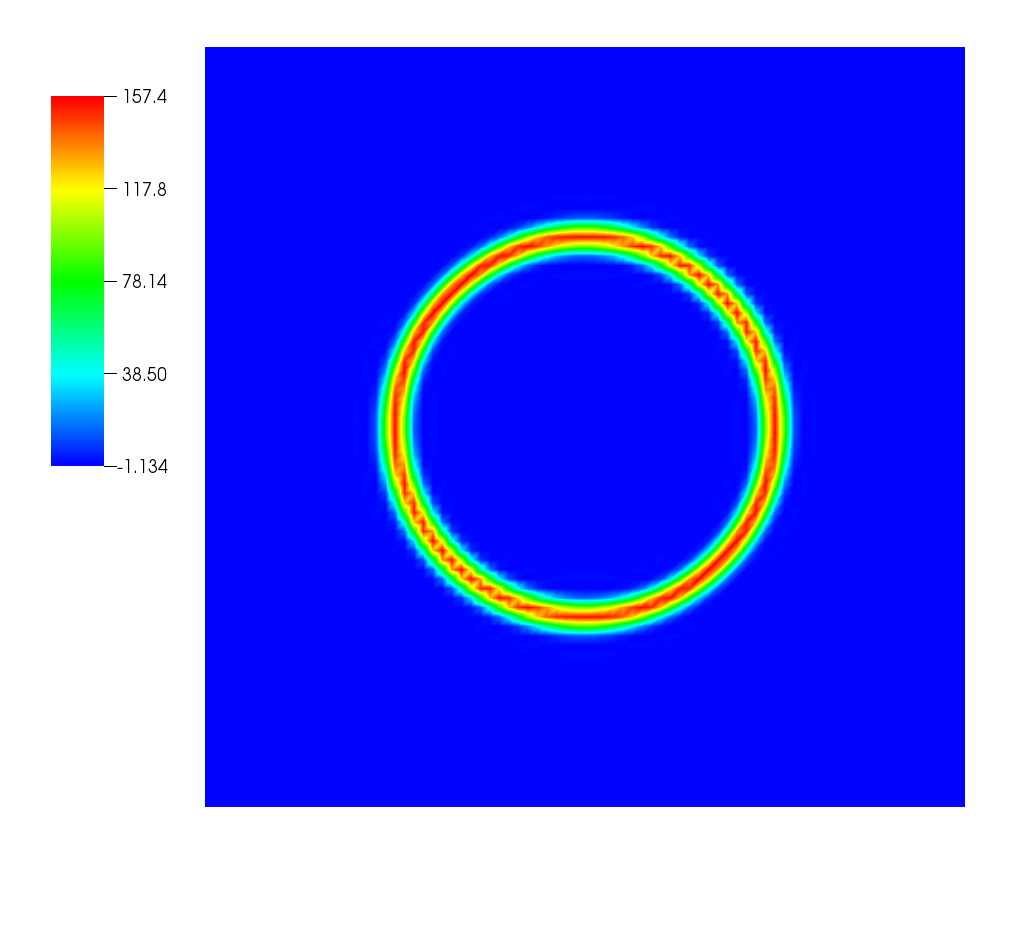}
\caption{80$\times$80 mesh.}
\end{subfigure}
\caption{The auxiliary variable $v$ for the various meshes.}
\label{fig:auxiliary}
\end{figure}

\subsection{Droplet coalescence 2D}
In this example, inspired by Gomez et al. \cite{gomez2010isogeometric}, we simulate the merging of two droplets into a
single one. Gravitational forces are absent Due to pressure and capillarity forces the single droplet then develops to a circular shape.
We take as computational domain the unit box $\Omega = [0,1]^d$ and apply no-penetration boundary conditions. The initial configuration consists of two droplet at rest ($\bu_0 = \mathbf{0}$) with centers at $\mathbf{c}_1=(0.4, 0.5)$ and $\mathbf{c}_2=(0.78, 0.5)$ and radii $r_1 = 0.25$ and $r_2 = 0.1$ respectively.
The diffuse interfaces of the droplets initially overlap on a small part of the domain.
If this not the case the droplets remain at their position and thus no merging would occur.
In contrast with the Navier-Stokes Korteweg equations, in this situation the interface has a finite width, due to the definition of $H_\eps(\phi)$. The Navier-Stokes Korteweg equations have no absolute notion of interface width; its effect is decaying exponentially. 
The droplets have a larger density ($\rho_1=100$) than the surrounding fluid ($\rho_2=1$) while the viscosities are equal: $\mu_1 = \mu_2 = 1$. We take as surface tension the low value of $\sigma = 0.1$ which causes a slowly merging process. To initialize the level-set we split the domain into two parts ($x \leq 0.665$ and $x>0.665$), such that each contains one droplet, and apply the standard distance initialization to each subdomain. We use $50\times50$ elements, set the time-step as $\Delta t = 0.1$ and take $\mathcal{C}=0.4$.

We show in the \cref{fig:sol0p2,fig:sol0p6,fig:sol1p0,fig:sol1p8,fig:sol3p0,fig:sol8.0} a detailed view of the merging process. The colors patterns are set per snapshot such that difference are most apparent. 

\begin{figure}[h!]
\begin{center}
 \includegraphics[width=0.70\textwidth]{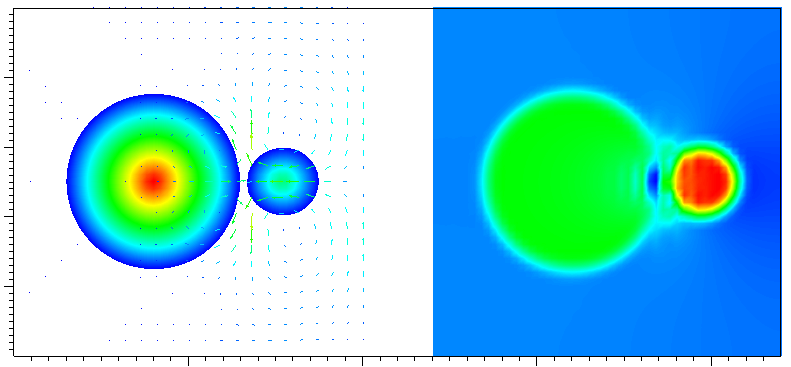}
\end{center}
\caption{Coalescence 2D. Solutions at $t=2$: level-set field and velocity arrows (left) and pressure field (right).}
\label{fig:sol0p2}
\end{figure}
\begin{figure}[h!]
\begin{center}
 \includegraphics[width=0.70\textwidth]{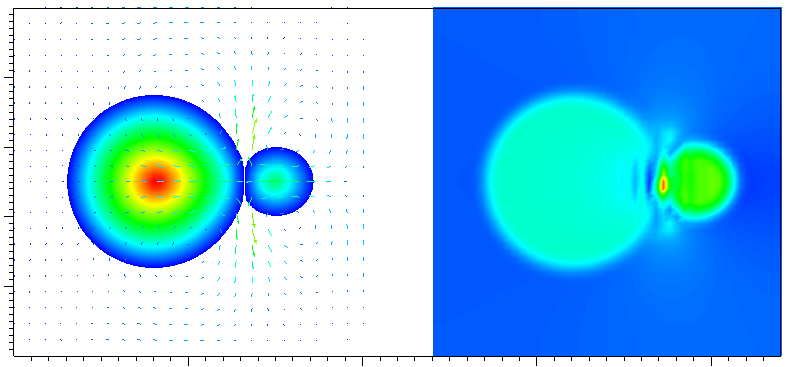}
\end{center}
\caption{Coalescence 2D. Solutions at $t=6$: level-set field and velocity arrows (left) and pressure field (right).}
\label{fig:sol0p6}
\end{figure}
\begin{figure}[h!]
\begin{center}
 \includegraphics[width=0.70\textwidth]{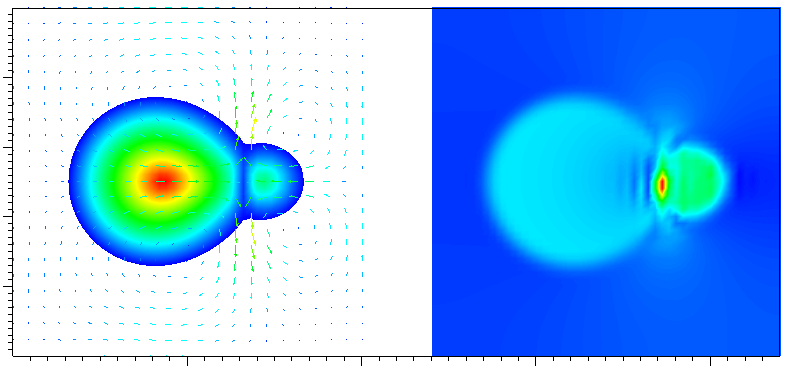}
\end{center}
\caption{Coalescence 2D. Solutions at $t=10$: level-set field and velocity arrows (left) and pressure field (right).}
\label{fig:sol1p0}
\end{figure}
\begin{figure}[h!]
\begin{center}
 \includegraphics[width=0.70\textwidth]{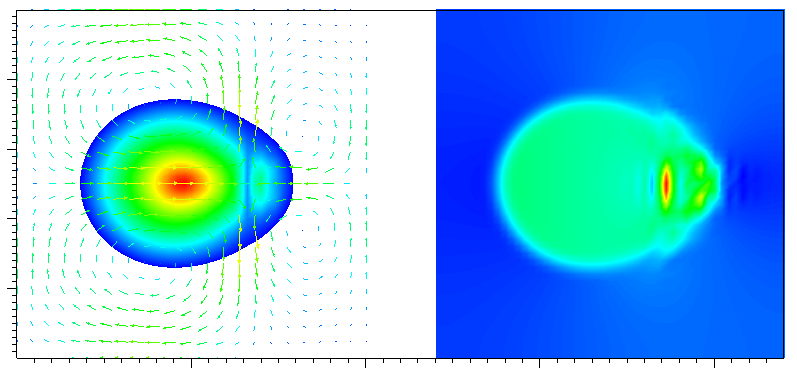}
\end{center}
\caption{Coalescence 2D. Solutions at $t=18$: level-set field and velocity arrows (left) and pressure field (right).}
\label{fig:sol1p8}
\end{figure}
\begin{figure}[h!]
\begin{center}
 \includegraphics[width=0.70\textwidth]{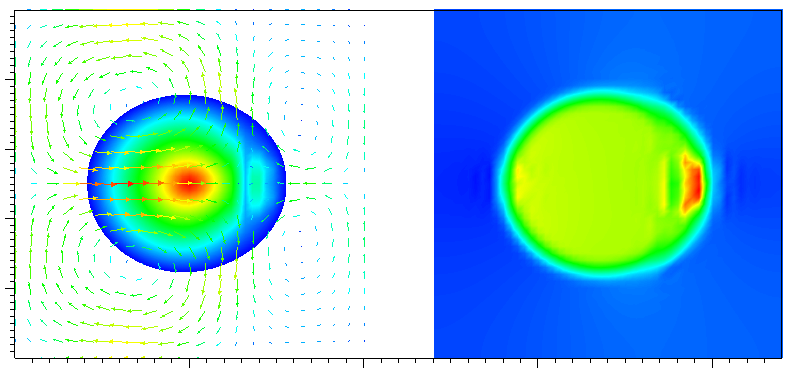}
\end{center}
\caption{Coalescence 2D. Solutions at $t=30$: level-set field and velocity arrows (left) and pressure field (right).}
\label{fig:sol3p0}
\end{figure}
\begin{figure}[h!]
\begin{center}
 \includegraphics[width=0.70\textwidth]{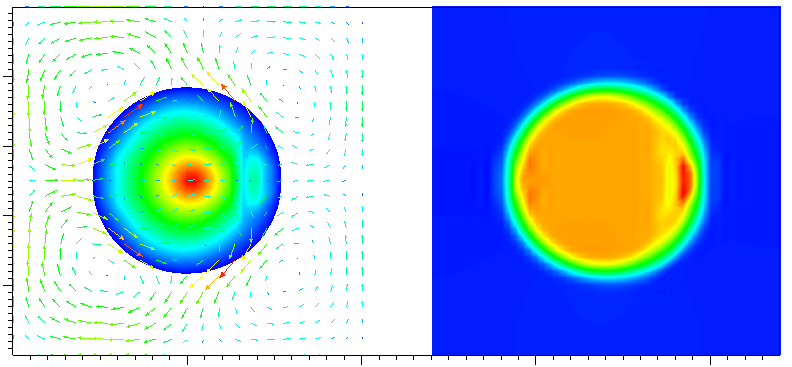}
\end{center}
\caption{Coalescence 2D. Solutions at $t=80$: level-set field and velocity arrows (left) and pressure field (right).}
\label{fig:sol8.0}
\end{figure}

In the \cref{fig:2D Energy evo,fig:2D Energy diss} we show the energy evolution and dissipation. In this case the theoretical value of the initial surface energy is $2 \pi (r_1+r_2) \sigma \approx 0.2199$. We observe that the total and surface energies monotonically decrease in time. The kinetic energy increases when the droplet move towards  each other ($t<10$) and decreases during the merging process and subsequently flattens out.

\clearpage

In order to test whether the equilibrium state has been reached we evaluate the circularity of the droplet. The circularity is defined as the fraction of the perimeter evaluated from the droplet volume and the perimeter itself:
\begin{align}
  \upgamma = \dfrac{2 \left(\pi \displaystyle\int_{\left\{\Omega: \phi >0\right\}} ~{\rm d}\Omega\right)^{1/2}}{\displaystyle\int_\Omega  \delta_\eps(\phi)\|\nabla \phi\|_{\epsilon,2}~{\rm d}\Omega}.
\end{align}
The circularity depicted in \cref{fig: Circularity} confirms the equilibrium state as $\upgamma$ tends to $1$.
\begin{figure}[h!]
\begin{subfigure}{0.49\textwidth}
\centering
\includegraphics[width=0.95\textwidth]{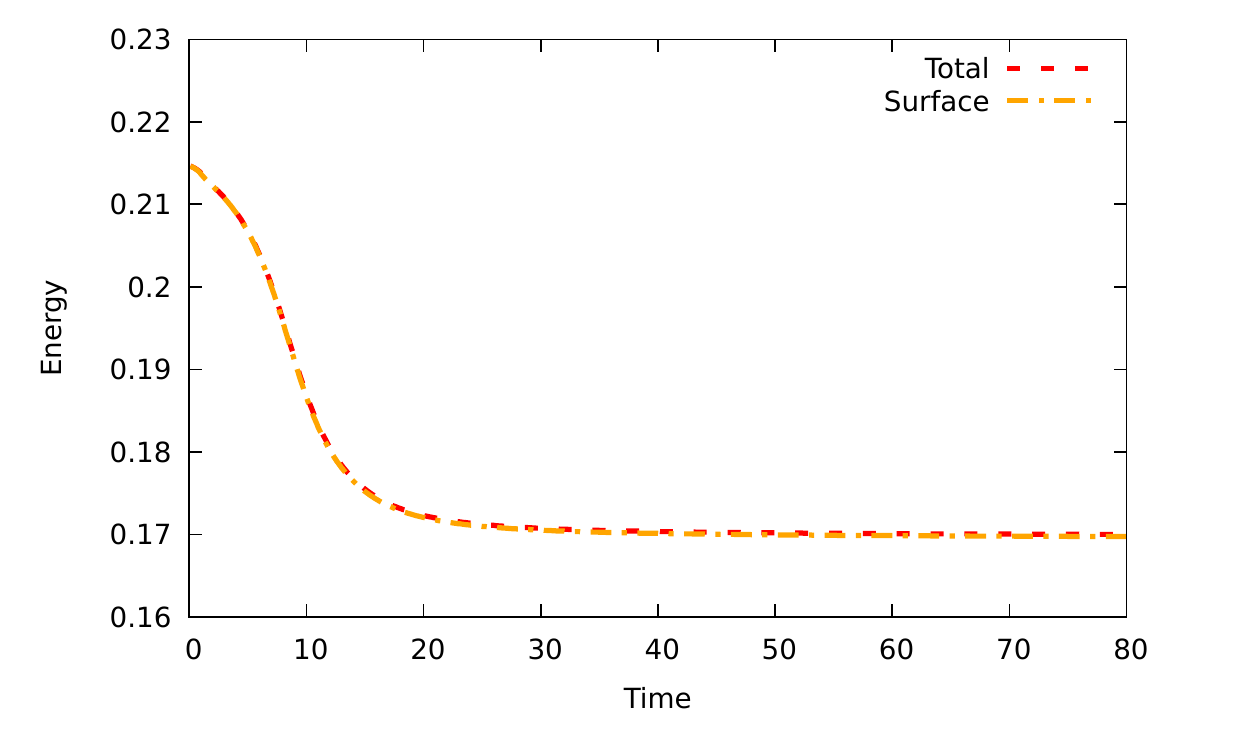}
\caption{Total and surface energy evolution.}
\end{subfigure}
\begin{subfigure}{0.49\textwidth}
\centering
\includegraphics[width=0.95\textwidth]{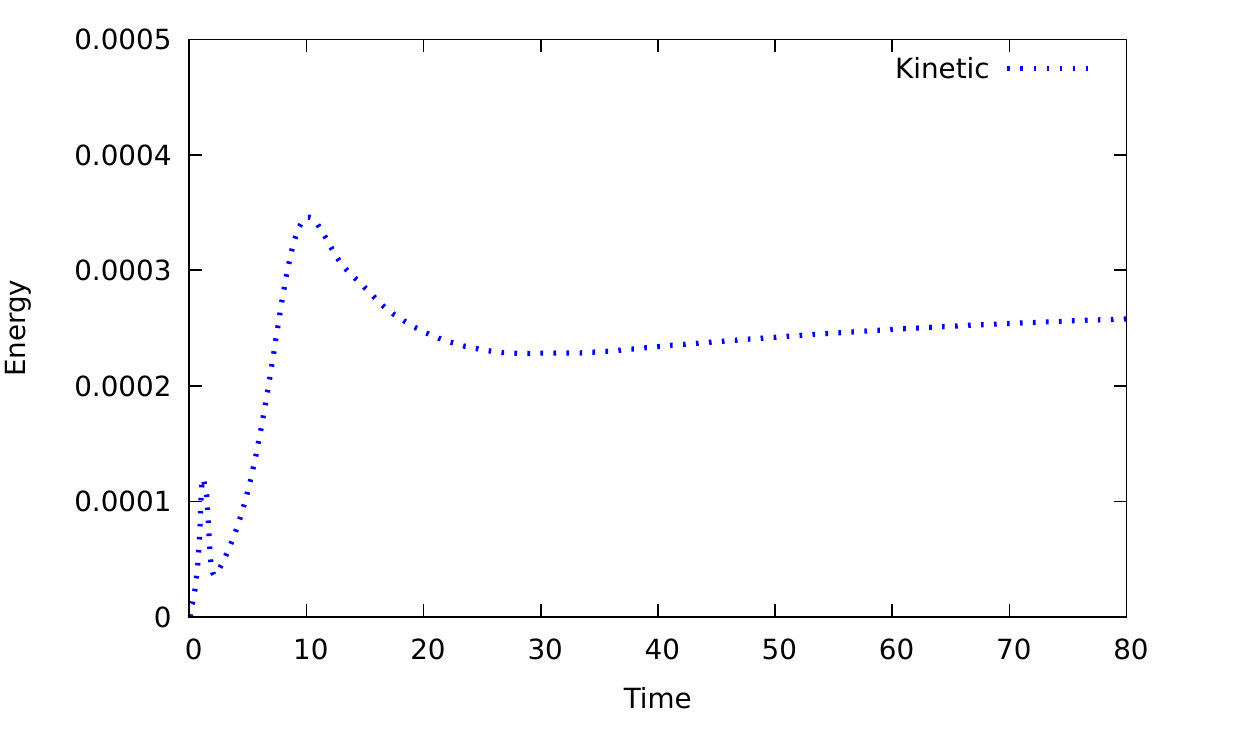}
\caption{Kinetic energy evolution.}
\end{subfigure}
\caption{Coalescence 2D. Energy evolution.}
\label{fig:2D Energy evo}
\end{figure}

\begin{figure}[h!]
\begin{subfigure}{0.49\textwidth}
\centering
\includegraphics[width=0.95\textwidth]{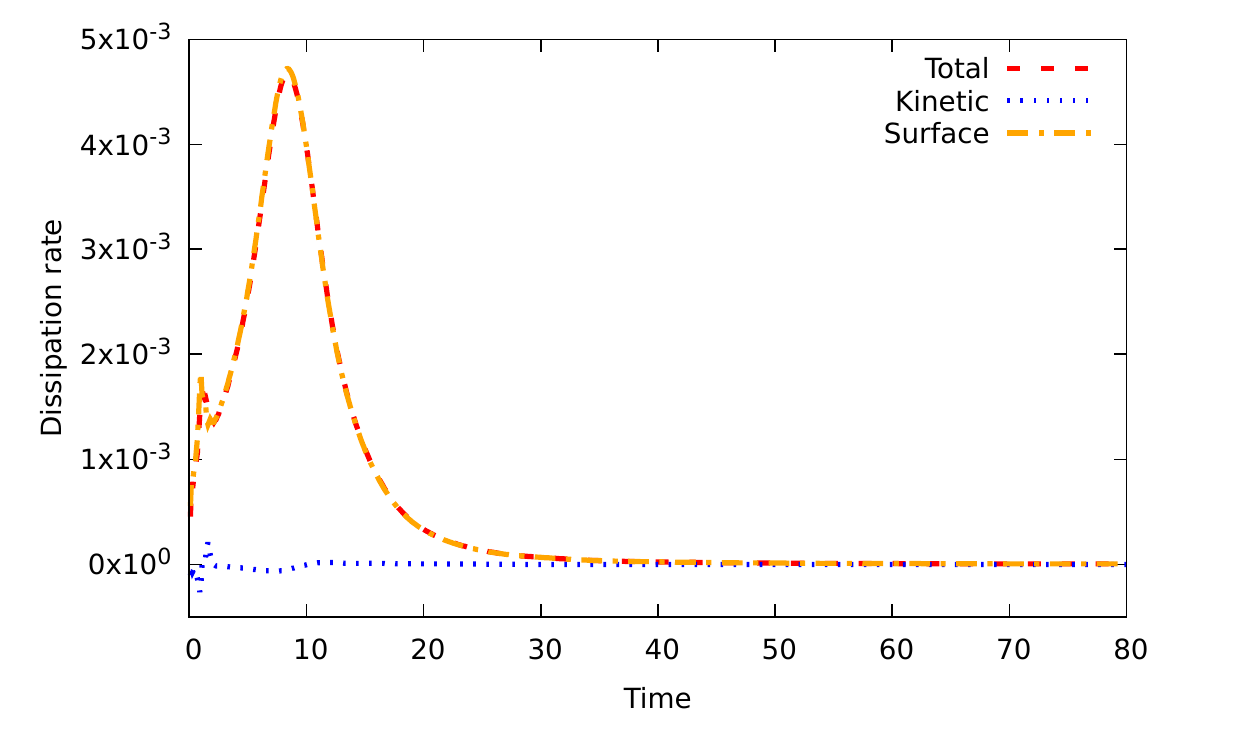}
\caption{Energy dissipation rate.}
\label{fig:2D Energy diss}
\end{subfigure}
\begin{subfigure}{0.49\textwidth}
\centering
\includegraphics[width=0.95\textwidth]{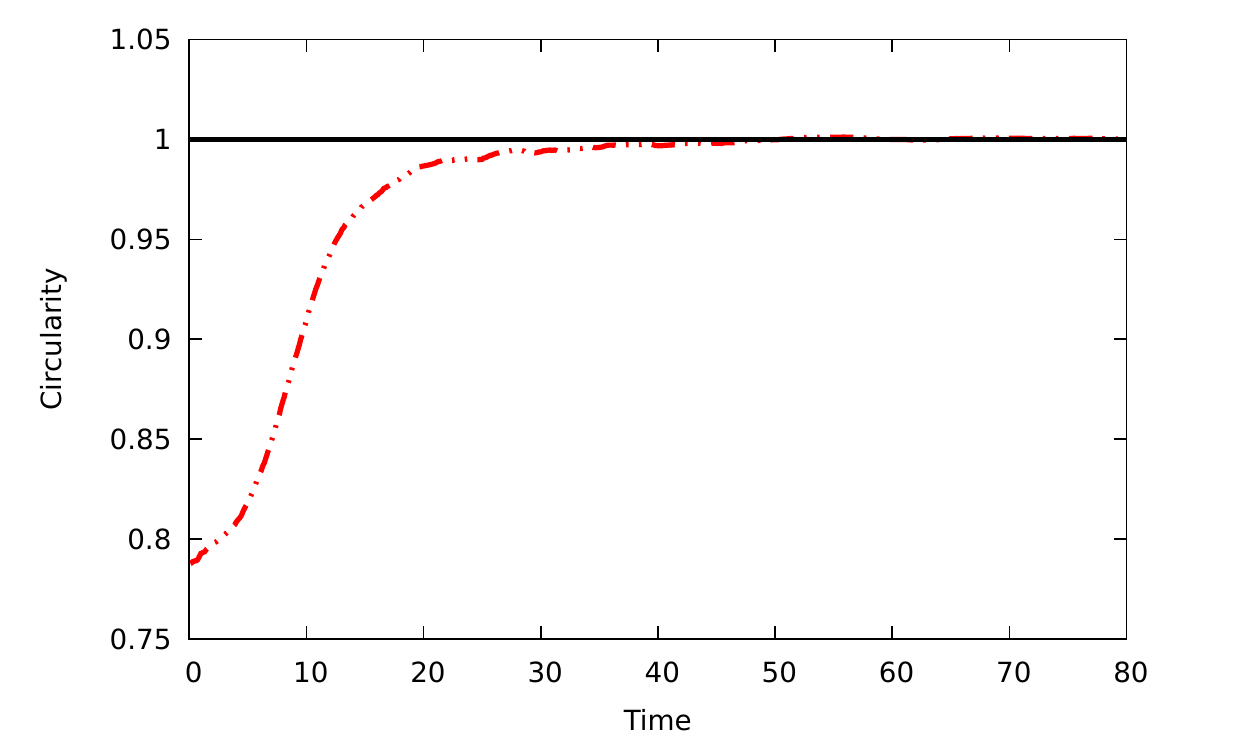}
\caption{Circularity.}
\label{fig: Circularity}
\end{subfigure}
\caption{Coalescence 2D. Energy dissipation rate and circularity.}
\label{fig:2D Energy diss-circ}
\end{figure}

\newpage
\newpage
\subsection{Droplet coalescence 3D}
Here we simulate the merging of two droplets in three dimensions. We use the same physical parameters as in the two-dimensional case. 
The centers of the droplets are at $\mathbf{c}_1= (0.4, 0.5, 0.6)$ and $\mathbf{c}_2=(0.75, 0.5, 0.5)$ and the radii remain the same: $r_1 = 0.25$ and $r_2 = 0.1$.
Also here the diffuse interfaces of the droplets initially overlap.
Again, to initialize the level-set we partition the domain, see \cref{fig:3D initial slice} and apply the standard distance initialization to each subdomain. The initial configuration is depicted in \cref{fig:3D initial condition}. We use $50\times50\times50$ elements, set the time-step as $\Delta t = 0.1$ and take $\mathcal{C}=0.1$.

\begin{figure}[h!]
\begin{subfigure}{0.49\textwidth}
\centering
\includegraphics[width=0.88\textwidth]{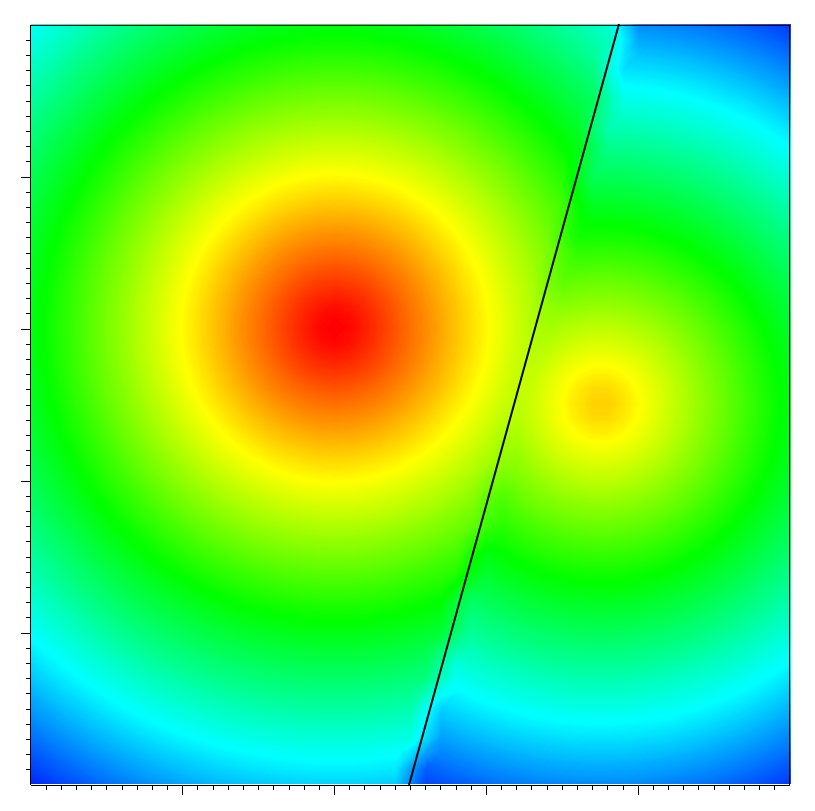}
\caption{Slice of initial condition at $y=0.5$}
\label{fig:3D initial slice}
\end{subfigure}
\begin{subfigure}{0.49\textwidth}
\centering
\includegraphics[width=0.95\textwidth]{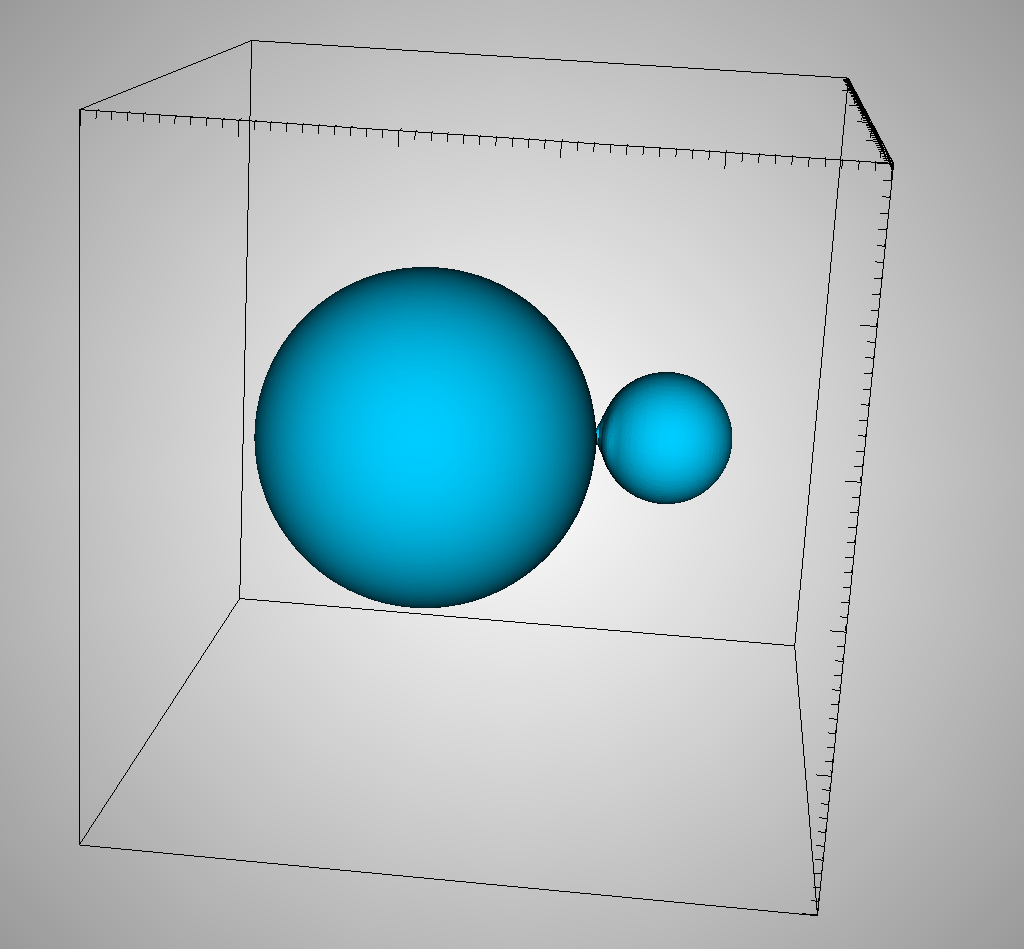}
\caption{Zero level-set contours of initial condition}
\label{fig:3D initial condition}
\end{subfigure}
\caption{Coalescence 3D. Initial condition.}
\label{fig:3D initial condition tot}
\end{figure}

We show in \cref{fig:Coalescence 3D. Solutions} snapshots of the merging process.
In \cref{fig:3D Energy evo diss} we visualize the energy evolution and dissipation. The theoretical value of the initial surface energy is $4 \pi (r_1^2+r_2^2) \sigma \approx 0.0911$. The behavior of the various energies is similar as in the two-dimensional case. Also in this case the energy-dissipative property of the numerical method is confirmed.

\begin{figure}[h!]
\begin{subfigure}{0.49\textwidth}
\centering
\includegraphics[width=0.95\textwidth]{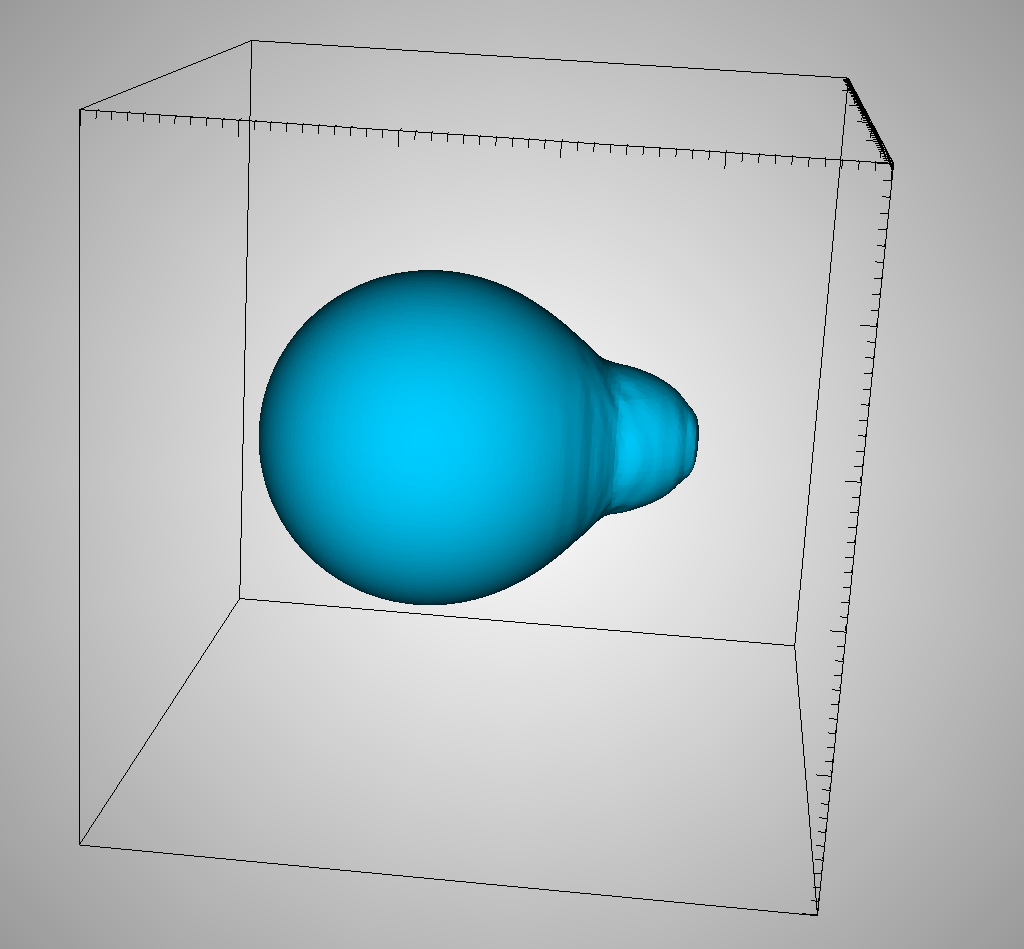}
\caption{Zero level-set contours at $t=10$}
\end{subfigure}
\begin{subfigure}{0.49\textwidth}
\centering
\includegraphics[width=0.95\textwidth]{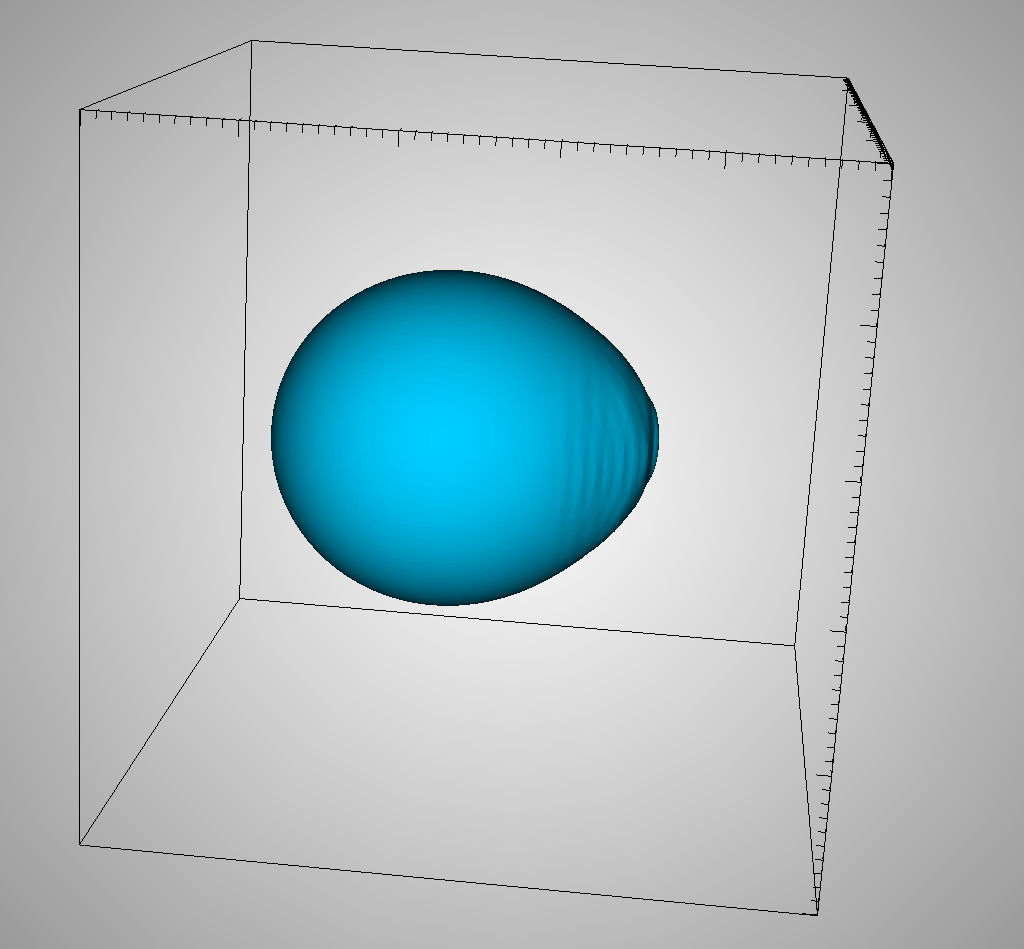}
\caption{Zero level-set contours at $t=20$}
\end{subfigure}
\caption{Coalescence 3D. Solutions at $t=10$ and $t=20$.}
\label{fig:Coalescence 3D. Solutions}
\end{figure}

\begin{figure}[h!]
\begin{subfigure}{0.33\textwidth}
\centering
\includegraphics[width=0.95\textwidth]{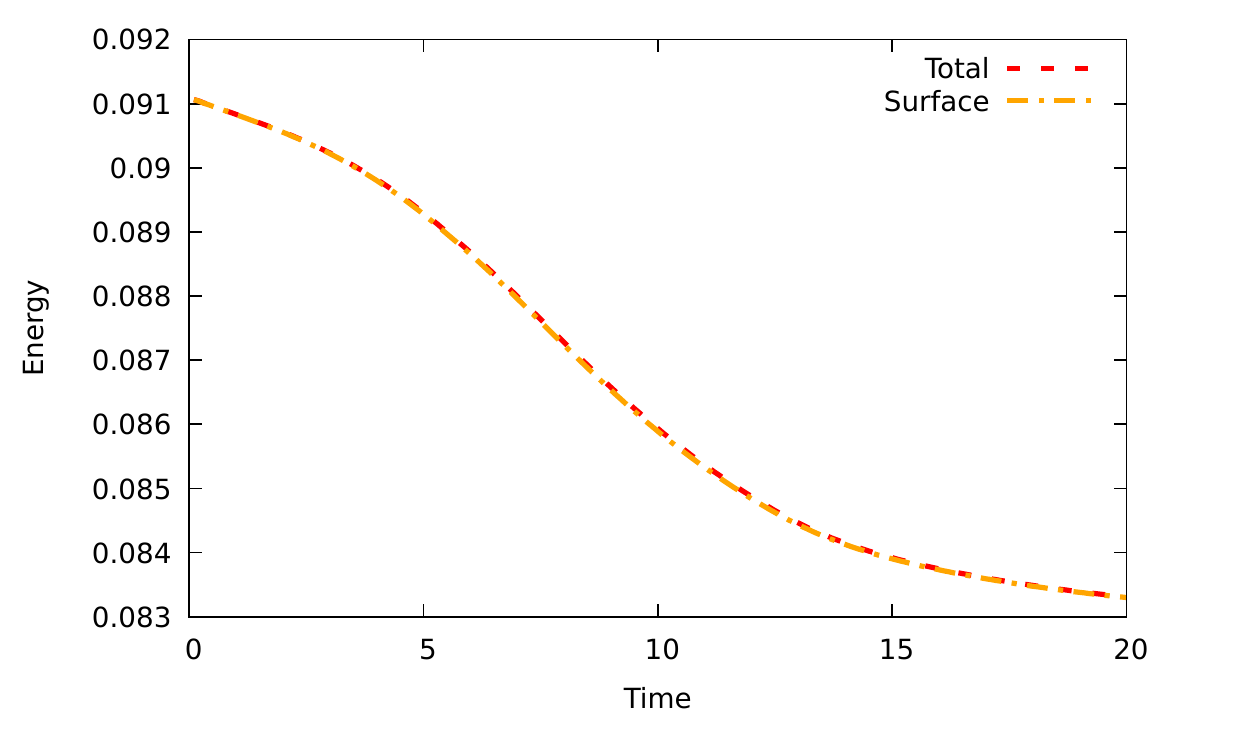}
\caption{Total and surface energy evolution.}
\end{subfigure}
\begin{subfigure}{0.33\textwidth}
\centering
\includegraphics[width=0.95\textwidth]{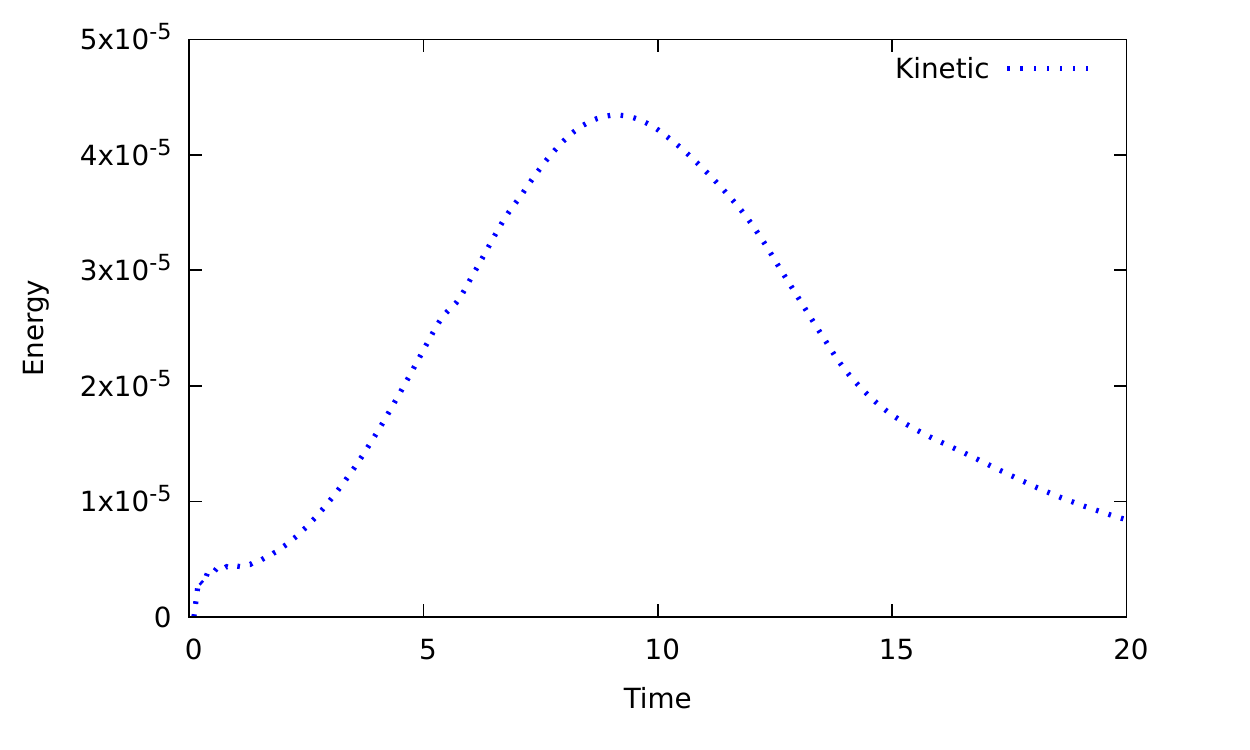}
\caption{Kinetic energy evolution.}
\end{subfigure}
\begin{subfigure}{0.33\textwidth}
\centering
\includegraphics[width=0.95\textwidth]{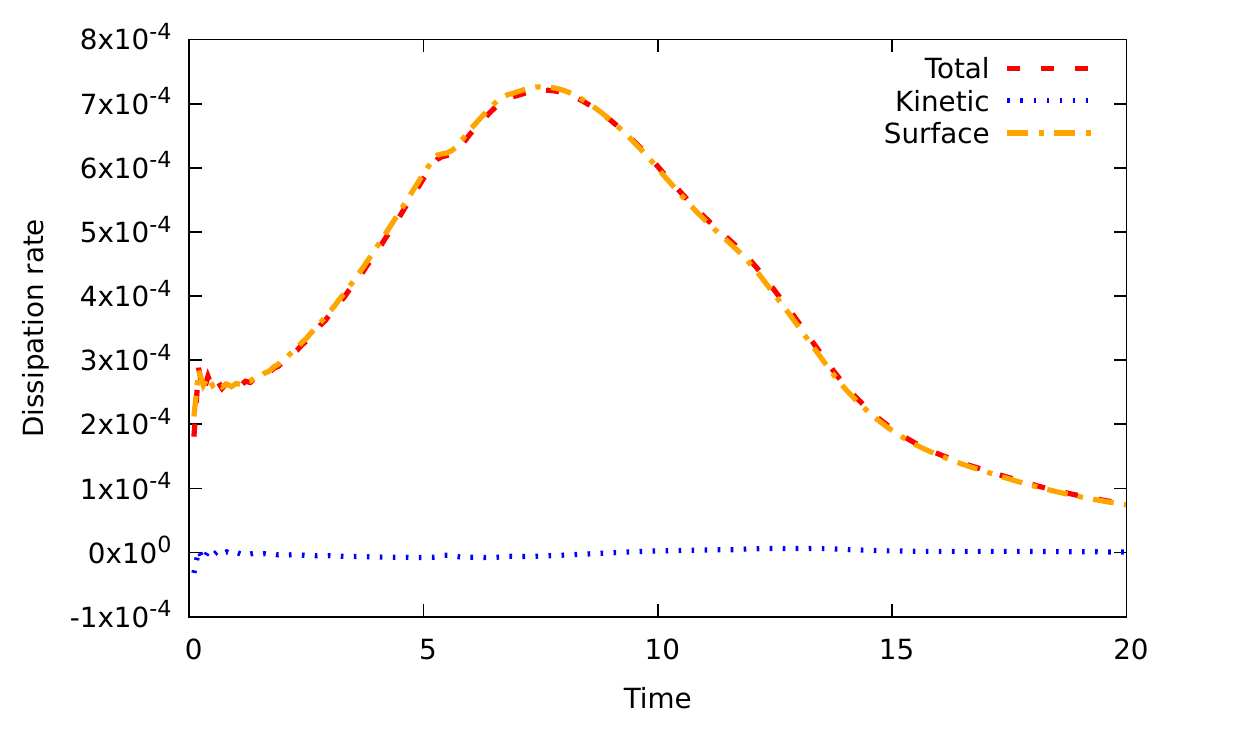}
\caption{Energy dissipation rate.}
\end{subfigure}
\caption{Coalescence 3D. Energy evolution and dissipation rate.}
\label{fig:3D Energy evo diss}
\end{figure}

\newpage

\section{Conclusion}

In this work we have proposed a new fully-discrete energy-stable level-set method for the incompressible Navier-Stokes equations with surface tension.
To the best knowledge of the authors, this is the first provable energy-dissipative level-set method.
Apart from being energetically-stable, the method satisfies the maximum principle for the density and is pointwise divergence-free.

We have provided a consistent derivation of our diffuse-interface level-set model starting from a sharp-interface model.
In addition we have presented a detailed analysis of both models in term of energy behavior.
This analysis implies that an energy-dissipative Galerkin-type discretization of the diffuse-interface level-set model poses severe restrictions on the functional spaces.
Independently, standard second-order temporal discretizations are also not associated with an energy-dissipative structure.
Lastly, the diffuse-interface model contains an unwanted regularization term.
We circumvent each of these problems by creating extra space via the concept of functional entropy variables.
This introduces an extra variable to the model which is coupled via the surface tension term.
This leads in a natural way to the fully-discrete energy-stable level-set method.
The eventual methodology use isogeometric analysis to ensure divergence-free solutions.
Furthermore, the method is equipped with an SUPG stabilization mechanism in the level-set equation that is energetically-balanced in the momentum equation. Additionally, we use a residual-based discontinuity capturing term to stabilize the momentum equation.
The temporal discretization is performed using a perturbed mid-point scheme.
We have presented numerical examples in two and three dimensions which confirm the energy-stability of the method.

We see several research directions for further work. A first suggestion is to equip the developed method with multiscale stabilization mechanisms that are energetically stable. Attainable solutions may be inspired by stabilization mechanisms that are energetically stable for single fluid flow \cite{EiAk17ii, evans2020variational}. Other possible research directions entail the development of energy dissipative re-distancing procedures. This would allow to simulate more violent flows, such as a dam-break problem, in an energy-dissipative manner. Another missing feature of the level-set method is local mass conservation. Perhaps local mass conservation may be obtained by using similar techniques as presented in this paper.
Lastly, we suggest to look into the construction of (energetically-stable) level-set methods that preclude parasitic currents.
\appendix

\section{Equivalence surface tension models}
We show equivalence of the surface tension models for the sharp-interface model and the diffuse-interface level-set model.
\subsection{Sharp interface model}\label{sec: Surface tension evaluation}
In order to avoid directly evaluating the curvature in the surface tension term, one may employ integration by parts as proposed by B\"{a}nsch \cite{bansch2001finite}.
First we introduce some notation.
The normal extensions of the scalar field $f$ and vector field $\mathbf{v}$ defined on $\Gamma$ are, see also \cite{buscaglia2011variational}:
\begin{subequations}\label{eq: normal extensions}
\begin{align}
  \hat{f}(\mathbf{x}) :=&~ f(\varPi_\Gamma(\mathbf{x})),\\
  \hat{\mathbf{v}}(\mathbf{x}) :=&~ \mathbf{v}(\varPi_\Gamma(\mathbf{x})),
\end{align}
\end{subequations}
where $\varPi_\Gamma(\mathbf{x})$ is defined as the normal projector of $\mathbf{x}$ onto the interface $\Gamma$.
The surface gradients of these fields are now given by
\begin{subequations}\label{eq: surface gradients}
\begin{align}
  \nabla_\Gamma f          :=&~ \nabla \hat{f},\\
  \nabla_\Gamma \mathbf{v} :=&~ \nabla \hat{\mathbf{v}},
\end{align}
\end{subequations}
while the tangential divergence of $\mathbf{v}$ is the trace of the surface gradient:
\begin{align}
  \nabla_\Gamma \cdot \mathbf{v} :=&~{\rm Tr}(\nabla_\Gamma \mathbf{v}) = \nabla \cdot \hat{\mathbf{v}}.
\end{align}
Note the slight abuse of notation; we use the same notation for the surface gradient as employed for the surface gradient in the diffuse level-set model.
Alternative expressions for the surface gradients are
\begin{subequations}\label{eq: surface gradients, alternative}
\begin{align}
  \nabla_\Gamma f          =&~ \mathbf{P}_T\cdot\nabla f,\\
  \nabla_\Gamma \mathbf{v} =&~ \nabla \mathbf{v} \cdot \mathbf{P}_T,
\end{align}
\end{subequations}
where $\mathbf{P}_T$ denotes the tangential projection tensor:
\begin{align}
 \mathbf{P}_T = \mathbf{I} - \hat{\boldsymbol{\nu}} \otimes \hat{\boldsymbol{\nu}},
\end{align}
where $\hat{\boldsymbol{\nu}}$ is continuous extension of the outward unit normal pointing from $\Omega_1$ into $\Omega_2$ and $\mathbf{I}$ is identity matrix.
Using the above identities we have
\begin{align}\label{eq: identities appendix}
 \nabla \cdot \hat{\bw} =~\nabla_\Gamma \cdot \bw  
                        =~{\rm Tr}(\nabla_\Gamma \bw)  
                        =~{\rm Tr}(\mathbf{P}_T\nabla \bw )  
                        =~\mathbf{P}_T : \nabla \bw .                                                                  
\end{align}
\begin{lem}\label{lem: Buscaglia} Buscaglia et al. \cite{buscaglia2011variational}: For any tangentially differentiable vector field $\bw$ we have:
   \begin{align}
 \displaystyle\int_{\Gamma(t)} \nabla_\Gamma\cdot \bw  ~{\rm d} \Gamma =~ \displaystyle \int_{\Gamma(t)} \kappa \hat{\boldsymbol{\nu}} \cdot \bw ~ {\rm d}\Gamma +~\displaystyle \int_{\partial \Gamma(t)} \boldsymbol{\nu}_\partial \cdot \bw~ {\rm d}(\partial \Gamma).
\end{align}
\end{lem}
Using \ref{eq: identities appendix} and \cref{lem: Buscaglia} we may write the surface tension term as
\begin{align}
 \frac{1}{\mathbb{W}{\rm e}}\displaystyle\int_{\Gamma(t)}  \kappa \boldsymbol{\nu}\cdot \bw ~{\rm d} \Gamma =&~ \frac{1}{\mathbb{W}{\rm e}}\displaystyle \int_{\Gamma(t)}  \nabla \cdot \hat{\bw}~ {\rm d}\Gamma -\frac{1}{\mathbb{W}{\rm e}}\displaystyle \int_{\partial \Gamma(t)} \boldsymbol{\nu}_\partial \cdot \bw~ {\rm d}(\partial \Gamma)\nn\\
                                                                  =&~\frac{1}{\mathbb{W}{\rm e}}\displaystyle \int_{\Gamma(t)}  \mathbf{P}_T:\nabla \bw ~{\rm d}\Gamma-\frac{1}{\mathbb{W}{\rm e}}\displaystyle \int_{\partial \Gamma(t)} \boldsymbol{\nu}_\partial \cdot \bw~ {\rm d}(\partial \Gamma).                                                                  
\end{align}

\subsection{Diffuse-interface level-set model}\label{appendix: surface tension derivation LS}
In the following we utilize index notation.
\begin{prop}\label{prop: id1}
It holds:
  \begin{align}
 \nabla_j \left( (P_T)_{ij}(\phi) \delta_\Gamma(\phi) \right) = -\delta_\Gamma(\phi)\dfrac{\nabla_i\phi}{\|\nabla \phi\|_{\epsilon,2}}\nabla_j\dfrac{\nabla_j\phi}{\|\nabla \phi\|_{\epsilon,2}} + \epsilon^2\dfrac{\nabla_i \delta(\phi)}{\|\nabla \phi\|_{\epsilon,2}}.
  \end{align}
\end{prop}
\begin{proof}
We compute
\begin{align}\label{eq: compute1}
  (P_T)_{ij}(\phi) \nabla_j \delta_\Gamma(\phi) =&~ \left(I_{ij}-\dfrac{\nabla_i\phi}{\|\nabla \phi\|_{\epsilon,2}}\dfrac{\nabla_j\phi}{\|\nabla \phi\|_{\epsilon,2}}\right)\left(\delta(\phi)\dfrac{\nabla_k\phi}{\|\nabla \phi\|_{\epsilon,2}}\nabla_j\nabla_k\phi + \|\nabla \phi\|_{\epsilon,2}\nabla_j\delta(\phi)\right)\nn\\
  =&~ \delta(\phi)\dfrac{\nabla_k\phi}{\|\nabla \phi\|_{\epsilon,2}}\nabla_i\nabla_k\phi + \|\nabla \phi\|_{\epsilon,2}\nabla_i\delta(\phi)\nn\\
  &~-\dfrac{\nabla_i\phi}{\|\nabla \phi\|_{\epsilon,2}}\dfrac{\nabla_j\phi}{\|\nabla \phi\|_{\epsilon,2}}\delta(\phi)\dfrac{\nabla_k\phi}{\|\nabla \phi\|_{\epsilon,2}}\nabla_j\nabla_k\phi-\dfrac{\nabla_i\phi}{\|\nabla \phi\|_{\epsilon,2}}\dfrac{\nabla_j\phi}{\|\nabla \phi\|_{\epsilon,2}} \|\nabla \phi\|_{\epsilon,2}\nabla_j\delta(\phi)\nn\\
  =&~ \delta_\Gamma(\phi)\dfrac{\nabla_k\phi}{\|\nabla \phi\|_{\epsilon,2}}\left(\dfrac{\nabla_i\nabla_k\phi}{\|\nabla\phi\|_{\epsilon,2}} -\dfrac{\nabla_i\phi}{\|\nabla \phi\|_{\epsilon,2}}\dfrac{\nabla_j\phi}{\|\nabla \phi\|_{\epsilon,2}}\dfrac{\nabla_j\nabla_k\phi}{\|\nabla \phi\|_{\epsilon,2}}\right)+ \epsilon^2\dfrac{\nabla_i \delta(\phi)}{\|\nabla \phi\|_{\epsilon,2}}\nn\\
  =&~ \delta_\Gamma(\phi)\dfrac{\nabla_k\phi}{\|\nabla \phi\|_{\epsilon,2}}(P_T)_{ij}\dfrac{\nabla_j\nabla_k\phi}{\|\nabla \phi\|_{\epsilon,2}} + \epsilon^2\dfrac{\nabla_i \delta(\phi)}{\|\nabla \phi\|_{\epsilon,2}}.
\end{align}
The second to last equality follows from expanding the gradient of the Dirac delta.
On the other hand we have:
\begin{align}\label{eq: compute2}
  \delta_\Gamma(\phi)\nabla_j(P_T)_{ij}(\phi)   =&~ -\delta_\Gamma(\phi)\dfrac{\nabla_i\phi}{\|\nabla \phi\|_{\epsilon,2}}\nabla_j\dfrac{\nabla_j\phi}{\|\nabla \phi\|_{\epsilon,2}}  - \delta_\Gamma(\phi)\dfrac{\nabla_j\phi}{\|\nabla \phi\|_{\epsilon,2}}\nabla_j\dfrac{\nabla_i\phi}{\|\nabla \phi\|_{\epsilon,2}} \nn\\
  =&~ -\delta_\Gamma(\phi)\dfrac{\nabla_i\phi}{\|\nabla \phi\|_{\epsilon,2}}\nabla_j\dfrac{\nabla_j\phi}{\|\nabla \phi\|_{\epsilon,2}}  \nn\\
  &~- \delta_\Gamma(\phi)\dfrac{\nabla_j\phi}{\|\nabla \phi\|_{\epsilon,2}}\left(\dfrac{\nabla_j\nabla_i\phi}{\|\nabla \phi\|_{\epsilon,2}} - \dfrac{\nabla_i\phi}{\|\nabla \phi\|_{\epsilon,2}}\dfrac{\nabla_k\phi}{\|\nabla \phi\|_{\epsilon,2}}\dfrac{\nabla_j\nabla_k\phi}{\|\nabla \phi\|_{\epsilon,2}}\right)\nn\\
  =&~ -\delta_\Gamma(\phi)\dfrac{\nabla_i\phi}{\|\nabla \phi\|_{\epsilon,2}}\nabla_j\dfrac{\nabla_j\phi}{\|\nabla \phi\|_{\epsilon,2}} - \delta_\Gamma(\phi)\dfrac{\nabla_j\phi}{\|\nabla \phi\|_{\epsilon,2}}(P_T)_{ik}\dfrac{\nabla_k\nabla_j\phi}{\|\nabla \phi\|_{\epsilon,2}}.
\end{align}
Addition of (\ref{eq: compute1}) and (\ref{eq: compute2}) yields:
\begin{align}\label{eq: compute3}
  \nabla_j \left( (P_T)_{ij}(\phi) \delta_\Gamma(\phi) \right) =&~ (P_T)_{ij}(\phi) \nabla_j \delta_\Gamma(\phi) + \delta_\Gamma(\phi)\nabla_j(P_T)_{ij}(\phi)  \nn\\
  =&~-\delta_\Gamma(\phi)\dfrac{\nabla_i\phi}{\|\nabla \phi\|_{\epsilon,2}}\nabla_j\dfrac{\nabla_j\phi}{\|\nabla \phi\|_{\epsilon,2}}+ \epsilon^2\dfrac{\nabla_i \delta(\phi)}{\|\nabla \phi\|_{\epsilon,2}}.
\end{align}
\end{proof}
\begin{lem} It holds:
  \begin{align}
 \frac{1}{\mathbb{W}{\rm e}}\displaystyle\int_\Omega  \delta_\Gamma(\phi) \nabla_j w_i (P_T)_{ij}(\phi)~{\rm d}\Omega =&~ \frac{1}{\mathbb{W}{\rm e}} \displaystyle\int_\Omega  \delta_\Gamma(\phi) \dfrac{\nabla_i \phi}{\|\nabla \phi \|_{\epsilon,2}}  \nabla_j \dfrac{\nabla_j \phi}{\|\nabla \phi \|_{\epsilon,2}} w_i ~{\rm d}\Omega \nn\\
 &~- \frac{1}{\mathbb{W}{\rm e}}\displaystyle\int_\Omega \epsilon^2\dfrac{\nabla_i \delta(\phi)}{\|\nabla \phi\|_{\epsilon,2}}w_i~{\rm d}\Omega.
  \end{align}
\end{lem}
\begin{proof}
Performing integration by parts we get:
\begin{align}\label{eq: compute0}
 \frac{1}{\mathbb{W}{\rm e}}\displaystyle\int_\Omega  \delta_\Gamma(\phi) \nabla_j w_i (P_T)_{ij}(\phi)~{\rm d}\Omega =&~ - \frac{1}{\mathbb{W}{\rm e}}\displaystyle\int_\Omega  \nabla_j \left(\delta_\Gamma(\phi)  (P_T)_{ij}(\phi)\right) w_i~{\rm d}\Omega\nn\\
 &~+ \frac{1}{\mathbb{W}{\rm e}}\displaystyle\int_{\partial \Omega}  \delta_\Gamma(\phi) n_j w_i (P_T)_{ij}(\phi)~{\rm d S}.
\end{align}
Under the standing assumption we suppress the line force term. Using \cref{prop: id1} finalizes the proof.
\end{proof}

\section{Energy evolution midpoint level-set discretization}\label{appendix: Standard mid-point Galerkin discretizations}
We provide the energy evolution of a standard time-discrete level-set method using the midpoint rule. We consider the conservative discretization, which reads for time-step $n$:\\

\noindent \textit{Given $\bu_n, p_n$ and $\phi_n$, find $\bu_{n+1}, p_{n+1}$ and $\phi_{n+1}$ such that:}

\begin{subequations}\label{weak form gen alpha conservative}
\label{eq:weak}
\begin{align}
 \dfrac{[\![\rho\bu]\!]_{n+1/2}}{\Delta t_n} +\nabla \cdot (\rho_{n+1/2} \bu_{n+1/2}\otimes \bu_{n+1/2}) + \nabla p_{n+1} -\nabla \cdot \boldsymbol{\tau}(\bu_{n+1/2},\phi_{n+1/2}) &\nn\\
 + \frac{1}{\mathbb{W}{\rm e}} \kappa(\phi_{n+1/2}) \boldsymbol{\nu}(\phi_{n+1/2}) \delta_\Gamma(\phi_{n+1/2})+\frac{1}{\mathbb{F}{\rm r}^2} \rho_{n+1/2} \boldsymbol{\jmath}&~=0, \label{weak form 1 mom time}\\
 \nabla \cdot \bu_{n+1/2} &~= 0, \label{weak form 1 cont time} \\
 \dfrac{[\![\phi]\!]_{n+1/2}}{\Delta t_n}+ \bu_{n+1/2} \cdot \nabla  \phi_{n+1/2} &~= 0, \label{weak form 1 mass time}
\end{align}
\end{subequations}
\textit{where $\rho \equiv \rho(\phi)$ on the indicated time-level.}

\begin{thm}
The time-discrete formulation (\ref{weak form gen alpha conservative}) satisfies the energy evolution property:
\begin{subequations}\label{eq: time-discrete evolution}
 \begin{align}
 \dfrac{[\![\mathscr{E}\left(\bu, \phi\right)]\!]_{n+1/2}}{\Delta t_n}  =&~ - \int_\Omega \nabla \bu_{n+1/2}: \boldsymbol{\tau}(\bu_{n+1/2},\phi_{n+1/2})~ {\rm d}\Omega + \afunc{error}\\
 \afunc{error} =&~ \Delta t_n^2\displaystyle\int_\Omega \tfrac{1}{8} \left\|\dfrac{[\![\bu]\!]_{n+1/2}}{\Delta t_n}\right\|^2 \dfrac{[\![\![\rho]\!]\!]_{n+1/2}}{\Delta t_n}~{\rm d} \Omega \nn\\
  &~-\frac{1}{\mathbb{W}{\rm e}\Delta t_n}\displaystyle\int_\Omega [\![\delta(\phi)]\!]_{n+1/2}\left(\|\nabla \phi_{n+1/2}\|_{\epsilon,2} - (\|\nabla \phi\|_{\epsilon,2})_{n+1/2} \right)~{\rm d}\Omega\nn\\
 &~-\frac{1}{\mathbb{W}{\rm e}\Delta t_n}\displaystyle\int_\Omega [\![\|\nabla\phi\|_{\epsilon,2}]\!]_{n+1/2} \left( \delta(\phi_{n+1/2})\dfrac{\|\nabla \phi\|_{n+1/2}}{\|\nabla \phi_{n+1/2}\|_{\epsilon,2}}-\delta(\phi)_{n+1/2}\right) ~{\rm d}\Omega \nn\\
 &~+ \frac{1}{\mathbb{W}{\rm e}\Delta t_n}\displaystyle\int_\Omega [\![\phi]\!]_{n+1/2}^3\left( \delta^{(3)}(\phi_{n+1/2})/24 + [\![\phi]\!]_{n+1/2}^2\delta^{(5)}(\phi_{n+\xi})/1920\right)\|\nabla \phi_{n+1/2}\|_{\epsilon,2} ~{\rm d}\Omega \nn\\
 &~+ \displaystyle\int_\Omega\frac{1}{\mathbb{W}{\rm e}\Delta t_n} \delta'(\phi_{n+1/2})[\![\phi]\!]_{n+1/2}\dfrac{\epsilon^2}{\|\nabla \phi_{n+1/2}\|_{\epsilon,2}} ~{\rm d}\Omega. \label{eq: error def}
\end{align}
\end{subequations}
\end{thm}

\begin{rmk} 
The semi-discrete convective method has the same energy evolution (\ref{eq: time-discrete evolution}). For completeness we provide the convective method:\\

\noindent \textit{Given $\bu_n, p_n$ and $\phi_n$, find $\bu_{n+1}, p_{n+1}$ and $\phi_{n+1}$ such that:}

\begin{subequations}\label{weak form gen alpha convective}
\label{eq:weak}
\begin{align}
 \rho_{n+1/2}\left(\dfrac{[\![\bu]\!]_{n+1/2}}{\Delta t_n} + \bu_{n+1/2} \cdot \nabla \bu_{n+1/2}\right) + \nabla p_{n+1} -\nabla \cdot \boldsymbol{\tau}(\bu_{n+1/2},\phi_{n+1/2})&\nn\\
 + \frac{1}{\mathbb{W}{\rm e}} \kappa(\phi_{n+1/2}) \boldsymbol{\nu}(\phi_{n+1/2})\delta_\Gamma(\phi_{n+1/2})+\frac{1}{\mathbb{F}{\rm r}^2} \rho_{n+1/2} \boldsymbol{\jmath}&~=0, \label{weak form 1 mom time}\\
 \nabla \cdot \bu_{n+1/2} &~= 0, \label{weak form 1 cont time} \\
 \dfrac{[\![\phi]\!]_{n+1/2}}{\Delta t_n}+ \bu_{n+1/2} \cdot \nabla  \phi_{n+1/2} &~= 0, \label{weak form 1 mass time}
\end{align}
\end{subequations}
\textit{where $\rho \equiv \rho(\phi)$ on the indicated time-level.}
\end{rmk}

\begin{proof}
We give the proof for the conservative formulation, that of the convective formulation follows analogously. Multiplication of the continuity equation by $q = p_{n+1} - \rho_{n+1/2}(\tfrac{1}{2}\bu_{n+1/2}\cdot\bu_{n+1/2}-\frac{1}{\mathbb{F}{\rm r}^2} y)$ and the level-set equation by $-([\![\![\rho]\!]\!]\tfrac{1}{2}\bu_{n+1/2}\cdot\bu_{n+1/2}-\dfrac{1}{\mathbb{F}{\rm r}^2}[\![\![\rho]\!]\!] y + \frac{1}{\mathbb{W}{\rm e}}\kappa(\phi_{n+1/2})) \delta(\phi_{n+1/2})$ and subsequently integrating yields:
\begin{subequations}\label{eq: cont proof weights 1 time discrete}
\begin{align}
  \displaystyle\int_\Omega (p_{n+1}- \rho_{n+1/2} (\tfrac{1}{2}\bu_{n+1/2}\cdot\bu_{n+1/2}-\frac{1}{\mathbb{F}{\rm r}^2} y)) \nabla \cdot \bu_{n+1/2} ~{\rm d} \Omega &~= 0,\\
   -\displaystyle\int_\Omega \left(\tfrac{1}{2}\bu_{n+1/2}\cdot\bu_{n+1/2}-\frac{1}{\mathbb{F}{\rm r}^2} y\right) \left(\dfrac{[\![\![\rho]\!]\!]_{n+1/2}}{\Delta t_n}+ \bu_{n+1/2} \cdot \nabla  \rho_{n+1/2}\right)~{\rm d} \Omega &\nn\\
    - \displaystyle\int_\Omega\left( \frac{1}{\mathbb{W}{\rm e}} \kappa(\phi_{n+1/2}) \delta(\phi_{n+1/2})\right)\left( \dfrac{[\![\phi]\!]_{n+1/2}}{\Delta t_n}+ \bu_{n+1/2} \cdot \nabla  \phi_{n+1/2} \right) ~{\rm d}\Omega&~= 0.
\end{align}
\end{subequations}
We add the equations (\ref{eq: cont proof weights 1 time discrete}) and find:
\begin{align}\label{eq: cont combining time}
   -\displaystyle\int_\Omega \left(\tfrac{1}{2}\bu_{n+1/2}\cdot\bu_{n+1/2} - \frac{1}{\mathbb{F}{\rm r}^2}y\right) \dfrac{[\![\![\rho]\!]\!]_{n+1/2}}{\Delta t_n}~{\rm d} \Omega ~~~&\nn\\
   - \displaystyle\int_\Omega\frac{1}{\mathbb{W}{\rm e}} \kappa(\phi_{n+1/2}) \delta(\phi_{n+1/2})\dfrac{[\![\phi]\!]_{n+1/2}}{\Delta t_n}~{\rm d}\Omega =&~- \displaystyle\int_\Omega (p_{n+1}-\rho_{n+1/2}\tfrac{1}{2}\bu_{n+1/2}\cdot\bu_{n+1/2}) \nabla \cdot \bu_{n+1/2} ~{\rm d} \Omega\nn\\
   &~+\displaystyle\int_\Omega \tfrac{1}{2}\bu_{n+1/2}\cdot\bu_{n+1/2} (\bu_{n+1/2} \cdot \nabla  \rho_{n+1/2})~{\rm d} \Omega \nn\\
   &~-\displaystyle\int_\Omega \frac{1}{\mathbb{F}{\rm r}^2}y (\bu_{n+1/2} \cdot \nabla  \rho_{n+1/2} + \rho_{n+1/2} \nabla \cdot \bu_{n+1/2})~{\rm d} \Omega\nn\\
   &~+ \displaystyle\int_\Omega \frac{1}{\mathbb{W}{\rm e}} \kappa(\phi_{n+1/2}) \delta(\phi_{n+1/2})\bu_{n+1/2} \cdot \nabla  \phi_{n+1/2}  ~{\rm d}\Omega.
\end{align}
We take the second term on the left-hand side of (\ref{eq: cont combining time}) in isolation and perform integration by parts to get:
\begin{align}\label{eq: identify RHS0}
 - \displaystyle\int_\Omega\frac{1}{\mathbb{W}{\rm e}} \kappa(\phi_{n+1/2}) \delta(\phi_{n+1/2})\dfrac{[\![\phi]\!]_{n+1/2}}{\Delta t_n}~{\rm d}\Omega =&~  \displaystyle\int_\Omega\frac{1}{\mathbb{W}{\rm e}} \nabla \left(\delta(\phi_{n+1/2})\dfrac{[\![\phi]\!]_{n+1/2}}{\Delta t_n} \right)\dfrac{\nabla \phi_{n+1/2}}{\|\nabla \phi_{n+1/2}\|_{\epsilon,2}} ~{\rm d}\Omega\nn\\
 =&~  \displaystyle\int_\Omega\frac{1}{\mathbb{W}{\rm e}\Delta t_n} \delta(\phi_{n+1/2})\nabla [\![\phi]\!]_{n+1/2} \cdot\dfrac{\nabla \phi_{n+1/2}}{\|\nabla \phi_{n+1/2}\|_{\epsilon,2}} ~{\rm d}\Omega \nn\\
 &~+ \displaystyle\int_\Omega\frac{1}{\mathbb{W}{\rm e}\Delta t_n} \delta'(\phi_{n+1/2})[\![\phi]\!]_{n+1/2}\|\nabla \phi_{n+1/2}\|_{\epsilon,2} ~{\rm d}\Omega \nn\\
  &~- \displaystyle\int_\Omega\frac{1}{\mathbb{W}{\rm e}\Delta t_n} \delta'(\phi_{n+1/2})[\![\phi]\!]_{n+1/2}\dfrac{\epsilon^2}{\|\nabla \phi_{n+1/2}\|_{\epsilon,2}} ~{\rm d}\Omega.
\end{align}
For the first term on the right-hand side we use
\begin{align}\label{eq: identify RHS1}
  \nabla [\![\phi]\!]_{n+1/2} \cdot\dfrac{\nabla \phi_{n+1/2}}{\|\nabla \phi_{n+1/2}\|_{\epsilon,2}} = [\![\|\nabla\phi\|_{\epsilon,2}]\!]_{n+1/2} \dfrac{\left(\|\nabla \phi\|_{\epsilon,2}\right)_{n+1/2}}{\|\nabla \phi_{n+1/2}\|_{\epsilon,2}},
\end{align}
while for the second term employ a truncated Taylor series in the form:
\begin{align}\label{eq: identify RHS2}
 [\![\delta(\phi)]\!]_{n+1/2} =&~[\![\phi]\!]_{n+1/2}\delta^{(1)}(\phi_{n+1/2}) + [\![\phi]\!]_{n+1/2}^3\delta^{(3)}(\phi_{n+1/2})/24 + [\![\phi]\!]_{n+1/2}^5\delta^{(5)}(\phi_{n+\xi})/1920,
\end{align}
for some $\xi \in (0,1)$.
Substitution of (\ref{eq: identify RHS1})-(\ref{eq: identify RHS2}) into (\ref{eq: identify RHS0}) and reorganizing gives:
\begin{align}\label{eq: identify RHS3}
 -& \displaystyle\int_\Omega\frac{1}{\mathbb{W}{\rm e}} \kappa(\phi_{n+1/2}) \delta(\phi_{n+1/2})\dfrac{[\![\phi]\!]_{n+1/2}}{\Delta t_n}~{\rm d}\Omega \nn\\
 =&~\frac{1}{\mathbb{W}{\rm e}\Delta t_n}\displaystyle\int_\Omega \delta(\phi)_{n+1/2}[\![\|\nabla\phi\|_{\epsilon,2}]\!]_{n+1/2} + [\![\delta(\phi)]\!]_{n+1/2} (\|\nabla \phi\|_{\epsilon,2})_{n+1/2} ~{\rm d}\Omega \nn\\
 &~+\frac{1}{\mathbb{W}{\rm e}\Delta t_n}\displaystyle\int_\Omega [\![\delta(\phi)]\!]_{n+1/2}\left(\|\nabla \phi_{n+1/2}\|_{\epsilon,2} - (\|\nabla \phi\|_{\epsilon,2})_{n+1/2} \right)~{\rm d}\Omega\nn\\
 &~+\frac{1}{\mathbb{W}{\rm e}\Delta t_n}\displaystyle\int_\Omega [\![\|\nabla\phi\|_{\epsilon,2}]\!]_{n+1/2} \left( \delta(\phi_{n+1/2})\dfrac{\|\nabla \phi\|_{n+1/2}}{\|\nabla \phi_{n+1/2}\|_{\epsilon,2}}-\delta(\phi)_{n+1/2}\right) ~{\rm d}\Omega \nn\\
 &~- \frac{1}{\mathbb{W}{\rm e}\Delta t_n}\displaystyle\int_\Omega [\![\phi]\!]_{n+1/2}^3\left( \delta^{(3)}(\phi_{n+1/2})/24 + [\![\phi]\!]_{n+1/2}^2\delta^{(5)}(\phi_{n+\xi})/1920\right)\|\nabla \phi_{n+1/2}\|_{\epsilon,2} ~{\rm d}\Omega\nn\\
   &~- \displaystyle\int_\Omega\frac{1}{\mathbb{W}{\rm e}\Delta t_n} \delta'(\phi_{n+1/2})[\![\phi]\!]_{n+1/2}\dfrac{\epsilon^2}{\|\nabla \phi_{n+1/2}\|_{\epsilon,2}} ~{\rm d}\Omega,
\end{align}
where the first term on the right-hand side represents the temporal change of surface energy (see \cref{prop: product-rule midpoint}):
\begin{align}\label{eq: identify RHS4}
 \frac{1}{\mathbb{W}{\rm e}}\displaystyle\int_\Omega \dfrac{[\![\delta_\Gamma(\phi)]\!]_{n+1/2}}{\Delta t_n}  ~{\rm d}\Omega =&~ 
 \frac{1}{\mathbb{W}{\rm e}\Delta t_n}\displaystyle\int_\Omega \delta(\phi)_{n+1/2}[\![\|\nabla\phi\|_{\epsilon,2}]\!]_{n+1/2}  + [\![\delta(\phi)]\!]_{n+1/2} (\|\nabla \phi\|_{\epsilon,2})_{n+1/2} ~{\rm d}\Omega.
\end{align}
Next we multiply the momentum equation by $\bu_{n+1/2}$  and subsequently integrate to get:
\begin{align}\label{eq: cont combining 2 time}
  &\displaystyle\int_\Omega\bu_{n+1/2}^T  \dfrac{[\![\rho\bu]\!]_{n+1/2}}{\Delta t_n}~{\rm d}\Omega + \displaystyle\int_\Omega \bu_{n+1/2} \nabla \cdot (\rho_{n+1/2} \bu_{n+1/2}\otimes \bu_{n+1/2})~{\rm d}\Omega + \displaystyle\int_\Omega\bu_{n+1/2}\nabla p_{n+1}~{\rm d}\Omega \nn\\
  &+ \displaystyle\int_\Omega\bu_{n+1/2} \nabla \cdot \boldsymbol{\tau}(\bu_{n+1/2},\phi_{n+1/2})~{\rm d}\Omega + \displaystyle\int_\Omega\bu_{n+1/2} \rho_{n+1/2} \frac{1}{\mathbb{F}{\rm r}^2}\boldsymbol{\jmath} ~{\rm d}\Omega \nn\\
  &+ \displaystyle\int_\Omega\frac{1}{\mathbb{W}{\rm e}} \kappa(\phi_{n+1/2}) \bu_{n+1/2}\cdot \boldsymbol{\nu}(\phi_{n+1/2})\delta_\Gamma(\phi_{n+1/2})  ~{\rm d}\Omega= 0
\end{align}
The time-derivative term may be written as
\begin{align}\label{eq:temp_term}
 \displaystyle\int_\Omega \bu_{n+1/2} \cdot \dfrac{[\![\rho\bu]\!]_{n+1/2}}{\Delta t_n} ~{\rm d}\Omega 
= &  ~\Delta t_n^{-1} \displaystyle\int_\Omega \tfrac{1}{2}  \rho_{n+1} \|\bu_{n+1/2}\|^2 -\tfrac{1}{2}  \rho_{n} \|\bu_{n}\|^2~{\rm d}\Omega \nn\\
  & + \Delta t_n^{-1} \displaystyle\int_\Omega \tfrac{1}{2}   (\rho_{n+1}-\rho_n) \bu_{n} \cdot\bu_{n+1}  ~{\rm d}\Omega.
\end{align}
Expanding the divergence operator in the convective term gives:
\begin{align} \label{eq:conv_term}
  \displaystyle\int_\Omega  \bu_{n+1/2} \nabla \cdot( \rho_{n+1/2} \bu_{n+1/2}\otimes \bu_{n+1/2})~{\rm d}\Omega 
								       =& \displaystyle\int_\Omega\tfrac{1}{2} \|\bu_{n+1/2}\|^2 \bu_{n+1/2} \cdot \nabla \rho_{n+1/2}  ~{\rm d}\Omega\nn\\
								       &+  \displaystyle\int_\Omega  \tfrac{1}{2} \|\bu_{n+1/2}\|^2 \rho_{n+1/2} \nabla \cdot  \bu_{n+1/2} ~{\rm d}\Omega.
\end{align}
Substitution of (\ref{eq:temp_term})-(\ref{eq:conv_term}) into (\ref{eq: cont combining 2 time}) and performing integration by parts gives:
\begin{align}\label{eq: cont combining 3 time}
  \Delta t_n^{-1} \displaystyle\int_\Omega \tfrac{1}{2}  \rho_{n+1} \|\bu_{n+1/2}\|^2 -\tfrac{1}{2}  \rho_{n} \|\bu_{n}\|^2~{\rm d}\Omega =&~ -\displaystyle\int_\Omega\tfrac{1}{2} \|\bu_{n+1/2}\|^2 \bu_{n+1/2} \cdot \nabla \rho_{n+1/2}  ~{\rm d}\Omega \nn\\
  &~-  \displaystyle\int_\Omega  \tfrac{1}{2} \|\bu_{n+1/2}\|^2 \rho_{n+1/2} \nabla \cdot  \bu_{n+1/2} ~{\rm d}\Omega \nn\\
								       &~- \displaystyle\int_\Omega\bu_{n+1/2}\nabla p_{n+1}~{\rm d}\Omega - \displaystyle\int_\Omega\bu_{n+1/2} \rho_{n+1/2} \frac{1}{\mathbb{F}{\rm r}^2}\boldsymbol{\jmath} ~{\rm d}\Omega \nn\\
  &~+ \displaystyle\int_\Omega\nabla\bu_{n+1/2} : \boldsymbol{\tau}(\bu_{n+1/2},\phi_{n+1/2})~{\rm d}\Omega \nn\\
  &~- \displaystyle\int_\Omega\frac{1}{\mathbb{W}{\rm e}} \kappa(\phi_{n+1/2}) \bu_{n+1/2}\cdot \boldsymbol{\nu}(\phi_{n+1/2})\delta_\Gamma(\phi_{n+1/2})  ~{\rm d}\Omega \nn\\
  &~ - \Delta t_n^{-1} \displaystyle\int_\Omega \tfrac{1}{2}   (\rho_{n+1}-\rho_n) \bu_{n} \cdot\bu_{n+1}  ~{\rm d}\Omega.
\end{align}
Addition of (\ref{eq: cont combining time}) and (\ref{eq: cont combining 3 time}) while using (\ref{eq: identify RHS3})-(\ref{eq: identify RHS4}) gives:
\begin{align}\label{eq: cont combining 4 time}
  \Delta t_n^{-1} \displaystyle\int_\Omega \tfrac{1}{2}  \rho_{n+1} \|\bu_{n+1}\|^2 +\frac{1}{\mathbb{F}{\rm r}^2}y \rho_{n+1} + \frac{1}{\mathbb{W}{\rm e}} \delta_\Gamma(\phi_{n+1})~{\rm d}\Omega & \nn\\
  - \Delta t_n^{-1} \displaystyle\int_\Omega \tfrac{1}{2}  \rho_{n} \|\bu_{n}\|^2 +\frac{1}{\mathbb{F}{\rm r}^2}y \rho_{n}+ \frac{1}{\mathbb{W}{\rm e}} \delta_\Gamma(\phi_{n})~{\rm d}\Omega&=~  \displaystyle\int_\Omega\nabla \bu_{n+1/2} : \boldsymbol{\tau}(\bu_{n+1/2}, \phi_{n+1/2})~{\rm d}\Omega\nn\\
  &~~~~+ \afunc{error},
\end{align}
with $\afunc{error}$ defined in (\ref{eq: error def}). Recognizing the left-hand side as the change in energy completes the proof.
\end{proof}



\bibliographystyle{unsrt}
\bibliography{references}

\end{document}